\documentclass[leqno]{amsart}
\usepackage{amsmath} 
\usepackage{amssymb,mathtools,stmaryrd}
\usepackage{mathrsfs,euscript}
\usepackage[table,dvipsnames]{xcolor}
\usepackage{hyperref,tikz-cd,enumitem}
\usepackage[utf8]{inputenc}
\hypersetup{
 colorlinks=true,
 linkcolor=DarkOrchid,
 filecolor=blue,
 citecolor=olive,
 urlcolor=orange,
 pdftitle={definite unitary},
}

\calclayout

\newtheorem{thm}{Theorem}[section]
\newtheorem{lem}[thm]{Lemma}
\newtheorem{prop}[thm]{Proposition}
\newtheorem{cor}[thm]{Corollary}

\theoremstyle{definition}
\newtheorem{defn}[thm]{Definition}

\theoremstyle{remark}
\newtheorem{rem}[thm]{Remark}

\makeatletter
\newsavebox{\@brx}
\newcommand{\llangle}[1][]{\savebox{\@brx}{\(\m@th{#1\langle}\)}%
  \mathopen{\copy\@brx\kern-0.5\wd\@brx\usebox{\@brx}}}
\newcommand{\rrangle}[1][]{\savebox{\@brx}{\(\m@th{#1\rangle}\)}%
  \mathclose{\copy\@brx\kern-0.5\wd\@brx\usebox{\@brx}}}
\makeatother

\newcommand{\smat}[1]{\left(\begin{smallmatrix} #1 \end{smallmatrix}\right)}
\newcommand{\id}{\mathbf{1}}
\newcommand{\oo}{\mathcal{O}} 

\newcommand{\Q}{{\mathbf{Q}}}
\newcommand{\Z}{{\mathbf{Z}}}
\newcommand{\Qp}{\mathbf{Q}_p}
\newcommand{\Zp}{\mathbf{Z}_p}

\newcommand{\R}{\mathbf R}
\newcommand{\C}{\mathbf C}
\newcommand{\A}{\mathbf A}
\newcommand{\dd}{\mathfrak{d}} 
\newcommand{\arch}{\mathbf{a}}
\newcommand{\finite}{\mathbf{h}}
\DeclareMathOperator{\Nr}{N}
\DeclareMathOperator{\Tr}{Tr}

\DeclareMathOperator{\End}{End}

\DeclareMathOperator{\Hom}{Hom}

\DeclareMathOperator{\Res}{Res}

\DeclareMathOperator{\Ad}{Ad}

\DeclareMathOperator{\Lie}{Lie}
\DeclareMathOperator{\GL}{GL}

\DeclareMathOperator{\UU}{U}
\DeclareMathOperator{\Sp}{Sp}

\DeclareMathOperator{\GSp}{GSp}
\DeclareMathOperator{\GUU}{GU}

\DeclareMathOperator{\mtr}{tr}
\DeclareMathOperator{\diag}{diag}

\DeclareMathOperator{\vol}{vol} 
\DeclareMathOperator{\val}{val} 

\DeclareMathOperator{\Inj}{Inj}
\DeclareMathOperator{\Isom}{Isom}

\DeclareMathOperator{\Gal}{Gal}

\newcommand{\fa}{\mathfrak{a}}
\newcommand{\fc}{\mathfrak{c}}

\newcommand{\fl}{\mathfrak{l}}
\newcommand{\fm}{\mathfrak{m}}
\newcommand{\fn}{\mathfrak{n}}
\newcommand{\fp}{\mathfrak{p}}

\newcommand{\fs}{\mathfrak{s}}

\newcommand{\bs}{\mathcal{S}} 

\newcommand{\F}{{\mathcal{F}}} 
\newcommand{\K}{{\mathcal{K}}} 
\newcommand{\qch}{\epsilon} 
\newcommand{\OF}{{\mathcal{O}_{\F}}}
\newcommand{\OK}{\mathcal{O}_{\K}}
\newcommand{\kk}{F} 
\newcommand{\E}{E} 

\newcommand{\bw}{{\overline{w}}}

\newcommand{\fG}{\mathfrak{G}}
\newcommand{\fX}{\mathfrak{X}}

\newcommand{\V}{\mathbf V} 
\newcommand{\W}{\mathbf W} 
\newcommand{\G}{\mathbf G} 
\newcommand{\X}{\mathbf H} 
\newcommand{\KK}{\mathbf K} 
\newcommand{\xx}{\mathbf x}
\newcommand{\yy}{\mathbf y}

\newcommand{\cA}{\mathcal A} 
\newcommand{\cB}{\mathcal B} 

\newcommand{\mm}{\mathbf{m}}

\newcommand{\balpha}{\boldsymbol{\alpha}} 
\newcommand{\bbeta}{\boldsymbol{\beta}} 
\newcommand{\bnu}{\boldsymbol{\nu}} 
\newcommand{\bmu}{\boldsymbol{\mu}} 

\newcommand{\wt}[1]{\underline{ #1 }}
\newcommand{\bwt}[1]{\underline{\boldsymbol { #1 }}}
\newcommand{\M}{\mathbf{M}}
\newcommand{\pM}{\mathbf{M}}
\newcommand{\pMk}{M_{\wt{k}}}
\newcommand{\AM}{\mathcal{A}}
\newcommand{\bun}{\mathcal{E}}
\newcommand{\ii}{\mathbf i}
\newcommand{\ff}{\mathbf f}
\newcommand{\her}{\EuScript{H} }

\newcommand{\ordp}{e}

\newcommand{\vk}{v_{-\wt{k}}} 
\newcommand{\lk}{\mathnormal{l}_{\wt{k}}} 
\newcommand{\blk}{\mathnormal{l}_{\bwt{k}}} 
\newcommand{\Lk}{L_{\wt{k}}} 
\newcommand{\pf}{f} 
\newcommand{\ppf}{\hat{f}} 

\newcommand{\euF}{\EuScript{F}} 
\newcommand{\B}{\mathbf{B}} 

\newcommand{\ww}{\boldsymbol{\omega}}
\DeclareMathOperator{\an}{an}
\DeclareMathOperator{\ord}{ord}
\DeclareMathOperator{\cts}{cts}

\DeclareMathOperator{\Art}{Art}
\DeclareMathOperator{\Fr}{Fr}

\newcommand{\hZ}{{\hat{\mathbf{Z}}}}

\newcommand{\oeu}{\EuScript{O}}

\newcommand{\geu}{\EuScript{G}}
\newcommand{\keu}{\EuScript{K}}
\newcommand{\1}{\mathbf{1}}

\newcommand{\ee}{\mathbf e}
\newcommand{\bX}{\mathbb{X}}
\newcommand{\bY}{\mathbb{Y}}
\newcommand{\bV}{\mathbb{V}}
\newcommand{\bW}{\mathbb{W}}
\newcommand{\btheta}{\boldsymbol{\theta}}
\newcommand{\bdelta}{\boldsymbol{\delta}}

\newcommand{\I}{\mathcal{I}}

\begin{document}

\title{Hida family of theta lift from $\UU(1)$ to definite $\UU(2)$}

\begin{abstract}
    Let $\K/\F$ be a CM extension satisfying the ordinary assumption for an odd prime $p$.
    In this article, we 
    construct Hida families that interpolates theta lifts 
    of algebraic Hecke characters of $\K$ 
    to a definite unitary group $\UU(2)$
    defined from skew-Hermitian spaces over $\K$,
    and show that the Hida family is primitive
    when $p$ is sufficiently large and the central $L$-values
    of an auxiliary Hecke character 
    satisfy the non-vanishing modulo $p$ condition in \cite{Hsieh12}. 
\end{abstract}

\author[Y-S.~Lee]{Yu-Sheng Lee}
\address{Department of Mathematics, University  of Michigan, Ann Arbor, MI 48109, USA}
\email{yushglee@umich.edu}

\date{\today}

\maketitle
\setcounter{tocdepth}{1}
\tableofcontents

\section*{Introduction}

Theta correspondence has proved to be a powerful method
for testing Langlands functoriality between dual reductive pairs 
and periods relations of automorphic forms.
In an abstract set-up, 
$(G,H)$ is a pair of reductive subgroups over a global field $\F$
that are centralizers to each other in a symplectic group $\Sp$,
and $\bs(\A)$ is a certain space of Bruhat-Schwartz functions
which is acted upon by the metaplectic cover of $\Sp$ via the Weil represenation 
once a nontrivial character $\psi\colon \A_\F/\F\to \C$ is fixed.
Then given $\phi\in \bs(\A)$ 
and an automorphic function $\ff$ of $H$,
the associated theta lift $\theta_\phi(\ff)$, when it converges,
is an automorphic function of $G$
that is expected to carry information about Langlands parameters of $\ff$.

Of special arithmetic interest is the case when the theta correspondence can be shown
to preserve integral structures of the ambient spaces of automorphic functions.
One can then compare 
the congruence relations satisfied by $\ff$
and those satisfied by $\theta_\phi(\ff)$.
In favorable situations,
this produces an automorphic function of $G$
not coming from theta correspondences that is congruent to $\theta_\phi(\ff)$,
which has applications toward Iwasawa main conjectures or conjectures of Bloch-Kato type
(cf \cite{Hida1994} and \cite{Agarwal2013} for example).
Or conversely, 
when the congruence relations are better understood,
one can compare automorphic periods of $\ff$ and $\theta_\phi(\ff)$
and obtain integral relations among them \cite{Prasanna2006}\cite{Tilouine2022}.

When the target group $G$ is quasi-split,
the integrality, or the $p$-integrality for some prime $p$, 
of $\theta_\phi(\ff)$ can be detected by computing the Fourier coefficients.
This is the case in \cite{Hida1993} 
when $\ff=\chi$ is an algebraic Hecke character of an imaginary quadratic field $\K$
and $\theta_\phi(\ff)$ is the CM modular form with $q$-expansion
\begin{equation}\label{eq:classicalCM}
    \theta(\chi)=\sum_{\mathfrak{a}\subset \OK} \chi(\mathfrak{a})e^{2\pi i \Nr(\mathfrak{a})z}.
\end{equation}
This is also the case for Yoshida lift \cite{Hsieh2017},
in which $\ff$ is a pair of modular forms
and the Fourier coefficients of the Yoshida lift $\theta_\phi(\ff)$ on $\GSp(4)$
can be related to the Bessel periods.
And in \cite{Finis}, 
the Fourier-Jacobi components of the endoscopic lift $\theta_\phi(\ff)$ on $\GUU(2,1)$ 
of a modular form $\ff$ can be expressed as sums of values of $\ff$ at CM points.
In the above cases,
the authors can further study the modulo-$p$ non-vanishing properties of $\theta_\phi(\ff)$,
using the modulo-$p$ non-vanishing properties of the periods.

For more general $G$, 
one can instead consider whether
$\theta_\phi(\ff)$ has sufficiently many toric integrals that are rational or $p$-integral.
In \cite{Harris1999},
this approach is applied to $(G,H)=(\UU(n+1),\UU(n))$ to prove 
the rationality of the theta lift.
By seesaw arguments,
the toric integrals are reduced to 
theta lifts from $\UU(1)$ to $\UU(1)$,
for which the rationality is known.
And in \cite{Prasanna2006}, this approach is applied
to obtain integral Jacquet-Langlands-Shimizu correspondence,
in which case the toric integral can be expressed by the Waldspurger formula.

In this article, we consider the case $(G,H)=(\UU(2),\UU(1))$ for a definite $\UU(2)$
with a different method.
In this case $\ff=\eta$ is a character,
and we obtain $p$-integrality of $\theta_\phi(\ff)$ by
exploiting the fact that the theta lift on $\UU(2)$ can be obtained via the see-saw below
from the theta lift to the quasi-split $\UU(2,2)$,
which admits Fourier expansions similar to \eqref{eq:classicalCM}.
\begin{equation}\label{eq:seesaw_Intro}
    \begin{tikzcd}
        \UU(2,2) \arrow[dr,dash]& \UU(1)\times\UU(1)^2\\
        \UU(2)\times \UU(1)^2\arrow[u,"\iota"]\arrow[ur,dash] & \UU(1) \arrow[u,"\textnormal{diag}"]
    \end{tikzcd}
\end{equation}
This method is inspired by the similar approach to obtain 
Klingen-Eisenstein series with integral Fourier-Jacobi coefficients
from Siegel-Eisenstein series on larger groups,
which has been used in \cite{Skinner2013}, \cite{Hsieh2014},
and \cite{Wan2020} and obtain Eisenstein congruences.
We remark that while our method can be easily 
extended to cover the case when $G=\UU(m,n)$ is of general signature,
new idea is required for $H$ larger than $\UU(1)$
since theta correspondence doesn't admit general pull-back formulae as does Eisenstein series.

The integral theta lift thus obtained,
which we refer to as the pull-back theta lifts,
can be naturally deformed into a $p$-adic Hida family
when the unitary groups are defined on a CM extension $\K/\F$.
We now state the main result of the paper explicitly.
Let $\F$ be a totally real number field 
$\K$ be a totally imaginary quadratic extension over $\F$.
We assume $p$ is unramified and every primes of $\F$ 
above $p$ is split in $\K$, which allows us to fix a CM type $\Sigma$.
Let $G$ be the definite unitary group over $\F$ defined as
\[
    \UU(2)(R)=\{g\in \GL_2(\K\otimes_\F R)\mid gg^*=\id_2\}.
\]
To define the Weil representation we fix a character
$\chi$ of $\A_{\K}^\times/\K^\times$
such that $\chi\mid_{\A_\F^\times}=\qch_{\K/\F}$.
We further assume 
\begin{enumerate}[label={($\chi$\arabic*)}]
    \item 
    The central $L$-value $L(\frac{1}{2},\chi)$ is nonzero.
    \item 
    The conductor $\fc$ of $\chi$ is prime-to-$p$.
    \item
    The Hecke character 
    $\chi_\circ\coloneqq \chi|\cdot|_\K^{1/2}$ is algebraic and
    has the infinity type $\Sigma^c$
\end{enumerate}
Now let $\fs=\fs^c$ be a prime-to-$p\fc$
ideal of $\K$ consists only of split primes.
Let $\K_{p^n\fs}$ be the ray class field 
of conductor $p^n\fs$ over $\K$,
$\K_{p^\infty\fs}=\bigcup_{n}\K_{p^n\fs}$,
and $\fG_{\fs}^a$ be the Galois group of the maximal pro-$p$ anticyclotomic
subextension in $\K_{p^\infty\fs}$.
We say an algebraic Hecke characters $\eta$ is $\fs$-admissible if
\begin{enumerate}
    \item $\eta$ has the infinity type $k\Sigma$ for $k\geq 0$.
    \item the prime-to-$p$ part of the conductor of $\eta$ divides $\fs$.
\end{enumerate}

\begin{thm}\label{thm:intro1}
    There exists a Hida family, which we define to be a 
    measure $\euF^\circ_\fs$ on $\fG^a_\fs$
    valued in the space of p-adic modular forms of $\UU(2)$,
    which interpolates 
    pull-back theta lifts of $\eta$ for admissible $\eta$.
    
    Furthermore, when $p$ is sufficiently large with respect to $\chi$
    and $\chi$ satisfies the condition of \cite[Theorem A]{Hsieh12},
    then the Hida family is primitive. In other word,
    at each branch of characters of $\fG_\fs^a$
    there exists a specialization of $\euF^\circ$
    which is non-vanishing modulo $p$.
\end{thm}
The precise statements of the theorem are given respectively
in Theorem\ref{thm:family} and Theorem \ref{thm:primitive},
we refer to which for the precise meaning of $p$
being large enough.

While the conditions \ref{cond:chi2} and \ref{cond:chi3}
are technical, the condition \ref{cond:chi1}
is necessary for our construction since,
by the nature of the see-saw,
the pull-back theta lift of $\eta$ 
we obtained from \eqref{eq:seesaw_Intro}
is $\theta_\phi(\ff)$ multiplied by 
a square root of the central L-value $L(\frac{1}{2},\chi)$.
The same Hecke L-value (and some twist of which) also shows up 
when we consider the toric perids of $\theta_\phi(\ff)$ as in \cite{Harris1999}.

Let $\I_\fs=\oo\llbracket\fG_\fs^a\rrbracket$,
where $\oo$ is the ring of integers of a sufficiently large 
field extension over $\Qp$.
We then define a $\I_\fs$-valued inner product
$\B(\euF^\circ_\fs,\euF^\circ_\fs)$ of $\euF^\circ_\fs$
with itself.
By comparing the inner product with that of theta lifts
and apply the Rallis inner product formula,
we obtain an following 
$p$-adic $L$-function $\mathcal{L}_{\fs}$
that can be seen as 
a projection of Katz's $p$-adic L-function 
\cite{Katz1978}\cite{Hida1993}
projected to the anticyclotomic direction.
We refer the reader to Theorem \ref{thm:function} 
for a precise statement of the following result.
\begin{thm}\label{thm:intro2}
    Put $\mathcal{L}_\fn=\Omega_p^{-2\Sigma}
    \B(\euF^\circ_{\fs},U_\fs^{-1}\euF^\circ_{\fs})
    \in \I_\fs$
    where $U_\fs$ is some Hecke action defined at $\fs$. Then
    when $\eta$ is admissible with the infinity type $k\Sigma$, 
    \begin{equation*}
        \frac{1}{\Omega_p^{(2k+2)\Sigma}}
        \int_{\fG_{\fs}^a}\hat{\eta}\,L_\fs=C(\K,\chi)
        \left(\frac{2\pi}{\Omega_\infty}\right)^{(2k+2)\Sigma}
    \textnormal{Im}(\delta)^{k}\tilde{\eta}(z_\delta^{-1})
    \frac{\Gamma((k+2)\Sigma)}{(2\pi)^{(k+2)\Sigma}}
    L(1,\chi^{-2}\tilde{\eta})
    \prod_{v\mid p\fs}E_v(1,\chi^{-2}\tilde{\eta})
    \end{equation*}
    where $C(\K,\chi)$ is a constant independent of $\eta$,
    $z_\delta\in\A_{\K,f}^\times$ is a fixed element, and
    $E_v(1,\chi^{-2}\tilde{\eta})$ is the modified Euler factor 
    $(1-\chi^{-2}\tilde{\eta}(\varpi_\bw)q_\bw^{-1})
    (1-\chi^{2}\tilde{\eta}^{-1}(\varpi_w))/
    {\varepsilon(1,(\chi^{-2}\tilde{\eta})_w,\psi_w)}$.
\end{thm}

We briefly mention our motivations for this construction.

\subsection*{Euler systems}
A direct consequence of Theorem \ref{thm:intro2}
is that the congruence module of $\euF^\circ_{\fs}$
is torsion over $\I_\fs$ and annihilated by the 
anticyclotomic $p$-adic L-function $\mathcal{L}_{\fs}$.
Though the result is already proved in \cite{Hida1994} and \cite{Hida2007},
we obtain the extra structure that
the pairings $\B(\cdot, \euF^\circ_{\fs})$ define
a system of maps from the space of Hida families to $\I_\fs$
as $\fs$ varies.

In \cite{urban}, such system is used in the case of Eisenstein congruences
to recover the classical Euler system of cyclotomic units.
In up-coming work, we explain how to apply the machinery in \textit{loc.cit}
to obtain an Euler system for $\eta$
and discuss applications towards 
anticyclotomic Iwasawa main conjectures.

\subsection*{Properties toward the limit}

Let $B$ be the definite quaternion algebra over $\F$
defined by the Hilbert symbol $(-1,-1)_\F$,
then $\GUU(2)=B^\times\times_{\mathbf{G}_m} \Res_{\K/\F}\mathbf{G}_m$.
One can check that 
the pull-back theta lifts of $\eta$ constructed in this article
can be interpreted as the Jacquet-Langlands transfer
of the classical automorphic induction associated 
to a Hecke character.
And the toric periods we used to prove the primitivity 
can also be interpreted as the toric period of the automorphic induction on $\GL_2(\F)$.
Such period identities have been investigated in \cite{CHAN2020},
with potential applications toward \cite{Bertolini2013} through a limiting process.

Particularly in our case, one may consider specializing $\euF^\circ_{\fs}$ 
to the $p$-adic avatar of an algebraic Hecke character $\eta$ of infinity type $-\Sigma$.
The resulting $p$-adic modular form
is then a $p$-adic analogue of Jacquet-Langlands transfer of a weight one modular form
that doesn't exist classically,
which still produces $L$-values outside the classical range of interpolation.
We hope to further investigate these $p$-adic modular forms
and their relations to classical construction of special units
as in \cite{Bertolini2012}.

\subsection*{Structure of the article}

After fixing the notations, 
we recollect general results of Weil representations in \S 2.
In particular we explain how to twist 
the classical Rallis inner product formula into a bilinear version that
expressing bilinear pairings
and squares of toric periods
in terms of products of local integrals.
A similar twist is also necessary
for the Waldspurger formula of toric periods in \cite{Bertolini2013} and \cite{Hsieh14},
and is also needed for our arithmetic application 
since complex conjugations doesn't preserve $p$-integrality.

Section \S 3 contains all the computations that are needed for our construction.
In which we pick explicit choices of Schwartz functions
adapted to $\chi$ and an algebraic Hecke character $\eta$
of infinity type $k\geq 0$.
We then compute the local integrals of
the Fourier coefficients of theta lifts on $\UU(2,2)$,
the inner product and toric periods as explained in the previous section.

In section \S 4 we review the algebraic theory of modular forms on the quasi-split 
$\UU(2,2)$.
These results are combined with the Fourier coefficients computations made in \S 5
to construct the measure interpolating pull-back theta lifts in \S 6.
We then use the Hida theory on the definite $\UU(2)$ reviewed in section \S 5
to show that the inner product and toric integral of the pull-back theta lifts
produce the desired $p$-adic L-functions.

\subsection*{Acknowledgement}

The results of this work are parts of the the author's Ph.D. thesis in Columbia University.
The author would like to thank his advisor, Eric Urban, for introducing him to the subject.
The author would also like to thank Professor Ming-Lun Hsieh for his patience and hospitality 
in explaining his works and for his encouragement, especially during the pandemic.
The current version of the article is a slight improvement of the original thesis,
and was inspired by conversations with Professor Kartik Prasanna during the author's 
postdoctorship in the University of Michigan.
Part of the draft of the article was finished when the author visited 
NCTS in Taipei during the summer of 2024 and 2025 and
the author would like to thank the Center for the hospitality
and the support during the period of the time.

\section{Notations}

Throughout the article, $\F$ is a totally real field
and $\K$ is a totally imaginary quadratic extension over $\F$.
Let $\arch=\Hom(\F, \C)$ denote the set of archimedean places of $\F$,
and $\finite$ the set of finite places of $\F$.
We fix an odd prime $p$ that is unramified in $\F$ and satisfies
the ordinary condition:
\begin{equation}\label{cond:ord}\tag{ord}
    \text{Every finite place of $\F$ above $p$ is split in $\K$}.
\end{equation}
Let $\dd_\F$ and $\dd_\K$ denote respectively 
the absolute ideals of different of $\F$ and $\K$;
and $\dd_{\K/\F}$ denotes the relative ideal of different,
which satisfies $\dd_\K=\dd_{\K/\F}\dd_\F$.
For a technical reason, throughout we also assume that  
every finite place $v\in\finite$ which divides $2$ is split in $\K$.

Let $\OF$ and $\OK$ denote respectively the rings of integers
of $\F$ and $\K$.
We use the letter $v$ to denote a place
(either archimedean or finite) of $\F$
and the letter $w$ to denote a place of $\K$.
If $v\in\arch$ is an archimedean place,
let $\F_v=\R$ and $|\cdot|_v$ be the usual absolute value on which.
Similarly, if $w$ is an archimedean place of $\K$,
then we let $\K_w=\C$ and $|\cdot|_w$ be the square
of the usual absolute value, so $|z|_w=|z\bar{z}|_v$ if $w\mid v$.
If $v\in\finite$ is a finite place, we let
$\F_v$ denote the completion of $\F$ at $v$;
$\oo_v$ denote the ring of integers of $\F_v$;
$\varpi_v$ be a uniformizer of the discrete valuation ring $\oo_v$;
$q_v$ be the cardinality of the residue field $\oo_v/(\varpi_v)$;
$|a|_v=q_v^{-\ord_v(a)}$ be the non-archimedean norm on $\F_v$,
where $\ord_v\colon \F_v\to \Z\cup\{\infty\}$
is the normalized discrete valuation such that $\ord_v(\varpi_v)=1$.
We adopt the same notations when $v$ is replaced by a finite place
$w$ of $\K$. This implies $|a|_w=|\Nr_{\K_w/\F_v}(a)|_v$ when $w\mid v$.

Let $\A_\F$ and $\A_\K$ denote respectively the rings of adeles over 
$\F$ and $\K$.
If $G$ is an algebraic group over $\F$ and $g\in G(\A_\F)$,
we let $g_v$ denote the component of $g$ at the place $v$,
or in other word $g=(g_v)_{v\in \arch\cup\finite}$.
In general we write $g_S=(g_v)_{v\in S}$
for any subset of places $S\subset \arch\cup\finite$,
but we also write  $g_\infty=(g_v)_{v\in\arch}$,
$g_f=(g_v)_{v\in\finite}$, and $g_p=(g_v)_{v\mid p}$.
We will often view a vector space $\bV$ over $\F$
as the algebraic group over $\F$ defined by
$\bV(R)\coloneqq \bV\otimes_\F R$ when $R$ is an $\F$-algebra.
We adopt the same notations for algebraic groups over $\K$.

We let $c\in \Gal(\K/\F)$ denote the unique nontrivial element.
When it is clear from the context 
we use $c(z)$, $z^c$, and $\bar{z}$ interchangeably to denote 
the action of $c$ on $z\in \K$.
We then define
$\A_\K^1=\{z\in \A^\times_\K=
\K\otimes_\F \A_\F \mid z\bar{z}=\Nr_{\K/\F}z=1\}$ and
$\K_v^1=\{z\in \K_v\coloneqq \K\otimes_\F\F_v\mid z\bar{z}=1\}$
for any place $v$ of $\F$.
The $n$-dimensional affine space over $\A_\K$ is denoted by $\A_\K^n$, 
and simply by $\A_\K$ if $n=1$.
Thus $\A_\K^1$ should not be confused with the 1-dimensional affine space.
We make the same distinctions between $\K_v^n$ and $\K_v^1$.

\subsection{CM types}

Throughout the article we fix an embedding $\iota_\infty:\bar{\Q}\to \C$
and an isomorphism $\iota:\C\cong \C_p$.
These fix an embedding $\iota_p=\iota\circ\iota_\infty:\bar{\Q}\to \C_p$.
Let $\bar{\Z}_p$ be the $p$-adic completion of 
the ring of algebraic integers $\bar{\Z}$ in $\C_p$.
We then let $\fm_p$ denote the maximal ideal in $\bar{\Z}_p$
and define $\fm=\iota_p^{-1}(\fm_p)$.

Let $S_p$ and $S_p^\K$ denote respectively 
the set of places of $\F$ and $\K$ that are above $p$
and let $I_\K=\Hom(\K,\bar{\Q})$,
which we identify with  $\Hom(\K,\C)$ and $\Hom(\K,\C_p)$ via
compositions with $\iota_\infty$ and $\iota_p$.
Given $\sigma\in I_\K$,
we let $w_\sigma\in S_p^\K$ denote the place induced by
$\sigma_p\coloneqq \iota_p\circ \sigma\in\Hom(\K,\C_p)$.
Then for $w\in S_p^\K$ we define
\[
    I_w=\{\sigma\in I_\K\mid w=w_\sigma \}=\Hom(\K_w,\C_p).
\]
When $\Sigma$ is a subset of $I_\K$, we define
$\Sigma^c=\{\sigma c\mid \sigma\in \Sigma\}$,
$\Sigma_p=\{w_\sigma\mid \sigma\in \Sigma\}$, and 
$\Sigma_p^c=\{\bw\mid w\in \Sigma_p\}=\{w_\sigma\mid \sigma\in \Sigma^c\}$.
Throughout the article, we fix a $p$-ordinary CM type,
which is a subset $\Sigma\subset I_\K$ such that
\[
    \Sigma\sqcup \Sigma^c=I_\K,\quad
    \Sigma_p\sqcup \Sigma_p^c=S_p^\K.
\]
The existence of such a $\Sigma$ is guaranteed by \eqref{cond:ord}.
We can then identify $\Sigma$ 
either with $\arch=\Hom(\F,\C)$ via restrictions,
or with the set of archimedean places of $\K$.
Similarly we will identify $\Sigma_p$ with the set $S_p$.

\subsection{Characters}

Let $\kappa=\sum_{\sigma\in \Sigma} 
a_\sigma\sigma+b_\sigma\sigma c\in \Z[I_\K]$. We define 
\[
    \kappa(\alpha_\infty)=\alpha_\infty^\kappa=
    \prod_{\sigma\in \Sigma} 
    (\alpha_\sigma)^{a_\sigma}(\bar{\alpha}_\sigma)^{b_\sigma}\in \C^\times,\quad
    \kappa(\alpha_p)=\alpha_p^\kappa=
    \prod_{w\in \Sigma_p}
    \prod_{\sigma\in I_w}
    \sigma_p(\alpha_w)^{a_\sigma}\sigma_p(\alpha_{\bw})^{b_\sigma}
    \in \C_p^\times.
\]
for $\alpha_\infty=(\alpha_\sigma)_{\sigma\in\Sigma}
\in \A_{\K,\infty}^\times$
and $\alpha_p=(\alpha_w,\alpha_{\bw})_{w\in\Sigma_p}\in 
\prod_{w\mid p}\K_w^\times$.
When the meaning is clear from the context, we also 
let $\Sigma$ denote the formal sum $\sum_{\sigma\in\Sigma}\sigma$
and similarly for $\Sigma^c$.
If $\chi\colon \A_\K^\times/\K^\times\to \C^\times$ 
is a Hecke character of $\K$,
we let $\chi_w$ be the component of $\chi$ at a place $w$ of $\K$.
We also write  $\chi_\infty=(\chi_w)_{w\in\Sigma}$,
$\chi_p=(\chi_w)_{w\in S_p^\K}$, and
$\chi_v=(\chi_w)_{w\mid v}$ when $v$ is a place of $\F$.
We say $\chi$ is an algebraic Hecke character of
the infinity type $\kappa$ if
$\chi_\infty(\alpha)=\alpha^\kappa$.
When this is the case we define the $p$-adic avatar 
$\hat{\chi}\colon \A_\K^\times\to \bar{\Z}_p^\times$  of $\chi$ by
\[
    \hat{\chi}(\alpha)=
    \iota(\chi(\alpha)\alpha_\infty^{-\kappa})\alpha_p^{\kappa}.
\]

If $w$ is a finite place of $\K$, 
we let $L(s,\chi_w)$ denote the local L-function  associated to $\chi_w$
and define $L(s,\chi_v)=\prod_{w\mid v}L(s,\chi_w)$
when $v\in \finite$.
If $\eta$ is a character of $\A_\K^1/\K^1$, 
we let $\tilde{\eta}$ denote the base change of $\eta$ to $\A_\K^\times$,
which is the Hecke character defined by 
$\tilde{\eta}(\alpha)\coloneqq \eta(\alpha/\alpha^c)$.
In the case when $\eta$ is the restriction of a Hecke character,
we have $\tilde{\eta}=\eta^{1-c}\coloneqq \eta(\eta^c)^{-1}$,
where we write $\eta^c$ for the Hecke character defined as
$\eta^c(\alpha)=\eta(\bar{\alpha})$.

\subsection{Matrices}

Suppose $m\in\text{M}_{r,s}(\K\otimes_\F R)$,
where $R$ is an $\F$-algebra.
We let $m^\intercal, m^c$, and $m^*$ denote respectively
the transpose, conjugate, and conjugate-transpose of $m$.
In particular we can define the space 
of $n$-by-$n$ Hermitian matrices over $R$ as
\[
    \her_n(R)=\{
    m\in \text{M}_n(\K\otimes_\F R)\mid m=m^*\}.
\]
When $r=s$ and $g\in \GL_r(\K\otimes_\F R)$ is invertible, we write
$g^{-\intercal}=(g^{-1})^\intercal$ and $g^{-*}=(g^{-1})^*$.
We write $\mtr(m)$ for the trace of a square matrix $m$,
and reserve $\Tr$ for the traces between fields extensions.

When $v$ is a place of $\F$ that splits into $w\bw$ in $\K$
there exists isomorphisms $\K_w\cong \F_v\cong \K_{\bw}$,
which defines an isomorphism
$\K_v\coloneqq \K\otimes_\F\K_v=\K_w\times\K_\bw\cong\F_v^2$.
Let $\iota_w\colon \K_v\to \F_v$ denote the projection to the 
component at $w$ and define $z_w=\iota_w(z)$ when $z\in \K_v$.
We define $\iota_\bw$ and $z_\bw$ similarly
and adopt the same notations when $\K$ is replaced 
by any vector spaces over $\K$.
For example, we can also define 
$\her_n(\F_v)=\{m\in M_n(\K_v)\mid m_\bw=m_w^\intercal\}$.

\subsection{Unitary groups}
When $(\W, \langle \cdot,\cdot\rangle_\W)$ 
is a skew-Hermitian space over $\K$ of dimension $n$.
We fix a basis $\{\ee_1,\cdots,\ee_n\}$ of $\W$ such that
the skew-Hermitian matrix  
\begin{equation*}
    \bdelta\coloneqq (\langle \ee_i,\ee_j\rangle_\W)=
    \diag(\delta_1,\cdots,\delta_n)
\end{equation*}
is diagonal.
We identify
$\W$ with row vectors in $\K^n$ via the basis,
and define the unitary similitude group $\GUU(\W)$,
which is an algebraic group over $\F$, as
\[
    \GUU(\W)(R)=\{g\in \GL_{n}(\K\otimes_\F R)\mid g\bdelta g^*=\nu(g)\bdelta
    \text{ for some } \nu(g)\in R^\times\}.
\]
Let $\nu\colon\GUU(\W)\to \mathbb{G}_m$ be the character
that sends $g$ to the similitude factor $\nu(g)$
so that $\UU(\W)=\ker(\nu)$.
When $v$ splits into $w\bw$ in $\K$, we have the isomorphism
\[
    (\iota_w, \nu)\colon \GUU(\W)(\F_v)\cong \GL_n(\F_v)\times \F_v^\times
\]
for $\iota_w$ defined above.
Consequently $\iota_w\colon \UU(\W)(\F_v)\to \GL_n(\F_v)$
is also an isomorphism.

Let $-\W$ be the skew-Hermitian space over $\K$ of dimension $n$
with basis $\{\ee_1^-,\cdots \ee_n^-\}$ 
such that the associated skew-Hermitian matrix is $-\bdelta$.
We identify $\GUU(\W)=\GUU(-\W), \UU(\W)=\UU(-\W)$, and define
\[
    G(\UU(\W)\times \UU(-\W))(R)=
    \{
        (g_1,g_2)\in \GUU(\W)(R)\times\GUU(-\W)(R)\mid 
        \nu(g_1)=\nu(g_2)
    \}.
\]

Then we define the skew-Hermitian space $\V=\W\oplus(-\W)$
and fix the basis
\begin{equation}\label{def:basisV}
    \xx^i=(2\delta_i)^{-1}(\ee_i-\ee_i^-), \quad
    \yy^i=(\ee_i+\ee_i^-),\quad i=1,\cdots,n,
\end{equation}
under which the associated skew-Hermitian matrix is equal to
$w_n\coloneqq\smat{&\id_n\\-\id_n&}$.
Therefore the subspaces
$Y_\K=\bigoplus_{i=1}^n\K\yy^i$ and
$X_\K=\bigoplus_{i=1}^n\K\xx^i$ are maximal isotropic in $\V$.
We identify $\V$ with row vectors in $\K^{2n}$ via the above basis
and define the unitary similitude group $\GUU(\V)=\GUU(n,n)$ as
\[
    \GUU(\V)(R)=\{g\in \GL_{2n}(\K\otimes_\F R)\mid 
    gw_ng^*=\nu(g)w_n \text{ for some } \nu(g)\in R^\times\}.
\]
By definition, there exists a natural group homomorphism
$\iota\colon G(\UU(\W)\times \UU(-\W))\to \GUU(\V)$
\begin{equation}\label{eq:iota}
    \iota(g_1,g_2)=
    \smat{(2\bdelta)^{-1}&-(2\bdelta)^{-1}\\\id_n&\id_n}
    \smat{g_1&\\&g_2}
    \smat{(2\bdelta)^{-1}&-(2\bdelta)^{-1}\\\id_n&\id_n}^{-1}.
\end{equation}
We again let $\nu$ denote the similitude character
so that $\UU(\V)=\UU(n,n)=\ker(\nu)$.
And when $v$ splits into $w\bw$ in $\K$, we similarly have the isomorphisms
\[
    (\iota_w, \nu)\colon \GUU(\V)(\F_v)\cong
    \GL_{2n}(\F_v)\times \F_v^\times \quad
    \iota_w\colon \UU(\V)(\F_v)\cong \GL_{2n}(\F_v).
\]

We also define $(V, (\cdot,\cdot)_V)$ to be the Hermitian space
given by $V=\K$ and $(z,z')_V=\bar{z}z'$.
Then the similitude group $\GUU(V)$ is simply the Weil restriction
$R_{\K/\F}\mathbb{G}_m$, where the similitude character
is equal to $\Nr_{\K/\F}$, and $\UU(V)$ is equal to $\K^1$, 
which we view as an algebraic group over $\F$ by
\[
    \K^1(R)=\{z\in (\K\otimes_\F R)^\times \mid z\bar{z}=1\}.
\]
We shall make the same identifications for 
the unitary group associated to any 
1-dimensional Hermitian or skew-Hermitian spaces. At last, we define
\[
    G(\UU(\V)\times \UU(V))(R)
    =\{(g,h)\in \GUU(\V)(R)\times \GUU(V)(R)\mid 
    \nu(g)=\nu(h)\}
\]
and 
$\GUU(\V)(R)^+
=\{ g\in \GUU(\V)(R)\mid \nu(g)\in \Nr_{\K/\F}(\K\otimes_\F R)^\times\}$,
which is also the image of the projection of 
$G(\UU(\V)\times \UU(V))(R)$ to $\GUU(\V)(R)$.

\subsection{Measure}

Let $\psi\colon \A_\F/\F\to \C$ 
be the canonical additive character such that
$\psi_v(x)=e^{2\pi i x}$ for any archimedean place $v\in \arch$.
For any place $v$ of $\F$, we let 
$dx_v$ denote the self-dual Haar measure on $\F_v$ with respect to $\psi_v$
and $d^\times x_v$ denote the Haar measure on $\F_v^\times$ given by
\[
    d^\times x_v=
    \begin{cases}
        d^\times x_v=|x|_v^{-1}\,dx_v,& v\in \arch;\\
        d^\times x_v=(1-q_v^{-1})^{-1}\cdot|x|_v^{-1}\,dx_v,& v\in \finite.
    \end{cases}
\]
We recall that the measures have the following properties.
\begin{itemize}
    \item $dx_v$ is the usual Lebesgue measure on $\R$ when $v\in \arch$.
    \item $\vol(\oo_v,dx_v)=\vol(\oo_v^\times, d^\times x_v)=|\dd_\F|_v^{1/2}$ when $v\in\finite$.
\end{itemize}
For a place $w$ of $\K$, we let $dz_w$ denote the self-dual Haar measure
on $\K_w$ with respect to $\psi_w$, 
where $\psi_\K\coloneqq \psi\circ \Tr_{\K/\F}$ and 
$\psi_w$ is its component at $w$,
and $dz_w^\times$ denote the Haar measure on $\K_w^\times$
defined similarly as above.
Then $dz_w$ is twice of the usual Lebesgue measure on $\C$
when $w$ is archimedean and 
$\vol(\oo_w,dz_w)=\vol(\oo_w^\times, d^\times z_w)=|\dd_\K|_w^{1/2}$
when $w$ is finite.

If $v$ is a place of $\F$, we let $dh_v$ be the quotient measure
on $\K_v^1\coloneqq\K^1(\F_v)=\{z\in \K\otimes_\F\F_v\mid z\bar{z}=1\}$
associated to the exact sequence
\[
    1\to \A_\F^\times\to
    \A_\K^\times \xrightarrow{z\mapsto z/\bar{z}} \A_\K^1\to 1.
\]
We note that the measure satisfies the following properties.
\begin{itemize}
    \item $\vol(\K_v^1, dh_v)=2\pi$ when $v\in \arch$.
    \item $\vol(\K_v^1, dh_v)=e_v|\dd_{\K/\F}|^{1/2}_w|\dd_\F|_v^{1/2}$
    when $v\in\finite$ is non-split,
    $w$ is the unique place above $v$, 
    and $e_v$ is the ramification index of $\K_w/\F_v$.
    \item $dh_v=dx_v^\times$ under the isomorphism 
    $\iota_w\colon \K_v^1\cong \oo_v^\times$ 
    when $v$ splits into $w\bw$ in $\K$.
    \item $\vol([\A_\K^1], dh)=2L(1,\qch_{\K/\F})$, where
    $\qch_{\K/\F}$ is the quadratic Hecke character of
    $\A_\F^\times/\F^\times$ associated to $\K/\F$ 
    by the global class field theory.
\end{itemize}
On the other hand, if $G$ is the unitary group
associated to a (skew-)Hermitian space over $\K$ of dimension
strictly greater than 1, we let $dg$ be the Tamagawa measure
on $G(\A_\F)$. This implies in particular $\vol([G])=2$.

\section{Weil representations}

In this section we recall the definition
of Weil representations
of the dual reductive pairs
$\UU(\V)\times \UU(V)$ and
$\UU(\W)\times \UU(V)$
and provide an ad-hoc definition 
of pull-back theta lifts from $\UU(V)$ to $\UU(\W)$
via $\UU(\V)$.
Since $\UU(V)=\K^1$, 
the pull-back theta lifts can be easily related to the usual ones
by a see-saw argument.
We then apply the Rallis inner product formula
to obtain the two formulae
in Proposition \ref{prop:Rallis} 
and Proposition \ref{prop:period}, 
which express respectively
the Petersson inner products and
the toric period integrals of pull-back theta lifts
in terms of products of local integrals.

Let $\bV$ be the symplectic space over $\F$ underlying $V\otimes_\K\V=\V$.
Then the additive character $\psi$ determines a Weil representation 
$\omega^{\square}_\psi$ 
of a metaplectic cover of $\Sp(\mathbb{V})$.
Upon choosing an auxiliary Hecke character 
$\chi$ of $\A^\times_\K/\K^\times$ which satisfies 
$\chi\mid_{\A_\F}=\qch_{\K/\F}$,
we can define a splitting of the metaplectic cover over $\UU(\V)$
and the Weil representation 
$\omega^\square_{\psi,\chi}$ of $\UU(\V)\times \UU(V)$ \cite{Ku94}.

To be more precise, 
$\omega^\square_{\psi,\chi}(g,h)$
is the representation of $\UU(\V)\times \UU(V)(\A_{\F})$
acting on $\bs(\A_{\K}^n)$, the space of Bruhat-Schwartz functions
on $\A_\K^n$.
Locally at each place $v$ of $\F$,
the action of $\omega^\square_{\psi,\chi}(g,h)$
on $\varphi\in \bs(\K_v^n)$
is determined by the formulae
\begin{equation}
\begin{split}\label{eq:Weil}
    &\omega^\square_{\psi,\chi}(\smat{A&\\&A^{-*}},h)
    \varphi(z)=\chi(\det A)|\det A|_\K^{1/2}\varphi(h^{-1}zA),
    \quad A\in \GL_n(\K_v)\\
    &\omega^\square_{\psi,\chi}(\smat{\id_n&B\\&\id_n},1)
    \varphi(z)=\psi\big(\mtr(Bz^*z)\big)\varphi(z),
    \quad B\in \her_n(\F_v)\\
    &\omega^\square_{\psi,\chi}(\smat{&-\id_n\\\id_n&},1)
    \varphi(z)=\gamma_v^{-n}
    \int_{\K_v^n}\varphi(w)\psi(-\Tr_{\K_v/\F_v}(wz^*))\,dw.
\end{split}
\end{equation}
Here $dw=\prod_{i=1}^ndw_{i}$ is the product of 
the self-dual measures $dw_i$ on $\K_v$ 
with respect to $\psi_v\circ\Tr_{\K_v/\F_v}$;
and $\gamma_v\in \mu_8$ is the Weil index
of the quadratic space over $\F_v$ underlying $\K_v$
given by $(z,z')\coloneqq\Tr_{\K_v/\F_v}(\bar{z}z')$.
Let $\Delta=\delta^2\in \F^\times$
for any $\delta=-\delta^c\in \K_v^\times$, then by \cite[A.1]{Ichino05}
\begin{equation*}
    \gamma_v=
    \gamma(-\Delta,\psi_v)
    \cdot\gamma(-1,\psi_v)^{-1}=
    \gamma(\Delta,\psi_v)^{-1}.
\end{equation*}
\begin{rem}
    Let $\langle \cdot,\cdot \rangle_\V$
    be the skew-Hermitian pairing on $\V$.
    Our choice of the symplectic pairing
    $\llangle\cdot,\cdot\rrangle_{\mathbb{V}}
    \coloneqq\Tr_{\K/\F}(\langle\cdot,\cdot\rangle_\V)$ on $\bV$
    is different from that of \cite{Y97}, \cite{hks}
    and \cite{Ku94} by a factor of $2$.
\end{rem}

\subsection{Doubling method}
Let $\bW$ be the symplectic space over $\F$
underlying $V\otimes_\K\W=\W$ and
fix any polarization $\bW=\bX+\bY$.
The character $\psi$ also determines a Weil representation $\omega_\psi$
of a metaplectic cover of $\Sp(\bW)$ acting on $\bs(\bX(\A_\F))$,
the space of Bruhat-Schwartz functions on $\bX(\A_\F)$.
The same auxiliary character $\chi$
also determines a splitting 
of $\UU(\W)\times \UU(V)$ and
a Weil representation $\omega_{\psi,\chi}$ of which
acting on $\bs(\bX(\A_\F))$
through the doubling method \cite{hks}.
In short, the representation is obtained from
the Weil representration $\omega^\square_{\psi,\chi}$
of $\UU(\V)\times \UU(V)$
through the doubling homomorphism \eqref{eq:iota}.

\begin{rem}
In principle one needs a pair of auxiliary characters
$(\chi_\W,\chi_V)$ to define the 
Weil representation of $\UU(\W)\times \UU(V)$.
Here we simply identify $\UU(V)$ with the center of $\UU(\W)$ and put
\[
    \omega_{\psi,\chi}(g,h)\coloneqq 
    \omega_{\psi,\chi}(h^{-1}g).
\]
By \cite[Corollary A.8]{hks},
this amounts to choosing $\chi_\W=\chi$ 
and $\chi_V=\chi^n$.
\end{rem}

The Weil representations 
$\omega^\square_{\psi,\chi}(g,h)$ and $\omega_{\psi,\chi}(g,h)$
are unitary with respect to the Hermitian pairings
$(\varphi,\varphi')=\int_{\K_v^n}\varphi(z)\overline{\varphi'(z)}\,dz$ and
$(\phi,\phi')=\int_{\bX(\F_v)}\phi(x)\overline{\phi'(x)}\,dx$
on $\bs(\K_v^n)$ and $\bs(\bX(\F_v))$ respectively,
so we may identify the contragradient $\omega^\vee_{\psi,\chi}$
as the representation of $\UU(\W)\times \UU(V)$
acting on $\bs(\bX(\A_\F))$ via
$\omega_{\psi,\chi}^\vee(g,h)\overline{\phi}=\overline{\omega_{\psi,\chi}(g,h)\phi}$
and similarly for $\omega^{\square\vee}_{\psi,\chi}(g,h)$.
In fact, it can be checked from the recipe in \cite{Ku94} and \cite{hks} that
$\omega_{\psi,\chi}^\vee(g)=\omega_{-\psi,\chi^{-1}}(g)$.


Given $z\in K_v^n\cong X_\K(\K_v)$
(via the basis $\xx^i$),
we identify $z$ with its image under
the projection $\V=\W+(-\W)\to \W$
and decompose $z=x+y$ with respect to
$\bW=\bX\oplus \bY$.
Write $\llangle\cdot,\cdot\rrangle_\bW
=\Tr_{\K_v/\F_v}(\langle\cdot,\cdot\rangle_\W)$.
Then the Weil representations
$\omega^\square_{\psi,\chi}$ and
$\omega_{\psi,\chi}$ are related by the intertwining isometry
\begin{equation}\label{eq:intertwine}
    \delta_\psi:\bs(\bX(\F_v))\otimes \bs(\bX(\F_v))\to \bs(\K_v^n)
    \quad
    \delta_\psi(\phi\otimes\phi^-)(z)=\int_{\mathbb{X}(\F_v)} 
    \psi_v(2\llangle u, y\rrangle_\bW)\phi(u+x)\phi^-(u-x)\,du.
\end{equation}

Precisely, 
let $\varphi_i=\delta_\psi(\phi_i\otimes\phi_i^-)$ for $i=1,2$, then
\begin{itemize}
    \item $(\varphi_1,\varphi_2)=(\phi_1,\phi_2)(\phi_1^-,\phi_2^-)$.
    \item $\omega^\square_{\psi,\chi}(\iota(g_1,g_2),h)\varphi=
    \delta_\psi\big(\omega_{\psi,\chi}(g_1,h)\phi\otimes
    \chi(\det g_2)
    \omega_{\psi,\chi}(g_2,h)\phi^-\big)$.
\end{itemize}
In other word, the Weil representations
fit into the see-saw diagram
\[
    \begin{tikzcd}
        \UU(\V) \arrow[dr,dash]& \UU(V)\times\UU(V)\\
        \UU(\W)\times\UU(\W)\arrow[u,"\iota"]\arrow[ur,dash] & \UU(V) \arrow[u,"\textnormal{diag}"]
    \end{tikzcd}
\]
Consequently, let
$\theta_\varphi^{\square}(g,h)=
\sum_{z\in \K^n}\omega^{\square\bullet}_{\psi,\chi}(g,h)\varphi(z)$ and
$\theta_\phi^\bullet(g,h)=\sum_{x\in \bX}\omega_{\psi,\chi}(g,h)\phi(x)$
be the theta kernels associated to 
$\varphi\in \bs(\A_\K^n)$ and
$\phi\in \bs(\bX(\A_\F))$, then we have
\begin{align}
    \theta_\varphi^\square(\iota(g_1,g_2),h)=\chi(\det g_2)\cdot 
    \theta_{\phi}(g_1,h)\theta_{\phi^-}^\vee(g_2,h),\quad
    & \text{ if }\varphi=\delta_\psi(\phi\otimes\phi^-);\label{eq:res_kernelB}\\
    \theta_{\varphi'}^\square(\iota(g_1,g_2),h)=\chi(\det g_2)\cdot 
    \theta_{\phi}(g_1,h)\overline{\theta_{\phi^-}(g_2,h)}\label{eq:res_kernelH},\quad
    & \text{ if }\varphi'=\delta_\psi(\phi\otimes\bar{\phi}^-).
\end{align}
Here $\theta^\vee_{\phi^-}$ denote the theta kernel
similarly defined with respect to the contragredient
$\omega^\vee_{\psi,\chi}(g,h)$.

\subsection{Pull-back theta lifts}

We now fix the polarization 
$\bW=\bX\oplus\bY$ given by
\[
    \bX=\bigoplus_{i=1}^n \F\delta_i\ee_i,\quad
    \bY=\bigoplus_{i=1}^n \F\ee_i.
\]
We also suppress $\psi$ and $\chi$ from the Weil representations
and write 
$\omega^\square(g,h), \omega^{\square\vee}(g,h)$ 
for the representations of 
$\UU(\V)\times \UU(V)$ on $\bs(\A_\K^n)$ and
$\omega(g,h), \omega^\vee(g,h)$
for the representations of 
$\UU(\W)\times \UU(V)$ on $\bs(\bX(\A_\F))$.

Let $T\coloneqq\prod_{i=1}^n\UU(\W_i)$
be the product of the unitary groups $\UU(\W_i)$ associated to 
the 1-dimensional skew-Hermitian subspaces $\W_i=\K\ee_i\subset \W$,
which we identify with the diagonal torus in $\UU(\W)$.
If we fix the polarization
$\bX_i=\F\delta_i\ee_i$ and $\bY_i=\F\ee_i$
for the symplectic space $\bW_i$ underlying $\W_i$,
then for $\bbeta=\diag(\beta_1,\cdots,\beta_n)\in T$,
where $\beta_i\in \UU(\W_i)$,
the action of $\omega(\bbeta,h)$ is given by the tensor product of
the Weil representations of $\UU(\W_i)\times \UU(V)$
acting on $\bs(\bX_i(\A_\F))$.
By abuse of notation we write these as 
$\omega(\beta_i,h)=\omega(h^{-1}\beta_i)$ and 
$\omega^\vee(\beta_i,h)=\omega^\vee(h^{-1}\beta_i)$.
Note that the theta lift of a character $\eta(h)$ of $\UU(V)(\A_\F)$
with respect to the theta kernel associated to 
$\phi_i\in \bs(\bX_i(\A_\F))$ is
\begin{equation}\label{eq:dim_one_lift}
    \theta_{\phi_i}(\eta)(\beta_i)\coloneqq
    \int_{[\UU(V)]}\theta_{\phi_i}(\beta_i,h)\eta(h)\,dh=
    \theta_{\phi_i}(\eta)(1)\cdot \eta(\beta_i)
\end{equation}
and vice versa for theta lifts of characters of $\UU(\W_i)$.
The non-vanishing of such lifts from $\UU(1)$ to $\UU(1)$
is discussed in \cite{Y97} 
via the Rallis inner product formula.

\begin{defn}\label{def:pullback}
    Let $\eta$ and $\nu_1,\cdots,\nu_n$ be characters of $\A_\K^1/\K^1$.
    We view $\eta$ and $\bnu=\{\nu_i\}_{i=1}^n$
    as characters of $\UU(V)(\A_\F)$ and $T(\A_\F)$ respectively,
    where $\bnu(\bbeta)\coloneqq\prod_{i=1}^n\nu_i(\beta_i)$
    for $\bbeta=\diag(\beta_1,\cdots,\beta_n)$.
    We define the pull-back theta lifts of the character $\eta$
    to $\UU(\W)$ via $\UU(\V)$, with respect to $\bnu$ and the theta kernel
    associated to $\varphi(z)\in \bs(\A_\K^n)$, by
    \begin{equation*}
        \theta^\square_\varphi(\eta,\bnu)(g)=
        \int_{[T]}
        \theta^\square_\varphi(\eta)(\iota(g,\bbeta))
        \bnu(\bbeta)\,d\bbeta=
        \int_{[T]}\int_{[\UU(V)]}
        \theta^\square_\varphi(\iota(g,\bbeta),h)\eta(h)\,dh\,
        \bnu(\bbeta)\,d\bbeta.
    \end{equation*}
\end{defn}
\begin{prop}\label{prop:pullback}
    For $\eta$ and $\nu_1,\cdots,\nu_n$ as above,
    the pull-back theta lift $\theta^\square_\varphi(\eta,\bnu)(g)$
    belongs to the space of the usual theta lifts
    of $\UU(\W)\times \UU(V)$ of the character $\chi^n\nu\eta$,
    where $\nu(h)=\prod_{i=1}^n\nu_i(h)$ is the product 
    of the characters $\nu_1,\cdots,\nu_n$.
\end{prop}
\begin{proof}
Reduce to the case when the Schwartz function $\varphi\in \bs(\A_\K^n)$
is the image under $\delta_\psi$ of a pure tensor
$\phi\otimes \phi^-$ for $\phi,\phi^-\in \bs(\bX(\A_\F))$ and
$\phi^-=\otimes \phi^-_i, \phi^-_i\in \bs(\bX_i(\A_\F))$,
for which the associated theta kernel is
\[
\theta^\square_\varphi(\iota(g,\bbeta),h)=
\theta_\phi(g,h)\cdot 
\prod_{i=1}^n\chi(\beta_i)\theta_{\phi^-_i}^\vee(\beta_i,h)
\]
by \eqref{eq:res_kernelB}.
It then follows from the see-saw
\[
    \begin{tikzcd}
        \UU(\V) \arrow[dr,dash]& \UU(V)\times\UU(V)^n\\
        \UU(\W)\times T\arrow[u,"\iota"]\arrow[ur,dash] & \UU(V) \arrow[u,"\textnormal{diag}"]
    \end{tikzcd}
\]
that the pull-back theta lift 
$\theta^\square_\varphi(\eta,\bnu)(g)$ is
\begin{equation*}
    \int_{[\UU(V)]}
    \theta_\phi(g,h)\cdot \prod_{i=1}^n
    \int_{[\UU(\W_i)]}
    \chi\nu_i(\beta_i)\theta_{\phi^-_i}^\vee(\beta_i,h)\,d\beta_i\, \eta(h)\,dh
    \overset{\eqref{eq:dim_one_lift}}{=}
    \prod_{i=1}^n\theta_{\phi^-_i}^\vee(\chi\nu_i)(1)\cdot 
    \int_{[\UU(V)]}\theta_\phi(g,h)\chi^n\nu\eta(h)\,dh.
\end{equation*}
\end{proof}

\subsection{Inner products and toric periods}

From now on we assume that $\W$ is anisotropic,
so that Petersson inner products
of automorphic forms of $\UU(\W)$ are well-defined.
Suppose $\varphi_1,\varphi_2\in \bs(\A_\K^n)$
are Schwartz functions of the form
$\delta_\psi(\phi_j\otimes\phi^-_j)$ for $j=1,2$,
where $\phi_j,\phi^-_j\in \bs(\bX(\A_\F))$
and $\phi^-_j=\otimes_{i=1}^n\phi^-_{ji}$.
It follows from the same computation in the proof of 
Proposition \ref{prop:pullback} that
\begin{multline}\label{eq:innerproduct_1}
    \int_{[\UU(\W)]}
    \theta^\square_{\varphi_1}(\eta,\bnu)(g)
    \overline{\theta^\square_{\varphi_2}(\eta,\bnu)(g)}
    \,dg\\=
    \prod_{i=1}^n 
    \theta_{\phi^-_{1i}}^\vee(\chi\nu_i)(1)
    \overline{\theta_{\phi^-_{2i}}^\vee(\chi\nu_i)(1)}
    \cdot 
    \int_{[\UU(\W)]}
    \theta_{\phi_1}(\chi^n\nu\eta)(g)
    \overline{\theta_{\phi_2}(\chi^n\nu\eta)(g)}
    \,dg.
\end{multline}
Since $\UU(\W)$ and each $\UU(\W_i)$ are anisotropic
and thus in Weil's convergence range for Siegel-Weil formula,
assume further that all Schwartz functions above
are pure tensors of Schwartz functions at local places, 
we can apply the Rallis inner product formula
(cf \cite{MLS} or \cite[Thm 2.6.]{Y97} when $n=1$)
to each product above and obtain
\begin{align*}
    &\frac{2}{\vol([\UU(\W)])}
    \int_{[\UU(\W)]}
    \theta_{\phi_1}(\chi^n\nu\eta)(g)
    \overline{\theta_{\phi_2}(\chi^n\nu\eta)(g)}\,dg=
    \vol([\A_\K^1])\cdot
    \prod_v\int_{\K_v^1}(\omega(\id_n,h)\phi_{1,v},\phi_{2,v})
    \chi^n\nu\eta_v(h)\,dh,\\
    &\frac{2}{(\vol[\A_\K^1])}
    \int_{[\UU(V)]}
    \theta^\vee_{\phi^-_{1i}}(\chi\nu_i)(h)
    \overline{\theta^\vee_{\phi^-_{2i}}(\chi\nu_i)(h)}\,dh=
    2\cdot \theta^\vee_{\phi^-_{1i}}(\chi\nu_i)(1)
    \overline{\theta^\vee_{\phi^-_{2i}}(\chi\nu_i)(1)}\\
    &\phantom{\frac{2}{(\vol[\A_\K^1])}
    \theta^\vee_{\phi^-_{1i}}(\chi\nu_i)(h)}=
    \vol([\A_\K^1])\cdot\prod_v\int_{\K^1_v}(\omega^\vee
    (\beta_i,1)\phi^-_{1i,v},\phi^-_{2i,v})
    (\chi\nu_i)_v(\beta_i)\,d\beta_i\\
    &\phantom{\int_{[\UU(V)]}
    \overline{\theta^\vee_{\phi^-_{2i}}(\chi\nu_i)(h)}\,dh
    }=
    \vol([\A_\K^1])\cdot\prod_v\int_{\K^1_v}(\omega^\vee
    (\beta_i,h)\phi^-_{1i,v},\phi^-_{2i,v})
    (\chi\nu_i)_v(\beta_i)\,d\beta_i \cdot \chi\nu_i(h^{-1}).
\end{align*}
We remark that it is more common to use the Tamagawa measures
for the Siegel-Weil formula. This difference between measures
is reflected by the fractions of volumes in the front of above equations.

Since $\delta_\psi$ is an isometry and
$\omega^\square(\iota(g_1,g_2),h)\cong \omega(g_1,h)\chi(\det g_2)\omega^\vee(g_2,h)$,
we see that
\[
    (\omega(\id_n,h)\phi_{1,v},\phi_{2,v})
    \prod_{i=1}^n
    (\omega^\vee(\beta_i,h)\phi^-_{1i,v},\phi^-_{2i,v})(\chi\nu_i)_v(\beta_i)=
    (\omega^\square(\iota(\id_n,\bbeta),h)\varphi_{1,v},\varphi_{2,v})\bnu(\bbeta).
\]
Now the lemma below is a consequence of 
\eqref{eq:innerproduct_1} and the Rallis inner product formula.

\begin{lem}\label{lem:HermRallis}
    Assume that  $\varphi_1,\varphi_2\in \bs(\A_\K^n)$
    are pure tensor of Schwartz functions in $\bs(\K_v^n)$, then 
    \begin{multline*}
        \frac{2^{n+1}}{\vol[\UU(\W)]}
        \int_{[\UU(\W)]}
        \theta^\square_{\varphi_1}(\eta,\bnu)(g)
        \overline{\theta^\square_{\varphi_2}(\eta,\bnu)(g)}\,dg\\=
        \vol[\A_\K^1]^{n+1}\cdot 
        \prod_v\int_{T\times \UU(V)(\F_v)}
        (\omega^\square(\iota(1,\bbeta),h)
        \varphi_{1,v},\varphi_{2,v})\eta_v(h)\bnu_v(\bbeta)\,d\bbeta dh.
    \end{multline*}
\end{lem}

On the other hand, 
let $\mu_1,\cdots,\mu_n$ be characters of $\A^1_\K/\K^1$ and 
define the characters 
$\bmu=\{\mu_i\}_{i=1}^n$ of $T(\A_\F)$
and $\mu(h)=\prod_{i=1}^n\mu_i(h)$ of $\UU(V)(\A_\F)$ as above.
We would also like to consider the toric period integral 
of the pull-back theta lifts given by
\begin{equation}\label{eq:defnperiod}
    P_{\bmu}(\theta^\square_\varphi(\eta,\bnu))\coloneqq
    \int_{[T]}\theta^\square_\varphi(\eta,\bnu)(\balpha)\bmu(\balpha)\,d\balpha.
\end{equation}
\begin{lem}\label{lem:Hermperiod}
    Assume that  $\varphi_1,\varphi_2\in \bs(\A_\K^n)$
    are pure tensor of Schwartz functions in $\bs(\K_v^n)$, then 
    \begin{equation*}
        P_{\bmu}(\theta^\square_{\varphi_1}(\eta,\bnu))
        \overline{P_{\bmu}(\theta^\square_{\varphi_2}(\eta,\bnu))}
        =\frac{\vol[\A_\K^1]^{2n+2}}{4^n}
       \prod_v\int_{T\times T(\F_v)}
        (\omega^\square(\iota(\balpha,\bbeta),\id)
        \varphi_{1,v},\varphi_{2,v})\bmu_v(\balpha)\bnu_v(\bbeta)\,d\balpha\bbeta
    \end{equation*}
    and is nonzero only when the product 
    $\chi^n\mu\nu\eta$ is the trivial character.
\end{lem}
\begin{proof}
    It suffices to consider the case when 
    $\varphi_1,\varphi_2$ are images under $\delta_\psi$ of
    $\phi_j\otimes\phi^-_j$ for $j=1,2$, where
    $\phi_j,\phi_j^-\in \bs(\bX(\A_\F))$ and 
    $\phi_j=\otimes_{i=1}^n\phi_{ji}, \phi^-_j=\otimes_{i=1}^n\phi^-_{ji}$.
    Then 
    $\theta^\square_{\varphi_j}(\iota(\balpha,\bbeta),h)=
    \prod_{i=1}^n\theta_{\phi_{ji}}(\alpha_i,h)
    \theta^\vee_{\phi^-_{ji}}(\beta_i,h)\chi(\beta_i)$ and 
    \[
        P_{\bmu}(\theta^\square_{\varphi_1}(\eta,\bnu))=
        \prod_{i=1}^n\theta_{\phi_{1i}}(\mu_i)(1)
        \theta^\vee_{\phi^-_{2i}}(\chi\nu_i)(1)\cdot 
        \int_{[\UU(V)]}\chi^n\mu\nu\eta(h)\,dh,
    \]
    which is nonzero only if $\chi^n\mu\nu\eta\equiv 1$.
    The lemma then follows again from the Rallis inner product formula
    \begin{align*}
        &2\cdot \theta_{\phi_{1i}}(\mu_i)(1)
        \overline{\theta_{\phi_{2i}}(\mu_i)(1)}=
        \vol([\A_\K^1])\cdot\prod_v\int_{\K^1_v}(\omega
        (\alpha_i,1)\phi_{1i,v},\phi_{2i,v})
        (\mu_i)_v(\alpha_i)\,d\alpha_i,\\
        &2\cdot \theta^\vee_{\phi^-_{1i}}(\chi\nu_i)(1)
        \overline{\theta^\vee_{\phi^-_{2i}}(\chi\nu_i)(1)}=
        \vol([\A_\K^1])\cdot\prod_v\int_{\K^1_v}(\omega^\vee
        (\beta_i,1)\phi^-_{1i,v},\phi^-_{2i,v})
        (\chi\nu_i)_v(\beta_i)\,d\beta_i,
    \end{align*}
    and that
    $\prod_{i=1}^n
    (\omega^\vee(\alpha_i,1)\phi_{1i,v},\phi_{2i,v})
    (\omega^\vee(\beta_i,1)\phi^-_{1i,v},\phi^-_{2i,v})\chi_v(\beta_i)=
    (\omega^\square(\iota(\balpha,\bbeta),1)\varphi_{1,v},\varphi_{2,v})$.
\end{proof}

\subsection{Bilinear variants}

For the application of the article it is necessary
to avoid taking the complex conjugation,
which does not preserve $p$-integrality in general.
Therefore we need to provide bilinear versions
of the inner product formula and the toric period integral.

Let $\varphi^\flat(z)\coloneqq \varphi(-\bar{z})$
for $\varphi(z)\in \bs(\K^n_v)$.
We define a bilinear pairing on $\bs(\K_v^n)$ by
\begin{equation}\label{eq:bilinear}
    (\varphi_{1},\varphi_{2})^\flat\coloneqq
    (\varphi_{1},\bar{\varphi}^\flat_{2})=
    \int_{\K^n_v}\varphi_{1}(z)\varphi_{2}(-\bar{z})\,dz.
\end{equation}
\begin{lem}\label{lem:J}
    The element $J=\smat{\id_n&\\&-\id_n}\in \GUU(\V)(\F_v)$
    satisfies the following properties.
    \begin{enumerate}
        \item $J\iota(g_1,g_2)J=\iota(g_2,g_1)$
        for $g_1,g_2\in \UU(\W)$.
        \item $\omega^\square_{\psi,\chi}(JgJ)=
        \chi(\det g)\omega^{\square\, \vee}_{\psi,\chi}(g)
        =\chi(\det g)\omega^\square_{-\psi,\chi^{-1}}(g)$
        for $g\in \UU(\V)$.
        \item $(\omega^\square(g)\varphi_1,
        \omega^\square(J\bar{g}J)\varphi_2)^\flat=
        (\varphi_1,\varphi_2)^\flat$ for $g\in \UU(\V)$.
    \end{enumerate}
    
\end{lem}
\begin{proof}
    The first property follows directly from \eqref{eq:iota}.
    For the other two properties
    it suffices to verify the equations when 
    $g$ is as in \eqref{eq:Weil}, which is straightforward
    when $g=\smat{A&\\&A^{-*}}$ or $\smat{\id_n&B\\&\id_n}$.
    When $g=\smat{&-\id_n\\\id_n&}$, the second property follows from
    $\gamma(\Delta,-\psi)=\gamma(\Delta,\psi)^{-1}=
    \chi(-1)\gamma(\Delta,\psi)$.
    Write $\varphi=\omega^\square(\smat{&\id_n\\-\id_n&})\varphi_2$
    and observe that 
    \[
        \overline{\varphi}^b(z)=\overline{\varphi}(-\bar{z})=
        \gamma_v^{-n}
        \int_{\K_v^n}\overline{\varphi}_2(w)\psi_v(\Tr_{\K_v/\F_v}(wz^\intercal))\,dw=
        \omega^\square(\smat{&-\id_n\\\id_n&})\overline{\varphi}_2^b(z).
    \]
    Now the third property for $g=\smat{&-\id_n\\\id_n&}$ follows from
    \[
    (\omega^\square(g)\varphi_1,\varphi)^\flat=
    (\omega^\square(g)\varphi_1,\overline{\varphi}^\flat)=
    (\omega^\square(g)\varphi_1,
    \omega^\square(g)\overline{\varphi}_2^\flat)=
    (\varphi_1,\overline{\varphi}_2^\flat)=(\varphi_1,\varphi_2)^\flat.
    \]
\end{proof}

\begin{lem}
    For $\varphi\in \bs(\A_\K^n)$ 
    we have
    $\theta^{\square\vee}_\varphi(\iota(g,\bbeta),h)
    =\theta^\square_{\varphi^\flat}(\iota(\bar{g},\bar{\bbeta}),\bar{h})$.
    Consequently
    \begin{align}
        \overline{\theta^\square_{\varphi}(\eta,\bnu)(g)}&=
        \int_{[T]}\int_{[\UU(V)]}
        \theta^{\square}_{\bar{\varphi}^\flat}(\iota(\bar{g},\bar{\bbeta}),\bar{h})\eta(\bar{h})\,dh
        \bnu(\bar{\bbeta})\,d\bbeta=
        \theta^\square_{\bar{\varphi}^\flat}(\eta,\bnu)(\bar{g}),
        \label{eq:Herm_to_Bilinear1}\\
        \overline{P_{\bmu}(\theta^\square_\varphi(\eta,\bnu)))}&=
        \int_{[T]}
        \theta^\square_{\bar{\varphi}^\flat}(\eta,\bnu)(\bar{\balpha})\bmu(\bar{\balpha})\,d\alpha=
        P_{\bmu}(\theta^\square_{\bar{\varphi}^\flat}(\eta,\bnu))).
        \label{eq:Herm_to_Bilinear2}
    \end{align}
\end{lem}
\begin{proof}
    Since $\overline{\mu_i(\alpha_i)}=\mu_i(\alpha_i)^{-1}=
    \mu_i(\bar{\alpha}_i)$ for each $i$,
    it suffices to prove the first equation, 
    which can be further reduced to the case
    when $\varphi=\delta_\psi(\phi\otimes\phi^-)$
    for $\phi(x),\phi^-(x)\in \bs(\bX(\A_\F))$.
    Apply the first two properties of the previous lemma, we see
    \[
        \theta^{\square\vee}_\varphi(\iota(g,\bbeta),h)=
        \chi(\det g\bbeta)^{-1}\theta^\square_\varphi(J(g,\bbeta)J,h)=
        \chi(\det\bbeta)^{-1}
        \theta_\phi(\bbeta,h)\theta^\vee_{\phi^-}(g,h)=
        \chi(\det\bar{\bbeta})
        \theta_{\phi^-}(\bar{g},\bar{h})
        \theta^\vee_\phi(\bar{\bbeta},\bar{h})
    \]
    where the last equality follows from \cite[Lem 2.2(iii)]{hks}
    since $g\mapsto \bar{g}=\Ad(g_0)g$ is the MVW-involution
    induced by the $\F$-linear map $g_0\in \End_\F(\W)$
    that acts as the complex conjugation with respect to the basis 
    $\{\ee_1,\cdots,\ee_n\}$.
    Now the equations follow from observing that for $z=x+y$
    \[
        \varphi(-\bar{z})=\delta_\psi(\phi\otimes\phi^-)(-x+y)=
        \int \psi(2\llangle u, y\rrangle_{\bW})\phi(u-x)\phi^-(u+x)\,du=
        \delta_\psi(\phi^-\otimes\phi)(x+y).
    \]
    Here recall that in \eqref{eq:intertwine},
    the decomposition $z=x+y$
    for $z=(z_i)\in \K^n$ 
    has $x$ being the the imaginary part of 
    $\sum_{i=1}^n (2\delta_i)^{-1}z_i\ee_i$,
    thus $-\bar{z}=-x+y$ if $z=x+y$.
\end{proof}

For $\varphi_{1,v},\varphi_{2,v}\in \bs(\K_v^n)$ 
we define the local integrals
\begin{align*}
    Z(\eta,\bnu,\varphi_{1,v}, \varphi_{2,v})&=
    \int_{T\times \UU(V)(\F_v)}
    (\omega^\square(\iota(\id,\bbeta),h)
    \varphi_{1,v},\varphi_{2,v})^\flat\eta_v(h)\bnu_v(\bbeta)\,d\bbeta dh\\
    P(\bmu,\bnu,\varphi_{1,v}, \varphi_{2,v})&=
    \int_{T\times T(\F_v)}
    (\omega^\square(\iota(\balpha,\bbeta),\id)
    \varphi_{1,v},\varphi_{2,v})^\flat\bmu_v(\balpha)\bnu_v(\bbeta)\,d\balpha\bbeta
\end{align*}

The following propositions follow from combining 
\eqref{eq:Herm_to_Bilinear1} with Lemma \ref{lem:HermRallis};
\eqref{eq:Herm_to_Bilinear2} with Lemma \ref{lem:Hermperiod},
and substituting $\varphi_2$ for $\bar{\varphi}_2^\flat$.

\begin{prop}\label{prop:Rallis}
    Under the same assumption in \ref{lem:HermRallis}
    \begin{equation*}
    \frac{2^{n+1}}{\vol[\UU(\W)]}
    \int_{[\UU(\W)]}
    \theta^\square_{\varphi_1}(\eta,\bnu)(g)
    \theta^\square_{\varphi_2}(\eta,\bnu)(\bar{g})\,dg\\=
    \vol[\A_\K^1]^{n+1}\prod_vZ(\eta,\bnu,\varphi_{1,v}, \varphi_{2,v}).
    \end{equation*}
\end{prop}
\begin{prop}\label{prop:period}
    Under the same assumption in \ref{lem:Hermperiod}
    and suppose $\chi^n\mu\nu\eta$ is the trivial character, then
    \begin{equation*}
        P_{\bmu}(\theta^\square_{\varphi_1}(\eta,\bnu))
        P_{\bmu}(\theta^\square_{\varphi_2}(\eta,\bnu))
        =\frac{\vol[\A_\K^1]^{2n+2}}{4^n}\prod_v
        P(\bmu,\bnu,\varphi_{1,v}, \varphi_{2,v}).
    \end{equation*}
\end{prop}

\section{Main computations}

We fix throughout the article a unitary Hecke character $\chi$
of $\A_\K^\times/\K^\times$ 
such that $\chi\mid_{\A_\F^\times}=\qch_{\K/\F}$.
Furthermore, we assume that $\chi$ satisfies the following conditions.
\begin{enumerate}[label={($\chi$\arabic*)}]
    \item \label{cond:chi1}
    The central $L$-value $L(\frac{1}{2},\chi)$ is nonzero.
    \item \label{cond:chi2}
    The conductor $\fc$ of $\chi$ is prime-to-$p$.
    \item \label{cond:chi3}
    The Hecke character 
    $\chi_\circ\coloneqq \chi|\cdot|_\K^{1/2}$ is algebraic and
    has the infinity type $\Sigma^c$
\end{enumerate}

We then choose $\delta\in \K^\times$ such that $\delta=-\delta^c$ and 
satisfies the conditions 
\begin{enumerate}[label={($\delta$\arabic*)}]
    \item\label{delta_cond1}
    $\prod_{w\mid v} \chi_w(\delta)\varepsilon(\frac{1}{2},\chi^{-1}_w,\psi_w)=1$
    for all places $v$ of $\F$.
    \item\label{delta_cond2} 
    $\delta$ belongs to $\oo_w^\times$ for all $w\in S_p^\K$.
\end{enumerate}
Indeed, an element $\delta'$ that satisfies \ref{delta_cond1} always exists
and is unique up to $\Nr(\K^\times)$
since \ref{cond:chi1} implies that the global $\varepsilon$-factor
$\varepsilon(\frac{1}{2},\chi^{-1})=1$.
And since every place above $p$ is split in $\K$,
by strong approximation we may choose a suitable $\alpha\in\K^\times$
such that $\delta\coloneqq\delta'\Nr_{\K/\F}(\alpha)$ 
satisfies \ref{delta_cond2}.
We also note that \ref{delta_cond1} and \ref{cond:chi3} 
together implies that $\text{Im}(\sigma(\delta))>0$
for all $\sigma\in\Sigma$ since 
$\varepsilon(\frac{1}{2},\chi^{-1}_\sigma,\psi_\sigma)=\sqrt{-1}$ and 
$\chi_\sigma(\delta)=-\sigma(\delta)/|\delta\overline{\delta}|^{1/2}$.

Let $\W$ be the 2-dimensional skew-Hermitian space over $\K$
defined by $\bdelta=\delta\id_2$ and $\V=\W+(-\W)$. 
Note that since $\bdelta=\delta\id_2$ we have
$\UU(\W)(R)=\{g\in \GL_{2}(\K\otimes_\F R)\mid gg^*=\id_2\}$.
We shall apply the results from the previous section 
to the Weil representations
$\omega^\square(g,h)$ of $\UU(\V)\times \UU(V)$ and
$\omega(g,h)$ of $\UU(\W)\times \UU(V)$ 
determined by $\psi$ and $\chi$.

To distinguish with the notation in the general cases,
a Schwartz function in $\bs(\A_\K^2)$ will be denoted by 
$\Phi(z), z\in \A_\K^2$.
Following \cite[\S 4]{Ichino05},
the Weil representation $\omega^\square(g,h)$ on $\UU(\V)\times \UU(V)$
can be extended to $G(\UU(\V)\times \UU(V))$ by
\[
\omega^\square(g,h)\Phi=
\omega^\square(g\smat{\id_2&\\&\nu(g)^{-1}\id_2},1)L(h)\Phi
\]
where $L(h)\Phi(z)\coloneqq|\Nr_{\K/\F} h|_{\F}^{-1}\Phi(h^{-1}z)$.
Since $\Nr_{\K_v/\F_v} h=\nu(g)$ 
for $(g,h)\in G(\UU(\V)\times \UU(V))(\F_v)$ we have
\begin{equation}\label{def:Weil_ext}
    \omega^\square(g,h)\Phi(z)=|\nu(g)|_{\F}^{-1}
    \omega^\square(g\smat{\id_2&\\&\nu(g)^{-1}\id_2},1)\Phi(h^{-1}z),\quad
    (g,h)\in G(\UU(\V)\times \UU(V))(\F_v).
\end{equation}
The theta kernel $\theta^\square_\Phi(g,h)$ associated to 
a Schwartz function $\Phi\in \bs(\A_\K^2)$ can then be extended to 
$(g,h)\in G(\UU(\V)\times \UU(V))(\A_\F)$ as well,
which allows us to define theta lifts of Hecke characters.

\begin{defn}\label{def:theta_ext}
We define the theta lift $\theta^\square_\Phi(\eta)$ of 
a Hecke character $\eta$ of $\A_\K^\times/\K^\times$
to $\GUU(\V)(\A_\F)$ as follows. 
First suppose $g\in \GUU(\V)(\A_\F)^+$, then we put
\begin{equation*}
    \theta^\square_\Phi(\eta)(g)
    =|\nu(g)|_\F
    \int_{[\UU(V)]}\theta^\square_P(g,hh_0)\eta(hh_0)\,dh,\,
    \text{ where } h_0\in \A_\K^\times
    \text{ is such that } \Nr(h_0)=\nu(g).
\end{equation*}
This is clearly independent of the choice of $h_0$ and is
an extension of the usual theta lifts of $\eta\vert_{\A_\K^1}$.
Since the function defined above
is left-invariant by $\GUU(\V)(\F)\cap \GUU(\V)(\A_\F)^+$,
it can be extended first 
to a left-$\GUU(\V)(\F)$-invariant function on
\[
\GUU(\V)(\F)\GUU(\V)(\A_\F)^+
=\{g\in \GUU(\V)(\A_{\F})\mid \qch_{\K/\F}(\nu(g))=1\},
\]
which is a connected open subgroup of index $2$ in $\GUU(\V)(\A_{\F})$,
then extended by zero to an automorphic function on $\GUU(\V)(\A_\F)$.
Note that by definition the central character
of $\theta^\square_\Phi(\eta)$ is equal to $\chi^2_\circ\eta$.
\end{defn}
Recall that if $\ff(g)$ is an automorphic function on $\GUU(\V)(\A_\F)$
and $\beta\in \her_2(\F)$ is a Hermitian matrix, 
the $\beta$-th Fourier coefficient of $\ff(g)$ is given by 
\[
    W_\beta(g,\ff)=
    \int_{[\her_2(\A_\F)]}
    \ff(\smat{\id_2&B\\&\id_2}g)\psi(-\mtr(B\beta))\,dB.
\]

\begin{lem}\label{lem:Whittaker}
    Let $\beta\in \her_2(\F)$ be a Hermitian matrix
    and $\theta^\square_\Phi(\eta)$ be the theta lift
    of a character $\eta$ of $\A_\K^\times/\K^\times$.
    Then the $\beta$-th Fourier coefficient
    $W_\beta(g, \theta^\square_\Phi(\eta))$ is nonzero only when
    \[
        \beta=z^*z=
        \smat{\bar{z}_1z_1& \bar{z}_1z_2\\\bar{z}_2z_1 & \bar{z}_2z_2}
        \text{ for some }
        z=(z_1,z_2)\in \K^2.
    \]
    Furthermore, when $\beta=z^*z$ is nonzero and
    $\Phi=\otimes_v\Phi_v$ is a pure tensor of local Schwartz functions,
    for $g\in \GUU(\V)(\A_\F)^+$ we have 
    $W_\beta(g,\theta^\square_\Phi(\eta))=\prod_vW_z(g_v,\eta,\Phi_v)$, 
    where
    \begin{equation*}
        W_z(g_v,\eta,\Phi_v)\coloneqq\int_{\UU(V)(\F_v)}
        \omega^\square
        (g_v\smat{\id_2&\\&\nu(g)_v^{-1}\id_2},1)\Phi_v((hh_0)^{-1}z)\eta_v(hh_0)\,dh
    \end{equation*}
    in which $h_0\in \A_\K^\times$ is any element
    such that $\nu(g)=\Nr(h_0)$.
    Note that while the local integrals does depend on the choice of $z$,
    the global product of which only depends on $\beta=z^*z$.
\end{lem}
\begin{proof}
    Since $\omega^\square(\smat{\id_2&B\\&\id_2})\Phi(z)=
    \psi(\mtr(Bz^*z))\Phi(z)$ by \eqref{eq:Weil} we see
    \[
        W_\beta(g,\theta^\square_\Phi(\eta))=
        |\nu(g)|_\F
        \sum_{z\in \K^2}\int_{[\UU(V)]}\int_{[\her_2(\A_\F)]}
        \omega^\square(g,hh_0)\Phi(z)\eta(hh_0)\,dh\,
        \psi(\mtr(B(z^*z-\beta))\,dB
    \]
    is nonzero only when 
    $\Omega_\beta\coloneqq\{z\in \K^2\mid z^*z=\beta\}$ is nonempty.
    If furthermore $\beta\neq 0$, then $\Omega_\beta$
    is a single orbit under the action of $\UU(V)(\F)=\K^1$, therefore
    \begin{equation*}
        |\nu(g)|_\F\sum_{z\in \Omega_\beta}\int_{[\UU(V)]}
        \omega^\square(g,hh_0)\Phi(z)\eta(hh_0)\,dh\\=
        |\nu(g)|_\F\int_{\UU(V)(\A_\F)}
        \omega^\square(g,h_0)\Phi(h^{-1}z)\eta(hh_0)\,dh
    \end{equation*}
    for any $z\in\Omega_\beta$
    and the lemma follows from applying \eqref{def:Weil_ext}
    to the right hand side of the equation.
\end{proof}

Our goal in the rest of the section is 
to make an explicit choice of local Schwartz functions 
$\Phi_v\in \bs(\K_v^2)$ at each place $v$ and compute 
the local integrals for
\begin{itemize}
    \item the Fourier coefficients of $\theta^\square_\Phi(\eta)$
    in Lemma \ref{lem:Whittaker}.
    \item the inner product of $\theta^\square_\Phi(\eta,\bnu)$
    with respect to $\bnu=\{\nu_1,\nu_2\}, \nu_1=\nu_2=\chi^{-1}$
    in Proposition \ref{prop:Rallis}.
    \item the toric period integral 
    $P_{\bmu}(\theta^\square_\Phi(\eta,\bnu))$
    for $\bnu$ as above and $\bmu=\{\mu_1,\mu_2\}$, where $\mu_1=(\mu\eta)^{-1}$ and
    $\mu_2=\mu$ for some finite order character $\mu$,
    in Proposition \ref{prop:period}.
\end{itemize}
The results are collected respectively in 
Proposition \ref{prop:Fourier_local},
Proposition \ref{prop:Rallis_local},
and Proposition \ref{prop:period_local}.

\begin{rem}
The computation in the proof of
Proposition \ref{prop:pullback} shows that 
the pull-back theta lift $\theta^\square_\Phi(\eta,\bnu)$ 
for $\bnu$ as above involves theta lifts between $\UU(1)\times\UU(1)$
of the trivial character.
Thus by the $\varepsilon$-dichotomy \cite{hks}
it is necessary to assume \ref{cond:chi1} and \ref{delta_cond1}
for the pull-back theta lift to be nontrivial.
\end{rem}

\subsection{The choice of Schwartz functions}\label{sec:choice}

Let $\fs=\bar{\fs}$ be an ideal of $\OK$
which is prime to $p\fc$ and is divisible only by split primes,
which we sometimes identify with an ideal of $\OF$.
For each $v\in\finite$ which divides $\fs$
we fix a choice of a finite place $w$ above $v$.
This amounts to fixing a decomposition
$\fs=\mathfrak{f}\bar{\mathfrak{f}}$ 
such that $\mathfrak{f}+\bar{\mathfrak{f}}=\OK$.
Henceforth, if $v=w\bw$ is a split place and $v\mid p\fs$,
we will assume either $w\in\Sigma_p$ or $w\mid\mathfrak{f}$.

Given an ideal $\fs$ as above, we say an algebraic Hecke character 
$\eta$ of $\A_\K^\times/\K^\times$
is $\fs$-admissible if 
\begin{itemize}
\item the prime-to-$p$ part of the conductor of $\eta$ divides $\fs$,
\item $\eta$ has the infinity type $k\Sigma$
for some integer $k\geq 0$,
\end{itemize}
and admissible if it is $\fs$-admissible for some ideal $\fs$ as above.

We now assign local Bruhat-Schwartz functions 
to an $\fs$-admissible character $\eta$ of the infinity type $k\Sigma$.

\subsubsection{The archimedean case}

When $\sigma\in \Sigma$ let $\C(\sigma)=\C$
denote the $\K\otimes\C$-module on which $\K$ acts through $\sigma$
and define $\X_\sigma=\{Z\in M_2(\C(\sigma))\mid i(Z^*-Z)>0\}$.
Then for $Z_\sigma\in \X_\sigma$ and $P_\sigma\in L_k$,
the subspace of homogeneous polynomials of degree $k$ in $\C[z_1,z_2]$,
we define the Schwart function
\begin{equation*}
    \Phi_{P_\sigma,Z_\sigma}(z)
    =P_\sigma(z)e^{2\pi i(zZ_\sigma z^*)}
    =P_\sigma(z)e^{2\pi i\mtr(Z_\sigma z^*z)}.
\end{equation*}

\subsubsection{The non-split case}

Suppose $v\in\finite$ is non-split and $w$ is the unique place above $v$.
We define 
\[
    \Phi_v(z_1,z_2)=\id_{(\varpi^m_w)}(z_1)\id_{(\varpi^m_w)}(z_2)
\]
for some $m\in \Z$ which will be zero for almost all $v$.
Here $\id_X$ denotes the characteristic function of a set $X$.
The values of $m$ will be determined in Proposition \ref{prop:nonsplit}.

\subsubsection{The split case}

Suppose $v$ splits into $v=w\bw$ in $\K$.
Recall that we have isomorphisms
\begin{align*}
    &\iota_w,\iota_\bw\colon \UU(\V)(\F_v)\to \GL_4(\F_v)&
    &g=(g_w,g_\bw)\quad 
    g_\bw=\smat{&-\id_2\\\id_2}g_w^{-\intercal}\smat{&\id_2\\-\id_2}\\
    &\iota_w,\iota_\bw\colon \UU(\W)(\F_v)\to \GL_2(\F_v)&
    &g=(g_w,g_\bw)\quad 
    g_\bw=g_w^{-\intercal}\notag
\end{align*}
for $g_w=\iota_w(g)$ and $g_\bw=\iota_\bw(g)$.
Similarly, for $z=(z_i,z_2)\in \K_v^2$ we write $z_i=(a_i,b_i)\in \F_v^2$
with $a_i=\iota_w(z_i)$ and $b_i=\iota_\bw(z_i)$.
Recall that when $v\mid p\fs$ we assume 
$w\mid \mathfrak{f}$ or $w\in\Sigma_p$.
We then define $\Phi_v(z)=\Phi_v(z_1,z_2)=\Phi_v(a_1,b_1,a_2,b_2)$ by
$\Phi_v(z)=\phi_{1,w}(a_1)\phi_{1,\bw}(b_1)
\phi_{2,w}(a_2)\phi_{2,\bw}(b_2)$
where $\phi_{i,w}, \phi_{i,\bw}$ are chosen as follows. 
\begin{enumerate}[label={($\phi$\arabic*)}]
    \item $\phi_{i,w}=\id_{\oo_v}=\phi_{i,\bw}$ for $i=1,2$ when $v\nmid p\fc\fs$.
    \item $\phi_{i,w}=\chi^{-1}_w\id_{\oo_v^\times}=\phi_{i,\bw}$
    for $i=1,2$ when $v\mid \fc$.
    \item $\phi_{1,w}(a_1)=\tilde{\eta}_w(a_1)\id_{\oo_v^\times}(a_1)$
    and $\phi_{2,w}=\id_{\oo_v}=\phi_{i,\bw}$ for $i=1,2$ when $v\mid p\fs$.
\end{enumerate}
Here recall that $\tilde{\eta}=\eta^{1-c}$ is the base change
and the component of which at $w$ is $\tilde{\eta}_w=\eta_w\eta_\bw^{-1}$.

\begin{lem}\label{lem:invarianceK}
    Suppose $v\mid \fs$ and $n=\val_v(\fs\dd_\F)$,
    then $\omega^\square(g,1)\Phi_v=\Phi_v$ for
    \[
        g \in\UU(\V)(\F_v)\cap \GL_4(\OK\otimes_{\OF}\oo_v)
        \text{ such that }
        g\equiv\smat{1&*&*&*\\&*&*&*\\&&1&\\&&*&*}\bmod\varpi_v^{n}.
    \]
    The same holds for $v\mid p$ as well if we replace 
    $\fs$ by the conductors of $\eta$ at $v$.
\end{lem}
\begin{proof}
    By the Iwahori decomposition,
    it suffices to verify the equality when $g$ is of the following forms.
    \begin{itemize}
    \item 
    $g=\smat{A&\\&A^{-*}}$ for $A\in \GL_2(\OK\otimes\oo_v)$ such that
    $A\equiv \smat{1&*\\&*}\bmod \varpi_v^n$;
    \item
    $g=\smat{\id_2& B\\&\id_2}$ for $B\in \her_2(\oo_v)$; 
    \item 
    $g=\smat{\id_2& \\ B&\id_2}$ for $B\in (\varpi_v^n)\her_2(\oo_v)$.
    \end{itemize}
    The verification is straightforward in the first two cases.
    For the third case, let $w=\smat{&-\id_2\\\id_2&}$
    and observe that $\smat{\id_2& \\ B&\id_2}=
    w^{-1}\smat{\id_2& -B\\ &\id_2}w$.
    Since $\Phi'=\omega^\square(w,1)\Phi$ is supported on
    $z=(a_1,b_1,a_2,b_2)$ such that $b_1\in (\dd_{\F}\fs)_v^{-1}$ and 
    $a_1,a_2,b_2\in (\dd_\F)_v^{-1}$, for which 
    $\mtr(Bz^*z)\subset (\dd_{\F})_v^{-1}$
    and hence $\psi_v(\mtr(Bz^*z))=1$
    when $B\in (\varpi_v^{n})\her_2(\oo_v)$.
    Consequently $\omega^\square(\smat{\id_2& -B\\ &\id_2},1)\Phi'=\Phi'$
    and the lemma follows.
\end{proof}

\subsection{Fourier coefficients}

\begin{lem}\label{lem:cuspidal}
    Let $A\in \GL_2(\A_{\K,f})$.
    The Fourier coefficient 
    $W_\beta(\smat{A&\\&A^{-*}},\theta^\square_\Phi(\eta))$ 
    is zero when $\beta=0$.
\end{lem}
\begin{proof}
    We have seen that $W_\beta(g,\theta^\square_\Phi(\eta))$ 
    is nonzero only when $\beta=z^*z$.
    If $\beta=0$, then $z=(0,0)$ is the unique vector such that
    $\beta=z^*z$.
    Since $\Phi(0)=0$ by our choice of $\Phi_v$ at $v\mid p$ we see that
    \[
        W_\beta(\smat{A&\\&A^{-*}},\theta^\square_\Phi(\eta))
        =\vol([\her_2])\int_{[\UU(V)]} 
        \chi_\circ(\det A)\Phi(0)\eta(h)\,dh=0.
    \]
\end{proof}

On the other hand, if $\beta=z^*z$ is nonzero,
by Lemma \ref{lem:Whittaker} we have
\[
    W_\beta(\smat{A&\\&A^{-*}},\theta^\square_\Phi(\eta))=
    \chi_\circ(\det A) \prod_v \int_{\UU(V)(\F_v)}
    \Phi_v(h^{-1}zA)\eta(h)\,dh.
\]
This suggests us to consider the local integrals
$W_z(\id_2,\eta,\Phi_v)=\int_{\UU(V)(\F_v)}
\Phi_v(h^{-1}z)\eta(h)\,dh$ for any $z\in\K_v^2$.

\begin{prop}\label{prop:Fourier_local}
    For each place $v$ and $z=(z_1,z_2)\in \K_v^2$
    we put $\beta=z^*z\in\her(\F_v)$.
    Furthermore, when $v=w\bw$ and $z=(z_1,z_2)=(a_1,b_1,a_2,b_2)$
    we assume that neither $(a_1,a_2)$ or $(b_1,b_2)$ is the zero vector.
    \begin{enumerate}
        \item At an archimedean place $\sigma\in\Sigma$, we have
        \[
             W_z(\id_2,\eta, \Phi_{P_\sigma, Z_\sigma})
             =\vol(\C^1)\cdot P_\sigma(z)e^{2\pi i\mtr(Z_\sigma\beta)}.
        \]
        \item At a non-split place $v$, 
        there exists a lattice $\Xi_{v}\subset\her(\F_v)$ such that
        \[
            W_z(\id_2,\eta, \Phi_v)=\vol(\K_v^1)\cdot
            \id_{\Xi_{v}}(\beta).
        \]
        \item At $v=w\bw, v\nmid p\fc\fs$,
        suppose $a_i\neq0$, then
        there exists a lattice $\Xi_{v}\subset \her(\oo_v)$
        and a polynomial $R_{z,v}(T)\in \Z[T]$ such that
        \[
            W_z(\id_2,\eta,\Phi_v)=\vol(\oo_v^\times)
            \cdot\tilde{\eta}_w(z_i)
            R_{z,v}(\tilde{\eta}_w^{-1}(\varpi_w))\id_{\Xi_{v}}(\beta).
        \]
        \item At $v=w\bw, v\mid \fc$,
        there exists a lattice $\Xi_{v}\subset \her(\oo_v)$ such that
        \[
            W_z(\id_2,\eta, \Phi_v)=\vol(\oo_v^\times)
            \cdot \tilde{\eta}_w(z_1)
            \chi^{-1}_w(\beta_{11}\beta_{22})\id_{\Xi_{v}}(\beta).
        \]
        \item At $v=w\bw, v\mid p\fs$,
        there exists a lattice $\Xi_{v}\subset \her(\oo_v)$ such that
        \[
            W_z(\id_2,\eta, \Phi_v)=\vol(\oo_v^\times)
            \cdot  \tilde{\eta}_w(z_1)
            \id_{\Xi_{v}}(\beta).
        \]
    \end{enumerate}
    In all the above cases, we put
    $\tilde{\eta}_w(z_1)=\tilde{\eta}_w(a_1)=0$ if $a_{1}=0$.
\end{prop}
\begin{proof}
    At $\sigma\in\Sigma$ the integrand in 
    $W_z(\id_2,\eta,\Phi_{P_\sigma,Z_\sigma})$ is constant
    since $\omega^\square(\id_2,h)\Phi_{P_\sigma,Z_\sigma}= 
    h^{-k}\Phi_{P_\sigma,Z_\sigma}
    =\eta_\sigma(h)^{-1}\Phi_{P_\sigma,Z_\sigma}$
    and the result follows directly.

    At a non-split place $v$ the character $\eta_w$,
    for the unique place $w\mid v$, is trivial on
    $\K_v^1\subset \oo_w^\times$ by assumption.
    Therefore the integrand is again constant and 
    $W_z(\id_2,\eta,\Phi_v)=\vol(\K_v^1)\cdot \id_{\Xi_v}(\beta)$ for
    $\Xi_{v}=(\varpi^{2m}_w)\her(\oo_v)$.
    
    At $v=w\bw$ we write $h\in \K_v^1$ as $h=(x,x^{-1})$
    for $\iota_w(h)=x=\iota_\bw(h)^{-1}\in \F_v^\times$.
    Then $\eta(h)=\eta_w(x)\eta_{\bw}^{-1}(x)=\tilde{\eta}_w(x)$ and 
    in all cases
        \begin{multline}\label{eq:Wsplit}
        W_z(\id_2,\eta,\Phi_v)=
        \int_{\F_v^\times}
        \phi_{1,w}(x^{-1}a_1)\phi_{1,\bw}(xb_1)
        \phi_{2,w}(x^{-1}a_2)\phi_{2,\bw}(xb_2)\tilde{\eta}_w(x)\,d^\times x\\
        =\vol(\oo_v^\times)\cdot \sum_{m\in \Z}
        \phi_{1,w}(\varpi_v^{-m}a_1)\phi_{1,\bw}(\varpi_v^mb_1)
        \phi_{2,w}(\varpi_v^{-m}a_2)\phi_{2,\bw}(\varpi_v^mb_2)
        \tilde{\eta}_w(\varpi_v^m).
    \end{multline}
    By our choice of $\Phi_v$ the summand at $m$ is nonzero only when
    $-\val_v(b_i)\leq m\leq \val_v(a_j)$ for all $i,j=1,2$.
    Therefore \eqref{eq:Wsplit} is actually a finite sum by 
    our assumptions on $z$.
    Moreover notice that the set of such $m$ is nonempty
    only when $\beta\in \her(\oo_v)$.

    When $w\nmid p\fc\fs$ and suppose $a_i=z_{i,w}\neq0$,
    then $\eqref{eq:Wsplit}$ is equal to 
    $\vol(\oo_v^\times)\cdot \tilde{\eta}_w(z_i)
    R_{z,v}(\tilde{\eta}_w^{-1}(\varpi_w))$ for
    \[
        R_{z,v}(T)\coloneqq 
        \sum_m 
        \id_{\oo_v}(\varpi^{-m}a_1)\id_{\oo_v}(\varpi^{m}b_1)
        \id_{\oo_v}(\varpi^{-m}a_2)\id_{\oo_v}(\varpi^{m}b_2) T^{\val_v(a_i)-m}.
    \]
    Then the third claim follows if we take 
    $\Xi_v=\her_2(\oo_v)\subset \her_2(\F_v)$.

    When $w\mid \fc$ the right hand side of \eqref{eq:Wsplit} is equal to
    \[
        \vol(\oo_v^\times)\cdot 
        \chi_w^{-1}(b_1a_1)\chi_w^{-1}(b_2a_2)\tilde{\eta}_w(\varpi_v^m)
        \sum_m 
        \id_{\oo^\times_v}(\varpi_v^{-m}a_1)
        \id_{\oo^\times_v}(\varpi_v^{m}b_1)
        \id_{\oo^\times_v}(\varpi_v^{-m}a_2)
        \id_{\oo^\times_v}(\varpi_v^{m}b_2)
    \]
    and the summand is nonzero when $-\val_v(b_i)=m=\val_v(a_j)$,
    in which case $\tilde{\eta}_w(\varpi_v^m)=\tilde{\eta}_w(a_1)$
    since $\eta$ is assumed to be unramified at $\fc$.
    Such an $m$ exists and is unique when $\beta$ belongs to the lattice
    \[
        \Xi_v\coloneqq\{\beta=(\beta_{ij})\in \her_2(\F_v)
        \mid \beta_{ij}\in \oo_v^\times \text{ for all }i,j \}.
    \]
    The fourth claim now follows from that 
    $a_1b_1=\beta_{11}$ and $a_2b_2=\beta_{22}$.

    When $v\mid p\fs$ the right hand side of \eqref{eq:Wsplit} is
    \[
        \vol(\oo_v^\times)\cdot \tilde{\eta}_w(a_1)
        \sum_m 
        \id_{\oo_v^\times}(\varpi^{-m}a_1)\id_{\oo_v}(\varpi^{m}b_1)
        \id_{\oo_v}(\varpi^{-m}a_2)\id_{\oo_v}(\varpi^{m}b_2).
    \]
    The summand is nonzero when $m$ satisfies
    $-\val_v(b_i)\leq m\leq \val_v(a_j)$ and $m=\val_v(a_1)$.
    Then the last claim follows if we take 
    $\Xi_v=\{\beta\in\her(\oo_v)\mid (\beta_{11},\beta_{12})\neq 0\}$.
\end{proof}

Put $C(\K^1)=\prod_{v\in\finite\, nonsplit}\vol(\K^1_v)
\prod_{v\in\finite\, split}\vol(\oo_v^\times)$,
which is a $p$-unit since $p$ is assumed to be unramfied in $\F$,
and define $\Xi=\prod_{v\in\finite}\Xi_v\subset \her(\A_{\F,f})$.
We summarize the above computations into a global result.
\begin{cor}\label{cor:Fourier}
    Suppose $\beta=z^*z$ for $z\in \K^2$ is nonzero
    and $A\in\GL_2(\A_{\K,f})$.
    Write $z'=zA$ as an element in $\A_{\K,f}^2$.
    Then $z'_v$ satisfies the assumption in the proposition above.
    Let $z''\in \A_{K,f}^{(p)}$ be the product of 
    \begin{itemize}
        \item $z_{i,w}'$ at $v=w\bw$ when $v\nmid p\fc\fs$,
        where $z_{i,w}'\neq 0$.
        \item $z_{1,w}'$ when $v=w\bw$ and $v\mid \fc\fs$.
    \end{itemize}
    Then $W_\beta(\smat{A&\\&A^{-*}},\theta^\square_\Phi(\eta))$
    is equal to 
    \[
    (2\pi)^{\Sigma}C(\K^1)
    \chi_\circ(\det A)\tilde{\eta}(z'')\cdot 
    \prod_{\sigma\in\Sigma}P_\sigma(z)
    \prod_{w\in \Sigma_p}\tilde{\eta}_w(z'_1)
    \cdot \prod_{v\in\finite\text{ split }}R_{z',v}(\tilde{\eta}_w^{-1}(\varpi_w))\id_{\Xi}(A^*\beta A)e^{2\pi i \sum_\sigma (Z_\sigma\beta)}
    \]
    where $R_{z',v}$ is as in the proposition above when 
    $v\nmid p\fc\fs$ and is the constant function
    $\chi_w^{-1}(\beta_{11}'\beta_{22}')$ when $v\mid\fc$,
    where $\beta'\coloneqq z'^{*}z'=A^*\beta A$, 
    and $R_{z',v}\equiv 1$ when $v\mid p\fs$.
\end{cor}

\begin{rem}
    Consider the Weil representation 
    $\omega^\square_1(g,h)$ of $G(\UU(1,1),\UU(V))$ on $\bs(\A_\K)$.
    A simplied version of the above computation shows that 
    the theta lift $\theta^\square_\varphi(\eta)(g)$ on $\GUU(1,1)$
    associated to the Schwartz function
    \[
        \varphi_\sigma(z)=z^{k_\sigma}e^{-2\pi\Nr(z)},\sigma\in \arch,\quad
        \varphi_v(z)=
        \otimes_{w\mid v}\varphi_w(z),\,
        \varphi_w=
        \begin{cases}
            \id_{\oo_w}(z), & w\nmid \fn,\\
            \eta_w(z)\id_{\oo_w^\times}(z), & w\mid \fn,
        \end{cases}\, v\in \finite,
    \]
    has Fourier expansion $\sum_{\fa\subset \OK, (\fa,\fn)=\OK}
    \eta^{-1}(\fa)e^{2\pi i \Nr(\fa)z}$.
    In other word, $\theta^\square_{1,\varphi}(\eta)(g)$
    is the classical automorphic induction 
    attached to the character $\eta^{-1}$ when restricted to $\GL_2(\F)$.
\end{rem}

\subsection{Inner product}\label{sec:Rallis}

Fix $\bnu=\{\nu_1,\nu_2\}, \nu_1=\nu_2=\chi^{-1}$.
We follow Proposition \ref{prop:Rallis} and write
\begin{equation}\label{eq:Rallis_local}
    Z(\eta,\Phi_1,\Phi_2)\coloneqq Z(\eta, \bnu,\Phi_1,\Phi_2)=
    \int_{(\K_v^1)^3}
    (\omega^\square(\iota(\id_2, \smat{\beta_1&\\&\beta_2}), h)\Phi_{v}, \Phi_v)^\flat
    \eta_v(h)\chi_v^{-1}(\beta_1\beta_2)\,d\beta_1d\beta_2dh,
\end{equation}
where we identify $(T\times \UU(V))(\F_v)=(\K_v^1)^3$.
When $\Phi_1=\Phi=\Phi_2$ we simply write
$Z(\eta,\Phi_1,\Phi_2)=Z(\eta,\Phi)$.

\begin{prop}\label{prop:Rallis_local}
    \hfill
    \begin{enumerate}
        \item At an archimedean place $\sigma\in \Sigma$
        let $P_\sigma(z)=z_1^{k}$ and 
        $Z_{\delta,\sigma}=-\sigma(2\delta)^{-1}\id_2\in\X_\sigma$,
        then
        \[
            Z(\eta_\sigma, \Phi_{P_\sigma, Z_{\delta,\sigma}})=
            \vol(\C_1)^3\cdot
            \frac{\textnormal{Im}(-\sigma(\delta))^{k+2}}{(2\pi)^k}\Gamma(k+1) 
        \]
        \item At a non-split place $v$ let $w\mid v$,
        there exists $\xi_v\in \UU(\V)(\F_v)$, which is independent of
        $\eta$ and equals $\id$ for almost all $v$, such that
        \[
            Z(\eta_v, \omega^\square(\xi_v)\Phi_v)=
           \vol(\K_v^1)^3\cdot c_v
        \]
        where $c_v\coloneqq |4\delta|_wq_w^{-2(e_v-1)}$ is a $p$-unit.
        \item At $v=w\bw, w\nmid p\fc\fs$,
        there exists $\xi_v\in \UU(\V)(\F_v)$
        independent of $\eta$ such that
        \[
            Z(\eta_v, \omega^\square(\xi_v)\Phi_v)=
            \tilde{\eta}_w^{-1}(\delta\dd_\F)\cdot\vol(\oo_v^\times)^3\cdot
            c_v\left[\frac{L(\frac{1}{2}, \chi_v)}{L(1,\qch_v)}\right]^2
            \frac{L(1,(\chi^{-2}\tilde{\eta})_v)}{\zeta_v(2)},
        \]
        where $c_v$ is the $p$-unit
        $\chi^{-4}_w(\delta)\chi^4_w(\dd_\F)|\dd_\F|_v^2$
        and $\qch_v$ is the component of $\qch_{\K/\F}$ at $v$.
        \item At $v=w\bw, w\mid \fc$,
        there exists $\xi_v\in \UU(\V)(\F_v)$
        independent of $\eta$ such that
        \[
            Z(\eta_v, \omega^\square(\xi_v)\Phi_v)=
            \tilde{\eta}_w^{-1}(\fc\delta\dd_\F)\cdot 
            \vol(\oo_v^\times)^3 \cdot c_v(1-q_v^{-1})^2
            \left[\frac{L(\frac{1}{2}, \chi_v)}{L(1,\qch_v)}\right]^2
        \]
        where $c_v$ is the p-unit
        $\chi^{-4}_w(\delta)|\dd_\F|_v^2
        \varepsilon(\frac{1}{2}, \chi_w,\psi_w)^4$.
        \item At $v=w\bw\mid p\fs$,
        there exists $\xi_v\in \UU(\V)(\F_v)$
        independent of $\eta$ such that
        \begin{multline*}
         Z(\eta_v,\bnu, 
         \omega^\square(\iota(\tau_v^n,1)\xi_v)\Phi_v,
         \omega^\square(\xi_v)\Phi_v)=
         \tilde{\eta}^{-1}_w(\delta)
         \cdot\vol(\oo_v^\times)^3 \cdot c_v(1-q_v^{-1})
        \left[\frac{L(\frac{1}{2}, \chi_v)}{L(1,\qch)}\right]^2\\
        \cdot q_v^{-n}(\chi_\circ^{-1}\tilde{\eta})_w(\varpi_w^n)\cdot
        \frac{L(1,(\chi^{-2}\tilde{\eta})_w)
        \varepsilon(0,(\chi^2\tilde{\eta}^{-1})_w,\psi_w)}
        {L(0,(\chi^2\tilde{\eta}^{-1})_w) }
        \end{multline*}
        where $c_v$ is the p-unit
        $\chi^{-5}_w(\delta)\chi^2_w(\dd_\F)|\dd_\F|_v^{5/2}
        \varepsilon(\frac{1}{2}, \chi_w,\psi_w)^2$,
        $\tau_v\in \UU(\W)(\F_v)$ is defined by
        $\iota_w(\tau_v)\coloneqq\smat{\varpi_v&\\&1}$,
        and $n\in\Z_{\geq 1}$ is such that 
        $\tilde{\eta}_w\mid_{1+\varpi_w^n\oo_w}=1$.
    \end{enumerate}
\end{prop}

Let $C(\K^1)=\prod_{v\in\finite\, nonsplit}\vol(\K^1_v)
\prod_{v\in\finite\, split}\vol(\oo_v^\times)$
and write $\Phi$ for the product of
\begin{itemize}
\item $\Phi_{P_\sigma,Z_{\delta,\sigma}}$ for $\sigma\in\Sigma$
and $P_\sigma(z)=z^k$,
\item $\omega^{\square}(\xi_v)\Phi_v$ for $v\in\finite$.
\end{itemize}
\begin{cor}\label{cor:Rallis}
    Decompose $\fc=\fc_+\fc_-$,
    where $\fc_+$ is the part of split places and
    $\fc_-$ is the part of non-split places.
    Let $C_1=\delta^{2\Sigma}\prod_{v\in\finite}c_v$ 
    for $c_v$ from the previous proposition and
    $z_\delta\in \A_{\K,f}^\times$ be the product of 
    \begin{itemize}
        \item $(\delta\dd_\F)_w$ at $v=w\bw$ and $w\nmid p\fc\fs$.
        \item $(\fc\delta\dd_\F)_w$ at $v=w\bw$ and $w\mid \fc$.
        \item $(-\delta)_w$ at $v=w\bw$ and $w\mid p\fs$.
    \end{itemize}
    and let $\tau$ is the product of $\tau_v^n$ at $v\mid p\fs$, then
    \begin{multline}\label{eq:fullRallis}
        \int_{[\UU(\W)]}
        \theta^\square_{\Phi}(\eta,\bnu)(g)
        \theta^\square_{\Phi}(\eta,\bnu)(\bar{g}\tau)\,dg=
        2(2\pi)^{3\Sigma}C(\K^1)^3C_1
        \frac{L(1,\epsilon_{\K/\F})}{\zeta^{(\fc)}(2)}
        \frac{\prod_{v\mid c^+}(1-q_v^{-1})^2}
        {\prod_{v\mid c^-}(1+q_v^{-1})}L(\frac{1}{2},\chi)^2\\
        \cdot \textnormal{Im}(-\delta)^{k\Sigma}\tilde{\eta}(z_\delta^{-1})
        \frac{\Gamma((k+1)\Sigma)}{(2\pi )^{k\Sigma}}
    L(1,\chi^{-2}\tilde{\eta})
    \prod_{v\mid p\fs}
    \frac{q_v^{-n}(\chi_\circ^{-1}
    \tilde{\eta})_w(\varpi_w^n)}{(1+q_v^{-1})}\cdot
    \frac{(1-\chi^{-2}\tilde{\eta}(\varpi_\bw)q_\bw^{-1})
    (1-\chi^{2}\tilde{\eta}^{-1}(\varpi_w))}
    {\varepsilon(1,(\chi^{-2}\tilde{\eta})_w,\psi_w)}.
    \end{multline}
    Here 
    $\Gamma((k+1)\Sigma)\coloneqq\prod_{\sigma\in\Sigma}\Gamma(k+1)$.
    And for each $v\mid p\fs$, we have fixed $v=w\bw$
    where $w\in\Sigma_p$ or $w\mid\mathfrak{f}$.
\end{cor}
\begin{proof}
    Using that 
    $\varepsilon(1,(\chi^2\tilde{\eta}^{-1})_w,\psi_w)
    \varepsilon(1,(\chi^{-2}\tilde{\eta})_w,\psi_w)=\tilde{\eta}_w(-1)$
    and $(1-q_v^{-1})(1+q_v^{-1})=1/\zeta_v(2)$ at $v\mid p\fs$,
    and that at all non-split $v\in\finite$
    \[
        1=\left[\frac{L(\frac{1}{2}, \chi_v)}
        {L(1,\qch_v)}\right]^2
        \frac{L(1,(\chi^{-2}\tilde{\eta})_v)}{\zeta_v(2)},
    \]
    the product of all the local terms from
    the proposition above is equal to 
    \begin{multline*}
    (2\pi)^{3\Sigma}C(\K^1)^3C_1
    \left[\frac{L(\frac{1}{2},\chi)}
    {L^{(\fc^-)}(1,\qch_{\K/\F})}\right]^2
    \frac{\prod_{c^+}(1-q_v^{-1})^2}{\zeta^{(\fc)}(2)}\\\cdot
    \textnormal{Im}(-\delta)^{k}\tilde{\eta}(z_\delta^{-1})
    \frac{\Gamma((k+1)\Sigma)}{(2\pi)^{k\Sigma}}
    L(1,\chi^{-2}\tilde{\eta})
    \prod_{p\fs}\frac{ q_v^{-n}(\chi_\circ^{-1}
    \tilde{\eta})_w(\varpi_w^n)}{(1+q_v^{-1})}\cdot
    \frac{(1-\chi^{-2}\tilde{\eta}(\varpi_\bw)q_\bw^{-1})
    (1-\chi^{2}\tilde{\eta}^{-1}(\varpi_w))}
    {\varepsilon(1,(\chi^{-2}\tilde{\eta})_w,\psi_w)}.
    \end{multline*}
    Then the corollary follows from Proposition \ref{prop:Rallis}
    and that $\vol([\A_\K^1])=2L(1,\epsilon_{\K/\F})$.
\end{proof}

\subsubsection{The archimedean case}\label{sec:arch}

At $\sigma\in \Sigma$, the action of
$g=\smat{A&B\\C&D}\in \GUU(\V)(\F_\sigma)^+$
on $Z_\sigma\in \X_\sigma$ given by 
$g_\sigma Z_\sigma=(AZ_\sigma+B)(CZ_\sigma+D)^{-1}$
is transitive and defines the cocyle
\[
J(g,Z_\sigma)=(\kappa(g,Z_\sigma),\mu(g,Z_\sigma))\in
\GL_2(\C)\times \GL_2(\C), \text{ where }
\kappa(g, Z_\sigma)=\bar{C}Z_\sigma+\bar{D}, 
\mu(g, Z_\sigma)=CZ_\sigma+D,
\]
which satisfies the usual relation
$J(g_1g_2,Z_\sigma)=J(g_1,g_2Z_\sigma)J(g_2,Z_\sigma)$.

\begin{defn}\label{def:weight}
    We define the representation $\rho(h_1,h_2)$
    of $\GL_2(\C)\times\GL_2(\C)$ on $L_k$ by
    \[
        \rho(h_1,h_2)P(z)=\det(h_2)^{-1}P(zh_1^{-\intercal}),
        z=(z_1,z_2).
    \]
    Note that $P(z)=z_1^{k}$ is a vector 
    of the minimal weight $-(k,0,1,1)$
    (with respect to the upper-triangular Borel subgroup).
\end{defn}

\begin{lem}\label{lem:arch}
    Given $P_\sigma\in L_\sigma$, $Z_\sigma\in \X_\sigma$,
    and $(g,h)\in G(\UU(\V)\times\UU(V))(\F_\sigma)$, then
    \begin{equation*}
        \omega^\square(g,h)\Phi_{P_\sigma, Z_\sigma}(z)
        =\nu(g)^{k+1}\cdot \rho(J(g,Z_\sigma))P_\sigma(h^{-1}z)
        e^{2\pi i\mtr(gZ_\sigma z^*z)}
    \end{equation*}
    for $\Phi_{P_\sigma, Z_\sigma}(z)=P_\sigma(z)
    e^{2\pi i\mtr(Z_\sigma z^*z)}$.
\end{lem}
\begin{proof}
    Since $J(g,Z_\sigma)$ satisfies the cocycle relation,
    it suffices to verify the equation when $(g,h)=(g,1)$,
    where $g\in \UU(\V)(\F_\sigma)$ is as in \eqref{eq:Weil},
    or, by \eqref{def:Weil_ext}, 
    when $(g,h)=(\smat{a\id_2&\\&\id_2}, h)$,
    where $h\in\C^\times$ and $a=h\bar{h}$.

    Applying the formulae \eqref{eq:Weil},
    the case when $g=\smat{A&\\&A^{-*}}$ follows from that
    \[
        \chi(\det A)|\det A|^{1/2}=\chi_\circ(\det A)=\det A^*,
    \]
    while the case when $g=\smat{\id_2&B\\&\id_2}$ follows 
    from $\psi(z)=e^{2\pi iz}$.
    When $g=\smat{&-\id_2\\\id_2&}$, we use
    $\gamma_\sigma=\gamma(\Delta,\psi_\sigma)^{-1}=\sqrt{-1}$
    (cf \cite[Lemma A.1]{Ichino05})
    and the Gauss integral formula 
    \[
        \int_{\C^2}P_\sigma(w)e^{2\pi i(wZ_\sigma w^*)}
        e^{2\pi i \Tr_{\C/\R}(-wz^*)}\,dw
        =\det(\sqrt{-1}Z_\sigma^{-1})P(zZ_\sigma^{-1})
        e^{2\pi i(z(-Z^{-1}_\sigma) z^*)}.
    \]
    Thus $\omega^\square(g,1)\Phi_{P_\sigma, Z_\sigma}=
    \det(Z_\sigma^{-1})P(zZ_\sigma^{-1})e^{2\pi i(z(-Z^{-1}_\sigma)z^*)}$
    as desired. 

    At last, when $(g,h)=(\smat{a\id_2&\\&\id_2}, h)$ with 
    $\nu(g)=a=h\bar{h}$, by definition \eqref{def:Weil_ext} we have
    \[
        \omega^\square(g,h)\Phi_{P_\sigma, Z_\sigma}(z)=
        a^{-1}\chi_\circ(a)^2
        P_\sigma(h^{-1}za)e^{2\pi i (z aZ_\sigma z^*)}=
        \alpha^{k+1}P_\sigma(h^{-1}z)e^{2\pi i (z aZ_\sigma z^*)}.
    \]
    On the other hand we have $gZ_\sigma=aZ_\sigma$
    and $J(g,Z_\sigma)=(\id_2,\id_2)$.
    Therefore $\nu(g)^{k+1}\cdot \rho(J(g,Z_\sigma))P_\sigma(h^{-1}z)
    =\alpha^{k+1}P_\sigma(h^{-1}z)$ as well.
\end{proof}

\begin{defn}\label{def:point}
    We define $\ii_\sigma=i\id_2$ and 
    $Z_{\delta,\sigma}=-\sigma(2\delta)^{-1}$ in $\X_\sigma$.
    Then for $(g_1,g_2)\in G(\UU(\W)\times\UU(-\W))(\F_\sigma)$,
    formula \eqref{eq:iota} implies that $g=\iota(g_1,g_2)$
    fixes $Z_{\delta,\sigma}$ and satisfies
    \begin{equation}\label{eq:cocycle}
        J(\iota(g_1,g_2),Z_{\delta,\sigma})
        =(\bar{g}_1,g_2)=(g_1^{-\intercal},g_2).
    \end{equation}
\end{defn}

Now, suppose $\beta_1,\beta_2$ and $h\in\K_\sigma^1$,
by Lemma \ref{lem:arch} and \eqref{eq:cocycle} we have
\[
    \omega^\square(\iota(\id_2,\smat{\beta_1&\\&\beta_2}),h)
    \Phi_{P_\sigma,Z_{\delta,\sigma}}=
    \rho(\id_2,\smat{\beta_1&\\&\beta_2})\cdot h^{-k}\cdot
    \omega^\square(g_\sigma)\Phi_{P_\sigma,Z_{\delta,\sigma}}=
    (\beta_1\beta_2)^{-1}h^{-k}\cdot
    \omega^\square(g_\sigma)\Phi_{P_\sigma,Z_{\delta,\sigma}}
\]
Therefore, write $-\sigma(2\delta)^{-1}=a^2i$, the integrand in 
$Z(\eta,\Phi_{P_\sigma,Z_{\delta,\sigma}})$
is the constant
\begin{multline*}
    (\Phi_{P_\sigma,Z_{\delta,\sigma}},
    \Phi_{P_\sigma,Z_{\delta,\sigma}})^\flat=
    (-1)^k
    \int_{\C}(z_1\bar{z}_1)^{k}e^{-4\pi a^2z_1\bar{z}_1}\,dz_1
    \int_{\C}e^{-4\pi a^2z_2\bar{z}_2}\,dz_2
    =\\4^{-1}(-4\pi)^{-k}a^{-4-2k}\Gamma(k+1)=
    \frac{\text{Im}(-\sigma(\delta))^{k+2}}{(2\pi)^k}\Gamma(k+1)
\end{multline*}
when $P_\sigma(z)=z_1^{k}$.
This proves Proposition \ref{prop:Rallis_local} in the archimedean case.


\subsubsection{The non-split case}\label{sec:non-split}

Suppse $v$ is non-split and $w$ is the unique place above $v$.
Since $\eta$ is assumed to be trivial on $\K_v^1$, and
$\Phi_v(z)=\id_{(\varpi^m_w)}(z_1)\id_{(\varpi^m_w)}(z_2)$
is invariant under translation by $h\in \K_v^1\subset \K_v^\times$,
we have
\begin{align}\label{eq:reduction}
    Z(\eta,\Phi_v)&=\vol(\K_v^1)\cdot
    \prod_{i=1}^2\int_{\K_v^1}(\omega^\square_1(\iota(1,\beta_i),1)
    \id_{(\varpi_v^m)},\id_{(\varpi^{m}_v)})^\flat \chi^{-1}(\beta_i)\,d\beta_i
    \notag\\
    &=\vol(\K_v^1)\cdot
    \prod_{i=1}^2\int_{\K_v^1}(\omega^\square_1(\iota(\beta_i,1),1)
    \id_{(\varpi^{m}_v)},\id_{(\varpi^{m}_v)})^\flat \,d\beta_i.
\end{align}
Here $\omega^\square_1$ is the Weil representation of 
$\UU(1,1)\times \UU(V)$.
We evaluate the integral following a similar computation from  \cite{Hsieh14}.

To simplify the notation let $\qch$ denote
the component of $\qch_{\K/\F}$ at $v$ and
write $\E/\kk$ for $\K_w/\F_v$.
Let $\oo_\E$ and $\oo$ be the rings of integers,
$\varpi_E, \varpi$ be the uniformizers,
$q_E=|\varpi_E|_E^{-1}$ and $q=|\varpi|^{-1}$ be the cardinality 
of the residue fields, and $e=e_v$ be the ramification index.
Pick an element
\[
    \btheta=2a\delta+b,\quad a,b\in \kk
\]
such that $\{1, \btheta\}$ is an $\oo$-basis for $\oo_\E$.
In fact we may assume $b=0$ since 
$b=2^{-1}(\btheta+\bar{\btheta})\in \oo$
and $v\nmid 2$ by our assumption.
Note that the relative ideal of different $\dd_{\E/\kk}$ is generated by 
$\btheta-\bar{\btheta}=4a\delta$.


\begin{lem}\label{lem:xiv}
    For $n\in \Z_{\geq0}$ define $a_n=a\varpi^n$ and
    $\iota_n(\beta,1)={\varsigma^{-1}_n}\iota(g,1)\varsigma_n$, where
    \[
    \varsigma_n\coloneqq
    \smat{(2a_n\delta)^{-1}&\\&-(2a_n\delta)}
    \]
    then 
    $(\omega^\square_1(\iota(\beta,1)\varsigma_n)\varphi,
    \omega^\square_1(\varsigma_n)\varphi')^\flat=
    (\omega^\square_1(\iota_n(\beta,1))\varphi,
    \omega^\square_1(-\id_2)\varphi')^\flat$ 
    for $\varphi(z)\in \bs(\K_v)$.
\end{lem}
\begin{proof}
    We only have to apply \eqref{eq:Weil} and verify that
    $(\omega^\square_1(\varsigma_n)\varphi, 
    \omega^\square_1(\varsigma_n)\varphi')^\flat$ is equal to 
    \[
        \int_{z\in \E}
        \qch(-1)\varphi((2a_n\delta)^{-1}z)
        \varphi'(-(2a_n\delta)^{-1}\bar{z})
        |2a_n\delta|_\E^{-1}\,dz=
        \qch(-1)\int_\E \varphi(z)\varphi'(\bar{z})\,dz
        =(\varphi,\omega^\square_1(-\id_2)\varphi')^\flat.
    \]
\end{proof}

For $n\in \Z_{\geq 0}$ define $d_n\in \Z$ so that 
\[
\val_\E(a_n\dd_\E)=
e\val_\kk(a_n\dd_\kk)+\val_\E(\dd_{\E/\kk})=2d_n+c
\]
for $c\in \{0,1\}$.
Since $\dd_{\K/\F}$ is prime to $2$ by our assumption,
when $\E/\kk$ is ramified we have $c=\val_\E(\dd_{\E/\kk})=1$.
\begin{prop}\label{prop:mn}
    When $n\in \Z_{\geq 0}$ and $d_n$ is as above, let
    $\mm_n(\beta)=(\omega^\square_1(\iota_n(\beta,1))\id_{-d_n}, 
    \omega^\square_1(-\id_2)\id_{-d_n})^\flat$ and 
    $\btheta_n=\varpi^n\btheta=2a_n\delta$, where
    $\id_d$ denotes the characteristic function $\id_{(\varpi_\E^d)}$, then
    \begin{equation}
    \mm_n(\beta)=
    |a_n|^{-1}\cdot
    \begin{cases}
        \chi(-(1+y\btheta_n))q^{-c/2}_\E,  & 
        \beta={1+y\btheta_n}/{1+y\bar{\btheta}_n},\, y\in (\varpi),\\
        \chi(-a_n(y+\btheta_n))|y+\btheta_n|_\E^{1/2}\varepsilon(\frac{1}{2},\qch,\psi)
        q_\E^{-c},  & 
        \beta={y+\btheta_n}/{y+\bar{\btheta}_n},\, y\in \oo.
    \end{cases}
\end{equation}
\end{prop}
\begin{proof}
    A direct computation shows that
    \begin{equation}\label{eq:iotan}
        \iota_n(\beta,1)=
        \begin{cases}
            \smat{1& a_n(\Tr(\btheta_n)+y\Nr(\btheta_n))\\&1}
            \smat{1+y\btheta_n&\\&(1+y\bar{\btheta}_n)^{-1}}
            \smat{1&\\-y/a_n&1},
             & 
            \beta={1+y\btheta_n}/{1+y\bar{\btheta}_n};\\
            \smat{1& -a_ny\\&1}
            \smat{y+\btheta_n&\\&(y+\bar{\btheta}_n)^{-1}}
            \smat{& a_n\\ -1/a_n&}
            \smat{1& -a_ny\\&1},
            & 
            \beta={y+\btheta_n}/{y+\bar{\btheta}_n}.
        \end{cases}
    \end{equation}
    Since $\gamma_v=\gamma(\Delta,\psi)^{-1}
    =\varepsilon(\frac{1}{2},\qch,\psi)^{-1}\qch(-1)$ by \cite{Kahn1987},
    \begin{multline*}
        \omega^\square(\smat{&a_n\\-1/a_n&})
        \id_{-d_n}(z)=
        \omega^\square(\smat{-a_n&\\&-a_n^{-1}}\smat{&-1\\1&})
        \id_{-d_n}(z)\\=
        \chi(-a_n)|a_n|_\E^{1/2}\cdot \gamma_v^{-1} |\varpi_\E^{-d_n}|_\E|\dd_\E|_\E^{1/2} \cdot 
        \id_{(\varpi_\E^{d_n})\dd_\E^{-1}}(a_nz)=
        \varepsilon(\frac{1}{2},\qch,\psi)\chi(a_n)
        \id_{-d_n-c}(z)q_E^{-c/2}.
    \end{multline*} 
    Since $\val_\E(a_n\dd_\E)=e\val_\kk(a_n\dd_\kk)+\val_\E(\dd_{\E/\kk})$,
    it can be checked directly that
    \[
        \psi(a_n y \Nr(z))\id_{-d_n-c}(z)=\id_{-d_n-c}(z)
        \text{ when }y\in (\varpi),\quad
        \psi(a_n y \Nr(z))\id_{-d_n}(z)=\id_{-d_n}(z)
        \text{ when }y\in \oo.
    \]
    Therefore
    \begin{align}
        &\omega^\square(\smat{1&\\y/a_n &1})\id_{-d_n}=
        \omega^\square(\smat{&-a_n\\1/a_n&}\smat{1&-a_ny\\ &1}
        \smat{&a_n\\-1/a_n&})\id_{-d_n}=\id_{-d_n} 
        &&\text{ when } y\in (\varpi),\label{eq:ns1}\\
        &\omega^\square(\smat{1&a_n y\\&1})\id_{-d_n}=
        \psi(a_n y \Nr(z))\id_{-d_n}(z)=\id_{-d_n}(z) 
        &&\text{ when } y\in \oo.\label{eq:ns2}
    \end{align}
    Now, when $\beta=1+y\btheta_n/1+y\bar{\btheta}_n$ with $y\in(\varpi)$,
    observe that $1+y\btheta_n\in \oo_E^\times$ and 
    $\Tr(\btheta_n)+y\Nr(\theta_n)\in \oo$.
    Therefore \eqref{eq:ns1} together with $\eqref{eq:iotan}$ show that
    $\omega^\square(\iota_n(\beta,1))\id_{-d_n}=
    \chi(1+y\btheta_n)\id_{-d_n}$ and
    \[
    \mm_n(\beta)=\chi(-(1+y\btheta_n))\vol((\varpi_\E^{-d_n})).
    \]
    On the other hand, when
    $\beta=y+\btheta_n/y+\bar{\btheta}_n$ with $y\in \oo$ we have
    \begin{gather*}
        \omega^\square(\iota_n(\beta,1))\id_{-d_n}(z)=\psi(-a_n y \Nr(z))
        \chi(y+\btheta_n)|y+\btheta_n|_\E^{1/2}\cdot 
        \varepsilon(\frac{1}{2},\qch,\psi)\chi(a_n)
        \id_{-d_n-c}(z)q_E^{-c/2},\\
        \begin{split}
            \mm_n(\beta)=\chi(-(y+\btheta_n))|y+\btheta_n|_\E^{1/2}
            \varepsilon(\frac{1}{2},\qch,\psi)\chi(a_n)q_\E^{-c/2}
            \int_\E
            \psi(-a_ny\Nr(z))\id_{-d_n-c}((y+\btheta_n)z)\id_{-d_n}(\bar{z})\,dz\\
            =\chi(-a_n(y+\btheta_n))|y+\btheta_n|_\E^{1/2}
            \varepsilon(\frac{1}{2},\qch,\psi)q_\E^{-c/2}
            \vol((\varpi_\E^{-d_n})).
        \end{split}
    \end{gather*}
    The proposition now follows from
    $\vol((\varpi_E^{-d_n}))=|a_n|_\E^{-1/2}q_\E^{-c/2}=|a_n|^{-1}q_\E^{-c/2}$.
\end{proof}

\begin{lem}\label{lem:measure}
    Let $\phi(\beta)$ be an integrable function on $E^1$, then
    \begin{equation}
        \int_{E^1}\phi(\beta)\,d\beta=
        q^{-n}|\dd_{\E/\kk}|^{1/2}_\E L(1,\qch)
        \int_{\kk}\phi(\frac{1+y\btheta_n}{1+y\bar{\btheta}_n})
        \frac{dy}{|1+y\btheta_n|_\E}.
    \end{equation}
\end{lem}
\begin{proof}
    Recall that the measure on $\E^1$ is defined so that
    $\int_{E^1}\phi(\beta)\,d\beta=\int_{E^\times}\tilde{\phi}(z)\,d^\times z$
    for any function $\tilde{\phi}(z)$ on $\E^\times$ such that
    $\phi(\frac{x+y\btheta_n}{x+y\bar{\btheta}_n})=
    \int_{\kk^\times}\tilde{\phi}(t(x+y\btheta_n))\,d^\times t$.
    Pick such a function and note that when $z=x+y\btheta_n$
    the self-dual measure $dz$ is given by
    $|4a_n\delta|^{1/2}_\E dxdy=q^{-n}|\dd_{\E/\kk}|^{1/2}_Edxdy$,
    where $dx, dy$ are self-dual. Therefore
    \begin{multline*}
        \int_{E^1}\phi(\beta)\,d\beta=
        \int_{\E^\times}\tilde{\phi}(z)\,d^\times z
        =\frac{q^{-n}|\dd_{\E/\kk}|_E^{1/2}}{1-q_\E^{-1}}
        \int_{\kk^{\times2}}\tilde{\phi}(x+y\btheta_n)
        \frac{dx\,dy}{|x+y\btheta_n|_\E}
        \overset{y\mapsto xy}{=}\\
        \frac{q^{-n}|\dd_{\E/\kk}|_E^{1/2}}{1-q_\E^{-1}}
        \int_{\kk^{\times2}}\tilde{\phi}(x(1+y\btheta_n))\frac{dx}{|x|}
        \frac{dy}{|1+y\btheta_n|_\E}=
        q^{-n}|\dd_{\E/\kk}|_E^{1/2}L(1,\qch)
        \int_{\kk^\times}\phi(\frac{1+y\btheta_n}{1+y\bar{\btheta}_n})
        \frac{dy}{|1+y\btheta_n|_\E}.
    \end{multline*}
\end{proof}

\begin{prop}\label{prop:mn_int}
    Define
    $c(\chi)\coloneqq\inf\{n\in \Z_{\geq 1}\mid \chi=1 \text{ on }
    1+\varpi^n\oo_\E\}$.
    When $n\geq c(\chi)-1$, we have
    \begin{equation}
        \int_{E^1}\mm_n(\beta)\,d\beta=
        \qch(-1)|4\delta|_\E^{1/2}\cdot q_E^{-c}\vol(\K_v^1)
    \end{equation}
\end{prop}
\begin{proof}
Combine the previous lemma with Proposition \ref{prop:mn},
\begin{align}
    \int_{E^1}\mm_n(\beta)\,d\beta=
    q^{-n}|\dd_{\E/\kk}|^{1/2}_\E &L(1,\qch)
    \left(
    \int_{(\varpi)}\mm_n(\frac{1+y\btheta_n}{1+y\bar{\btheta}_n})\,dy+
    \int_{\oo}\mm_n(\frac{y+\btheta_n}{y+\bar{\btheta}_n})
    \frac{dy}{|y+\btheta_n|_\E}
    \right)\notag\\=
    |\dd_{\E/\kk}|^{1/2}_\E L(1,\qch)\qch(-1)|a|^{-1}
    &\left(
    q_\E^{-c/2}\int_{(\varpi)}\chi(1+y\btheta_n)\,dy+
    \varepsilon(\frac{1}{2},\qch,\psi)
    \chi(a_n)q_\E^{-c}\int_{\oo}\chi(y+\btheta_n)
    \frac{dy}{|y+\btheta_n|^{1/2}_\E}
    \right)\notag\\
    =|\dd_{\E/\kk}|^{1/2}_\E L(1,\qch)\qch(-1)|a|^{-1}
    &\notag\\\cdot 
    \Bigg(
    q^nq_\E^{-c/2}&\int_{(\varpi)^{n+1}}\chi(1+y\btheta)\,dy
    +\varepsilon(\frac{1}{2},\qch,\psi)
    \chi(a)q_\E^{-c}\int_{(\varpi^{-n})}\chi(y+\btheta)
    \frac{dy}{|y+\btheta|^{1/2}_\E}\Bigg)\label{eq:m8}.
\end{align}
By the definition of $c(\chi)$, when $n\geq c(\chi)-1$
the first integral in \eqref{eq:m8} is equal to 
\[
    q^nq_\E^{-c/2}\int_{(\varpi)^{n+1}}\chi(1+y\btheta)\,dy
    =q^{-1}q_\E^{-c/2}|\dd_\kk|^{1/2}.
\]
On the other hand, by \cite[Proposition 8.2.]{hks}
\[
    \int_\kk\chi(y+\btheta)|y+\btheta|_\E^{-s-\frac{1}{2}}\,dy
    =\chi(4a\delta)|4a\delta|_\E^{-s}
    \frac{\varepsilon(s+\frac{1}{2},\chi^{-1}, \psi_\E)}
    {\varepsilon(2s,\qch,\psi)}
    \frac{L(2s,\qch)}{L(1-2s,\qch)}.
\]
Therefore the second integral in \eqref{eq:m8} is equal to 
the value at $s=0$ of
\[
    \chi(a\delta)|4a\delta|_\E^{-s}
    \frac{\varepsilon(s+\frac{1}{2},\chi^{-1}, \psi_\E)}
    {\varepsilon(2s,\qch,\psi)}
    \frac{L(2s,\qch)}{L(1-2s,\qch)}-
    \int_{\kk\setminus \varpi^{-n}\oo}
    \chi(y+\btheta)
    |y+\btheta|_\E^{-s-\frac{1}{2}}\,dy.
\]
But $\chi(y+\btheta)=\chi(y)=\qch(y)$ 
and $|y+\btheta|_\E=|y|_\E=|y|^2$
when $\val_\kk(y)\leq -c(\chi)$.
Therefore
\[
    \int_{\kk\setminus \varpi^{-n}\oo}
    \chi(y+\btheta)
    |y+\btheta|_\E^{-s-\frac{1}{2}}\,dy=
    \begin{cases}
        -(-1)^{n}q^{-2s(1+n)}L(2s,\qch)(1-q^{-1})|\dd_\kk|^{1/2},
        & \E/\kk \text{ unramified};\\
        0,
        & \E/\kk \text{ ramified}.
    \end{cases}
\] 
When $\E/\kk$ is ramified observe that $q_\E=q, c=1$ and 
$\varepsilon(\frac{1}{2},\qch,\psi)=
|\dd_{E/\kk}|_E^{1/2}|\dd_\kk|^{1/2}\varepsilon(0,\qch,\psi)$.
Thus \eqref{eq:m8} is
\[
    q^{-3/2}|\dd_\kk|^{1/2}
    +
    q^{-1}\chi(\delta)\varepsilon(\frac{1}{2},\chi^{-1}, \psi_\E)
    |\dd_{\E/\kk}|_E^{1/2}|\dd_\kk|^{1/2}
    =2q^{-3/2}|\dd_\kk|^{1/2}
    =2q^{-1}|\dd_{\E/\kk}|_E^{1/2}|\dd_\kk|^{1/2}
\]
by the assumption \ref{delta_cond1}.
And when $\E/\kk$ is unramified observe that
$q_\E=q^2$,
$\varepsilon(\frac{1}{2},\qch,\psi)=|\dd_\kk|^{1/2}
\varepsilon(0,\qch,\psi)=\qch(\dd_\kk)$, and
$\qch(a_n\dd_\kk)=(-1)^{2d_n+c}=(-1)^c$.
Thus by \ref{delta_cond1} again \eqref{eq:m8} is
\begin{equation*}
    q^{-1-c}|\dd_\kk|^{1/2}+
    q^{-2c}|\dd_\kk|^{1/2}
    \left[\frac{1+q^{-1}}{2}+\qch(a_n\dd_\kk)
    \frac{1-q^{-1}}{2}\right]
    =(1+q^{-1})|\dd_\kk|^{1/2}q^{-2c}
\end{equation*}
The result now follows from 
$|\dd_{\E/\kk}|_\E^{1/2}=|4\delta a|_\E^{1/2}=|4\delta|_\E^{1/2}|a|$
and $\vol(\K_v^1)=e|\dd_{\E/\kk}|_E^{1/2}|\dd_\kk|^{1/2}$.
\end{proof}

We now set $\Phi_v(z)=\id_{-d}(z_1)\id_{-d}(z_2)$
for $d$ given as follows.
\begin{itemize}
\item When $\K_w/\F_v$ is unramified, 
pick $n=c(\chi)-1$ or $c(\chi)$ 
so that $2d\coloneqq \val_\E(a_n\dd_\E)$ is even. 
\item When $\K_w/\F_v$ is ramified.,
pick $n=c(\chi)-1$ and set $d$ such that $\val_\E(a_n\dd_\E)=2d+1$.
\end{itemize}
Note that since $n=0$ for almost all $n$ and $\dd_{\E/\kk}=(4a\delta)$,
indeed we have $d=\val_\E(a\dd_\E)=0$ or almost all $n$.
The following proposition follows from Lemma \ref{lem:xiv},
Proposition \ref{prop:mn_int}, and the reduction \eqref{eq:reduction}.
\begin{prop}\label{prop:nonsplit}
    Set $\xi_v=\diag(\varsigma_n,\varsigma_n)\in \UU(\V)(\F_v)$.
    We have 
    \begin{equation}\label{eq:reduction_twist}
    Z(\eta,\omega^\square(\xi_v)\Phi_v)=
    \vol(\K_v^1)\cdot 
    \left[
        \int_{\K_v^1}\mm_n(\beta)\,d\beta
    \right]^2
    =\vol(\K_v^1)^3\cdot |4\delta|_w\cdot q_w^{-2(e_v-1)}
    \end{equation}
\end{prop}

\begin{prop}\label{prop:non_split_inv}
    For almost all $v\in\finite$ that is non-split we have
    $\omega^\square(\iota(\bbeta,1))\Phi_v=\Phi_v$ for all $\bbeta\in T(\F_v)$.
\end{prop}
\begin{proof}
    The computation in Proposition \ref{prop:mn} shows that
    $\omega^\square_1(\iota_n(\beta_i,1))\id_{-d_n}(z_i)=\id_{-d_n}(z_i)$
    for almost all $v$.
    The result then follows since $n=0$, and hence $\xi_v=\id$,
    for almost all $v$.
\end{proof}

\subsubsection{The split case}

Suppose $v$ splits into $v=w\bw$ in $\K$.
Since $\delta=-\delta^c$, we have
$\delta_v\coloneqq\iota_w(\delta)=-\iota_\bw(\delta)$.
We then identify the image of $\delta$ in $\K_v$ 
with $(\delta_v,-\delta_v)$ and 
define $\xi_v=(\xi_w,\xi_\bw)\in \UU(\V)(\F_v)$ by
\begin{equation}\label{def:xi}
    \xi_w=
    \smat{(2\delta_v)^{-1}\id_2& -(2\delta_v)^{-1}\id_2\\\id_2 & \id_2},\quad
    \xi_\bw=
    \smat{2^{-1}\id_2& -2^{-1}\id_2\\\delta_v\id_2 & \delta_v\id_2}.
\end{equation}
\begin{prop}\label{prop:xi_split}
    For $g_1,g_2\in \UU(\W)(\F_v)$ and
    $\Phi,\Phi'\in \bs(\K_v^2)$
    \[
        (\omega^\square(\iota(g_1,g_2)\xi_v)\Phi,
        \omega^\square(\xi)\Phi')^\flat=
        (\omega^\square(\smat{(g_{1,w}, g_{2,\bw})&\\&(g_{2,w}, g_{1,\bw})})\Phi,
        \omega^\square((\delta_v^{-1},\delta_v)\smat{&-\id_2\\\id_2&})\Phi')^\flat.   
    \]
\end{prop}
\begin{proof}
    Write $\xi=\xi_v$.
    We have $\iota_{\xi}(g_1,g_2)\coloneqq \xi^{-1}\iota(g_1,g_2)\xi=
    \smat{(g_{1,w}, g_{2,\bw})&\\&(g_{2,w}, g_{1,\bw})}$
    since by definition \eqref{eq:iota}
    \begin{equation*}
    \iota(g_1,g_2)_w=
    \smat{(2\delta_v)^{-1}\id_2&-(2\delta_v)^{-1}\id_2\\\id_2&\id_2}
    \smat{g_{1,w}&\\&g_{2,w}}
    \smat{(2\delta_v)^{-1}\id_2&-(2\delta_v)^{-1}\id_2\\\id_2&\id_2}^{-1}
    \end{equation*}
    Then by Lemma \ref{lem:J}
    $(\omega^\square(\iota(g_1,g_2)\xi)\Phi,
    \omega^\square(\xi)\Phi')^\flat=
    (\omega^\square(\iota_{\xi}(g_1,g_2))\Phi,
    \omega^\square((J\bar{\xi}J)^{-1}\xi)\Phi')^\flat$
    and the proposition follows from
    $(J\bar{\xi}J)^{-1}\xi=(\delta_v^{-1},\delta_v)\smat{&-\id_2\\\id_2&}$.
    Here $(\delta_v^{-1},\delta_v)\in\K_v^1$ is such that 
    $\iota_w(\delta_v^{-1},\delta_v)=\delta_v^{-1}$.
\end{proof}

With the help of the proposition, we compute the action of 
Hecke operators on $\omega^\square(\xi_v)\Phi_v$ when $v\nmid \fc$.
Identify $\UU(\W)(\F_v)$ with $\GL_2(\F_v)$  via the isomorphism $\iota_w$.
We consider the double coset operators given by
\begin{align}
    T_w&=\GL_2(\oo_v)\smat{\varpi_v&\\&1}\GL_2(\oo_v)=
    \bigcup_{b\in \oo_v/(\varpi_v)}\smat{\varpi_v&b\\&1}\GL_2(\oo_v)
    \cup  \smat{1&\\&\varpi_v}\GL_2(\oo_v)\\
    U_w&=I_v^n\smat{\varpi_v&\\&1}I_v^n=
    \bigcup_{b\in \oo_v/(\varpi_v)}\smat{\varpi_v&b\\&1}I_v^n
\end{align}
which act respectively on the subspaces of Schwartz functions
that are invariant by $\GL_2(\oo_v)$ or 
$I_v^n\coloneqq
\{g\in\GL_2(\oo_v)\mid g\equiv\smat{1&*\\&*}\bmod\varpi_v^n\}$.

\begin{prop}\label{prop:inv_Up}
\hfill
\begin{enumerate}
    \item When $v\nmid p\fc\fs$,
    $\omega^\square(\xi_v)\Phi_v$ is invariant by 
    $\omega^\square(\iota(g,\id))$ for $g\in \GL_2(\oo_v)$ and 
    \begin{multline*}
    T_w\omega^\square(\xi_v)\Phi_v\coloneqq
    \sum_{b\in \oo_v/(\varpi_v)}\omega^\square(\iota
    (\smat{\varpi_v&b\\&1},\id)\xi_v)\Phi_v+
    \omega^\square(\iota(\smat{1&\\&\varpi_v},\id)\xi_v)\Phi_v\\=
    \chi_\circ^{-1}(\varpi_w)\cdot
    \omega^\square(\iota(\smat{\varpi_v&\\&\varpi_v},\id_2)\xi_v)\Phi_v+
    q_w\chi_\circ(\varpi_w)\Phi_v.
    \end{multline*}
    \item 
    When $v\mid p\fs$,
    $\omega^\square(\xi_v)\Phi_v$ is invariant by 
    $\omega^\square(\iota(g,\id))$ for $g\in I_v^n$,
    where $\tilde{\eta}_w\vert_{1+\varpi_v^n\oo_v}\equiv 1$, and 
    \begin{equation*}
    U_w\omega^\square(\xi_v)\Phi_v\coloneqq
    \sum_{b\in \oo_v/(\varpi_v)}\omega^\square(\iota
    (\smat{\varpi_v&b\\&1},\id_2)\xi_v)\Phi_v=
    \chi_\circ^{-1}(\varpi_w)\cdot
    \omega^\square(\iota(\smat{\varpi_v&\\&\varpi_v},\id_2)\xi_v)\Phi_v.
    \end{equation*}
\end{enumerate}
\end{prop}
\begin{proof}
    Since $\iota_{\xi}(g,1)$ is of the form $\smat{A&\\&A}$
    it is easy to verify the invariance.
    Now observe that 
    \begin{align*}
        \iota_{\xi}(\smat{\varpi_v&b\\&1},\id_2)=\smat{A&\\&A^{-*}}
        \text{ for } A=(\smat{\varpi_w&b\\&1},\id_2)
        \in \GL_2(\K_v)=\GL_2(\F_v\times\F_v)\\
        \iota_{\xi}(\smat{1&\\&\varpi_v},\id_2)=\smat{A&\\&A^{-*}}
        \text{ for } A=(\smat{1&\\&\varpi_v},\id_2)
        \in \GL_2(\K_v)=\GL_2(\F_v\times\F_v).
    \end{align*}
    Thus when $v\nmid p\fc\fs$, we have that 
    $\omega^\square(\iota_\xi(\smat{1&\\&\varpi_v},1))\Phi_v(z)=
    \chi_\circ(\varpi_w)
    \id_{\oo_v}(a_1)
    \id_{\oo_v}(b_1)
    \id_{\oo_v}(\varpi_va_2)
    \id_{\oo_v}(b_2)$ and
    \begin{multline*}
        \sum_{b\in \oo_v/(\varpi_v)}\omega^\square(\iota_{\xi}
        (\smat{\varpi_v&b\\&1},\id_2),1)\Phi_v=
        \chi_\circ(\varpi_w)\sum_{b\in \oo_v/(\varpi_v)}
        \id_{\oo_v}(\varpi_v a_1)
        \id_{\oo_v}(b_1)
        \id_{\oo_v}(ba_1+a_2)
        \id_{\oo_v}(b_2)
        \\=\chi_\circ(\varpi_w)
        \id_{\oo_v^\times}(\varpi_v a_1)
        \id_{\oo_v}(b_1)
        \id_{\oo_v}(\varpi_va_2)
        \id_{\oo_v}(b_2)+
        q_v\chi_\circ(\varpi_w)
        \id_{\oo_v}(a_1)
        \id_{\oo_v}(b_1)
        \id_{\oo_v}(a_2)
        \id_{\oo_v}(b_2)
    \end{multline*}
    Similarly, when $v\mid p\fs$ we have that 
    \begin{multline*}
        \sum_{b\in \oo_v/(\varpi_v)}\omega^\square(\iota_{\xi}
        (\smat{\varpi_v&b\\&1},\id_2),1)\Phi_v=
        \chi_\circ(\varpi_w)\sum_{b\in \oo_v/(\varpi_v)}
        \id_{\oo_v^\times}(\varpi_v a_1)
        \id_{\oo_v}(b_1)
        \id_{\oo_v}(ba_1+a_2)
        \id_{\oo_v}(b_2)
        \\=\chi_\circ(\varpi_w)
        \id_{\oo_v^\times}(\varpi_v a_1)
        \id_{\oo_v}(b_1)
        \id_{\oo_v}(\varpi_va_2)
        \id_{\oo_v}(b_2).
    \end{multline*}
    Now the proposition follows from
    $\omega^\square(\iota_{\xi}(\smat{\varpi_v&\\&\varpi_v},\id),1)
    \Phi_v(a_1,b_1,a_2,b_2)=\Phi_v(\varpi_va_1,b_1,\varpi_va_2,b_2)$.
\end{proof}

Back to the computation of \eqref{eq:Rallis_local}.
Observe that 
for 
$\Phi_v(z)=
\phi_{1,w}(a_1)\phi_{1,\bw}(b_1)
\phi_{2,w}(a_2)\phi_{2,\bw}(b_2)$
\[
    \omega^\square((\delta_v^{-1},\delta_v)\smat{&-\id_2\\\id_2&})\Phi_v(z)=
    \chi^{-4}_w(\delta_v)
    \hat{\phi}_{1,w}(-b_1\delta_v)\hat{\phi}_{1,\bw}(-a_1\delta_v^{-1})
    \hat{\phi}_{2,w}(-b_2\delta_v)\hat{\phi}_{2,\bw}(-a_2\delta_v^{-1})
\]
where $\hat{\phi}(x)=\int_{\F_v}\phi(y)\psi(xy)\,dy$ denotes 
the Fourier transform of a function $\phi\in\bs(\F_v)$.

Write $h$=$(x,x^{-1})$ and $\beta_i=(y_i,y_i^{-1})$ for $i=1,2$
again with $x=\iota_w(h)=\iota_\bw(h)^{-1}$
and $y_i=\iota_w(\beta_i)=\iota_\bw(\beta_i)^{-1}$
and apply Proposition \ref{prop:xi_split} gives
\begin{align}
    Z(\eta_v, &\omega^\square(\xi_v)\Phi_v)=
    \chi_w^{-4}(\delta_{v})\notag\\\cdot
    &\int_{\F_v^{\times}}\int_{\F_v^2}
    \phi_{1w}(x^{-1}a_1)\hat{\phi}_{1w}(a_1\delta_v)
    \phi_{2w}(x^{-1}a_2)\hat{\phi}_{2w}(a_2\delta_v)
    \tilde{\eta}_w(x)\,d^\times x\,da_1da_2\notag\\\cdot 
    &\int_{\F_v^{\times}}\int_{\F_v}
    \phi_{1\bw}(xb_1y_1^{-1})\hat{\phi}_{1\bw}(b_1\delta_v^{-1})
    \chi^{-1}_w(y_1)|y_1|_v^{-1/2}\,d^\times y_1\,db_1 \notag\\\cdot
    &\int_{\F_v^{\times}}\int_{\F_v}
    \phi_{2\bw}(xb_2y_2^{-1})\hat{\phi}_{2\bw}(b_2\delta_{v}^{-1})
    \chi^{-1}_w(y_2)|y_2|_v^{-1/2}\,d^\times y_2\,db_2 \notag\\
    &\overset{y_i\mapsto xy_i,\, a_i\mapsto \delta_v^{-1}a_i,\,
    b_i\mapsto \delta_v b_i}{=}
    \chi_w^{-4}(\delta_{v})\notag\\\cdot
    &\int_{\F_v^{\times}}\int_{\F_v^2}
    \phi_{1w}(x^{-1}a_1\delta_{v}^{-1})\hat{\phi}_{1w}(a_1)
    \phi_{2w}(x^{-1}a_2\delta_{v}^{-1})\hat{\phi}_{2w}(a_2)
    (\chi^{-2}\tilde{\eta})_w(x)|x|_v^{-1}\,d^\times x\,da_1da_2\label{eq:m3}\\\cdot 
    &\int_{\F_v^{\times}}\int_{\F_v}
    \phi_{1\bw}(b_1\delta_{v}y_1^{-1})\hat{\phi}_{1\bw}(b_1)
    \chi^{-1}_w(y_1)|y_1|_v^{-1/2}\,d^\times y_1\,db_1 \label{eq:m4}\\\cdot
    &\int_{\F_v^{\times}}\int_{\F_v}
    \phi_{2\bw}(b_2\delta_{v}y_2^{-1})\hat{\phi}_{2\bw}(b_2)
    \chi^{-1}_w(y_2)|y_2|_v^{-1/2}\,d^\times y_2\,db_2 \label{eq:m5}
\end{align}
Making the change of variable 
$y_i\mapsto b_i\delta_v y_i$, lines \eqref{eq:m4} and \eqref{eq:m5} become
\begin{multline*}
    \chi^{-1}_w(\delta_v)|\delta_v|^{-1/2}(1-q_v^{-1})\cdot 
    \int_{\F_v^\times}
    \phi_{i,\bw}(y_i^{-1})\chi^{-1}_w(y_i)|y_i|^{-1/2}\,d^\times y_i
    \int_{\F_v^\times}
    \hat{\phi}_{i,\bw}(b_i)\chi^{-1}_w(b_i)|b_i|^{1/2}\,d^\times b_i\\=
    \chi^{-1}_w(\delta_v)|\delta_v|^{-1/2}(1-q_v^{-1})\cdot 
    Z(\frac{1}{2}, \phi_{i,\bw}, \chi_w)
    Z(\frac{1}{2}, \hat{\phi}_{i,\bw}, \chi^{-1}_w)
\end{multline*}
where $Z(s, \phi_{i,\bw}, \chi_w)$ is the Tate integral.
Since  $Z(\frac{1}{2}, \phi_{i,\bw}, \chi_w)=
\vol(\oo_v^\times)L(\frac{1}{2}, \chi_w)$
for our choice of $\phi_{i,\bw}$,
we can conclude from the local functional equation that 
\begin{equation*}
    Z(\frac{1}{2}, \phi_{i,\bw}, \chi_w)
    Z(\frac{1}{2}, \hat{\phi}_{i,\bw}, \chi^{-1}_w)=
    \vol(\oo_v^\times)^2\cdot 
    \varepsilon(\frac{1}{2}, \chi_w,\psi_w)L(\frac{1}{2}, \chi_v).
\end{equation*}
For line \eqref{eq:m3}
recall that $\phi_{2,w}$ is equal to either 
$\id_{\oo_v}$ or $\chi_w^{-1}\id_{\oo_v^\times}$,
therefore
$\hat{\phi}_{2,w}(a_2)$ is equal to either $|\dd_v|_v^{1/2}\id_{\dd_v^{-1}}$,
where $\dd_v$ denotes the component of $\dd_\F$ at $v$,
or
$\chi_w(a_2)\varepsilon(1,\chi_w,\psi_w)
\id_{(\dd_v\fc)^{-1}}^\times(a_2)$,
where $\id^\times_\fa=\id_{\fa}-\id_{\varpi_v\fa}$
for $\fa\subset \F_v$. Then
\begin{align}
    \int_{\F_v}\phi_{2,w}(x^{-1}a_2\delta_v^{-1})\hat{\phi}_{2,w}(a_2)\,da_2=
    \begin{cases}
        1, & x\notin (\delta_v\dd_v)^{-1}\\
        |x\delta_v\dd_v|_v & x\in (\delta_v\dd_v)^{-1}
    \end{cases}\label{eq:split_case1}\quad
    \text{ when } \phi_{2,w}&=\id_{\oo_v}\\
    \int_{\F_v}\phi_{2,w}(x^{-1}a_2\delta_v^{-1})\hat{\phi}_{2,w}(a_2)\,da_2=
    |\dd_v|_v^{1/2}(1-q_v^{-1})\cdot |x\delta_v|\hat{\phi}_{2,w}(x\delta_v)\quad
    \text{ when } \phi_{2,w}&=\chi_w^{-1}\id_{\oo_v^\times},\quad
    \label{eq:split_case2}
\end{align}

Apply \eqref{eq:split_case1} to 
$\phi_{1,w}=\phi_{2,w}=\id_{\oo_v}$ when $w\nmid p\fc\fs$.
We conclude that 
line \eqref{eq:m3} is equal to 
\begin{multline*}
    \int_{x\notin(\delta_v\dd_v)^{-1}}
    (\chi^{-2}\tilde{\eta})_w(x)|x|_v^{-1}\,d^\times x+
    |\delta_v\dd_v|_v^2\cdot \int_{x\in(\delta_v\dd_v)^{-1}}
    (\chi^{-2}\tilde{\eta})_w(x)|x|_v\,d^\times x\\=
    \vol(\oo_v^\times)\cdot 
    (\chi^{-2}\tilde{\eta})_w(\delta_v\dd_v)^{-1}|\delta_v\dd_v|_v
    \frac{L(1,(\chi^{-2}\tilde{\eta})_v)}{L(2,\qch_v)}
\end{multline*}

Apply \eqref{eq:split_case2} to 
$\phi_{1,w}=\phi_{2,w}=\chi^{-1}_w\id_{\oo_v^\times}$
when $w\mid \fc$.
We conclude that 
line \eqref{eq:m3} is equal to 
\begin{multline*}
    |\dd_v|_v(1-q_v^{-1})^2 |\delta_v|^2
    \chi^2_w(\delta_v)\varepsilon(1,\chi_w,\psi_w)^2
    \int_{\val_v(x\delta_v)=-\val_v(\dd_v\fc)}
    \tilde{\eta}_w(x)|x|_v\,d^\times x\\
    =\vol(\oo_v^\times) (1-q_v^{-1})^2
    \tilde{\eta}_w^{-1}(\fc\delta_v\dd_v)
    \cdot \chi^{2}_w(\delta_v)|\delta_v\dd_v|_v\cdot 
    \varepsilon(\frac{1}{2},\chi_w,\psi_w)^2.
\end{multline*}


When $v\mid p\fs$, 
let $\tau_v\in \UU(\W)(\F_v)$ 
be such that $\tau_w\coloneqq\iota_w(\iota_v)=\smat{\varpi_v&\\&1}$, then
\[
    \omega^\square
    (\iota_{\xi}(\tau_v^n, \id_2))\Phi_{v}(z)=
    \chi_w(\varpi_w^n)|\varpi_v|_v^{n/2}
    \phi_{1,w}(a_1\varpi_v^n)\phi_{1,\bw}(b_1)
    \phi_{2,w}(a_2)\phi_{2,\bw}(b_2).
\]
Therefore 
$Z(\eta_v,\bnu,\omega^\square(\iota(\tau_v^n,1)\xi_v)\Phi_v,
\omega^\square(\xi_v)\Phi_v)$ 
is equal to the product of $\chi_w^{-4}(\delta_v)$, \eqref{eq:m4}, \eqref{eq:m5}, and 
\begin{equation}\label{eq:m3'}
    \chi_w(\varpi_w^n)|\varpi_v|_v^{n/2}
    \int_{\F_v^{\times}}\int_{\F_v^2}
    \phi_{1w}(x^{-1}a_1\varpi_v^n\delta_{v}^{-1})\hat{\phi}_{1w}(a_1)
    \phi_{2w}(x^{-1}a_2\delta_{v}^{-1})\hat{\phi}_{2w}(a_2)
    (\chi^{-2}\tilde{\eta})_w(x)|x|_v^{-1}\,d^\times x\,da_1da_2.
\end{equation}
Choose $n\in \Z_{\geq 1}$ so that $\tilde{\eta}_w$ 
is trivial on $1+\varpi_w^n$. Suppose $\tilde{\eta}_w$ is ramified, then
\[
    \int_{\F_v}\phi_{1,w}(x^{-1}a_1\varpi_v^n\delta_v^{-1})
    \hat{\phi}_{1,w}(a_1)\,da_1=
    |\dd_v|_v^{1/2}(1-q_v^{-1})\cdot |x\delta_v\varpi_v^{-n}|_v
    \hat{\phi}_{1,w}(x\delta_v\varpi_v^{-n})
\]
is supported on $(\delta_v\dd_v)^{-1}$.
Combine this with \eqref{eq:split_case1}, we see that
\eqref{eq:m3'} is equal to
\begin{multline*}
    \chi_w(\varpi_w^n)|\varpi_v|_v^{n/2}\cdot 
    |\dd_v|_v^{1/2}(1-q_v^{-1})|\delta_v\varpi_v^{-n}|_v\cdot 
    \int_{\F_v^\times}\hat{\phi}_{1,w}(x\delta_v\varpi_v^{-n})
    (\chi^{-2}\tilde{\eta})_w(x)|x\delta_v\dd_v|_v\,d^\times x
    \\\overset{x\mapsto \varpi_v^n\delta_v^{-1}x}{=}
    (\chi^{-1}|\cdot|_v^{1/2}\tilde{\eta})_w(\varpi^n_w)\cdot 
    |\dd_v|_v^{3/2}(1-q_v^{-1})
    \cdot (\chi\tilde{\eta}^{-1})_w(\delta_v)|\delta_v|_v
    \int_{\F_v^\times}\hat{\phi}_{1,w}(x)
    (\chi^{-2}\tilde{\eta})_w(x)|x|_v\,d^\times x
\end{multline*}
Apply the local functional equation again, from 
$\int_{\F_v^\times}\phi_{1,w}(x)
(\chi\tilde{\eta}^{-1})_w(x)\,d^\times x=\vol(\oo_v^\times)$ we have
\begin{equation}\label{eq:splitlast}
\int_{\F_v^\times}\hat{\phi}_{1,w}(x)
    (\chi^{-2}\tilde{\eta})_w(x)|x|_v\,d^\times x=
    \frac{L(1,(\chi^{-2}\tilde{\eta})_w)
    \varepsilon(0,(\chi^2\tilde{\eta}^{-1})_w,\psi_w)}
    {L(0,(\chi^2\tilde{\eta}^{-1})_w) }
    \vol(\oo_v^\times)
\end{equation}

Suppose $\tilde{\eta}_w$ is unramified,
then $\phi_{1,w}=\id_{\oo_v^\times}$, and $ \hat{\phi}_{1,w}=
|\dd_v|_v^{1/2}(-q_v^{-1}\id_{(\dd_v^{-1}\varpi_v^{-1})}+
)\id_{\dd_v^{-1}})$. And
\[
    \int_{\F_v}\phi_{1,w}(x^{-1}a_1\varpi_v^n\delta_v^{-1})\hat{\phi}_{1,w}(a_1)\,da_1=
    |\dd_v|^{1/2}(1-q_v^{-1})\cdot |x\delta_v\varpi_v^{-n}|_v
    \hat{\phi}_{1,w}(x\delta_v\varpi_v^{-n})
\]
is supported on $x\delta_v\in \dd_v^{-1}\varpi_v^{n-1}\subset \dd_v^{-1}$
since $n\geq 1$. 
Therefore the computation is exactly the same with 
the case when $\tilde{\eta}$ is ramified
and \eqref{eq:splitlast} also holds.
We can now combine all the above computations and obtain the formulae
listed in Proposition \ref{prop:Rallis_local}.

\subsection{Toric period}\label{sec:period}

Identify $(T\times T)(\F_v)=(\K_v^1)^4$ and 
apply Proposition \ref{prop:period} to the case
\[
\bnu=\{\nu_1,\nu_2\}, \nu_1=\nu_2=\chi^{-1} \text{ and }
\bmu=\{\mu_1,\mu_2\}, \mu_1=(\eta\mu)^{-1}, \mu_2=\mu.
\]
Since each local Schwartz function $\Phi_v(z)$
has the form $\varphi_{1,v}(z_1)\varphi_{2,v}(z_2)$
for $\varphi_{i,v}\in \bs(\K_v)$, we put
$P(\eta,\mu, \Phi_v)\coloneqq
P(\bmu,\bnu, \Phi_v, \Phi_v)=P_1(\eta\mu,\varphi_1)P_2(\mu^{-1},\varphi_2)$ for
\begin{align*}
    P(\eta\mu,\varphi_{1,v})&\coloneqq
    \int_{(\K_v^1)^2}
    (\omega_1^\square(\iota(\alpha_1, 
    \beta_1), 1)\varphi_{1,v}, \varphi_{1,v})^\flat
    (\eta\mu)^{-1}(\alpha_1)\chi^{-1}(\beta_1)
    \,d\alpha_1d\beta_1\\
    &=
    \int_{(\K_v^1)^2}
    (\omega_1^\square(\iota(1, 
    \beta_1), h_1)\varphi_{1,v}, \varphi_{1,v})^\flat
    \chi^{-1}(\beta_1)\eta\mu(h_1)
    \,dh_1d\beta_1,\\
    P(\mu^{-1},\varphi_{2,v})&\coloneqq
    \int_{(\K_v^1)^2}
    (\omega_1^\square(\iota(\alpha_2, 
    \beta_2), 2)\varphi_{2,v}, \varphi_{2,v})^\flat
    \mu(\alpha_2)\chi^{-1}(\beta_2)
    \,d\alpha_2d\beta_2\\
    &=
    \int_{(\K_v^1)^2}
    (\omega_1^\square(\iota(1, 
    \beta_2), h_2)\varphi_{2,v}, \varphi_{2,v})^\flat
    \chi^{-1}(\beta_2)\mu^{-1}(h_2)
    \,dh_2d\beta_2.\\
\end{align*}

To simplify the notation, we assume $\mu$ is a finite order
Hecke character which is ramified only at a prime $\fl$ of $\K$
such that $\fl\neq\bar{\fl}$ and 
$\fl$ is prime to $p\fc\fs$.

\begin{prop}\label{prop:period_local}
    Let notations be as in Proposition \ref{prop:Rallis_local}.
    \begin{enumerate}
        \item At an archimedean place $\sigma\in\Sigma$ we have
        \[
            P(\eta,\mu, \omega^\square(g_\sigma)\Phi_{P_\sigma,\ii_\sigma})=
            \vol(\C_1)^4\cdot
            \frac{\textnormal{Im}(-\delta)^{k+2}}{(2\pi)^k}\Gamma(k+1).
        \]
        \item At a non-split place $v$ let $w\mid v$. Then
        \[
            P(\eta, \mu, \omega^\square(\xi_v)\Phi_v)=
           \vol(\K_v^1)^4\cdot c'_v
        \]
        where $c'_v=|4\delta|_wq_w^{-2(e_v-1)}$ is a $p$-unit
        \item At $v=w\bw$ with $v\nmid p\fs\fl\bar{\fl}$,
        \[
            P(\eta, \mu, \omega^\square(\xi_v)\Phi_v)=
            \tilde{\eta}_w^{-1}(\fc\delta\dd_\F)\cdot 
            \vol(\oo_v^\times)^4\cdot c'_v
            \left[
            \frac{L(\frac{1}{2}, \chi_v)}{L(1,\qch)}
            \right]^2
            \frac{L(\frac{1}{2}, (\chi\tilde{\mu})_v)}{L(1,\qch)}
            \frac{L(\frac{1}{2}, (\chi(\tilde{\eta}\tilde{\mu})^{-1})_v)}{L(1,\qch)}.
        \]
        where $c'_v\coloneqq  \chi^{-4}_w(\delta)|\dd_\F|_v^2
        \varepsilon(\frac{1}{2}, \chi_w,\psi_w)^2
        \varepsilon(\frac{1}{2}, (\chi\tilde{\mu})_w,\psi_w)
        \varepsilon(\frac{1}{2}, (\chi\tilde{\mu}^{-1})_w,\psi_w)$ is a $p$-unit.
        \item At $v=w\bw$ with $w=\fl$ replace $\phi_{1,w}=\phi_{2,w}=\id_{\oo_v}$
        by $\id_{\oo_v^\times}$, then 
        \[
            P(\eta, \mu, \omega^\square(\xi_v)\Phi_v)=
            \tilde{\eta}_w^{-1}(\fc\fl^m\delta\dd_\F)\cdot 
            \vol(\oo_v^\times)^4\cdot c'_v
            \left[
            \frac{L(\frac{1}{2}, \chi_v)}{L(1,\qch)}
            \right]^2
            \frac{L(\frac{1}{2}, (\chi\tilde{\mu})_v)}{L(1,\qch)}
            \frac{L(\frac{1}{2}, (\chi(\tilde{\eta}\tilde{\mu})^{-1})_v)}{L(1,\qch)}.
        \]
        where $c'_v\coloneqq  \chi^{-4}_w(\delta)|\dd_\F|_v^2
        \varepsilon(\frac{1}{2}, \chi_w,\psi_w)^2
        \varepsilon(\frac{1}{2}, (\chi\tilde{\mu})_w,\psi_w)
        \varepsilon(\frac{1}{2}, (\chi\tilde{\mu}^{-1})_w,\psi_w)$ is a $p$-unit and $\fl^m$ is the conductor of $\tilde{\mu}_\fl$.
        \item At $v=w\bw$ with $v\mid p\fs$, then
        \[
            P(\eta, \omega^\square(\xi_v)\Phi_v)=
            \tilde{\eta}_w^{-1}(\delta)\cdot 
            \vol(\oo_v^\times)^4\cdot c'_v
            \left[
            \frac{L(\frac{1}{2}, \chi_v)}{L(1,\qch)}
            \right]^2
            \frac{L(\frac{1}{2}, (\chi\tilde{\mu})_v)}{L(1,\qch)}
            \frac
            {\varepsilon(\frac{1}{2}, \chi(\tilde{\eta}\tilde{\mu})^{-1})_w,\psi_w)}
            {L(1,\qch)}
            \frac {L(1/2,(\chi^{-1}\tilde{\eta})_w)}
            {L(1/2, (\chi\tilde{\eta}^{-1})_w)}.
        \]
        where $c'_v\coloneqq  \chi^{-4}_w(\delta)|\dd_\F|_v^2
        \varepsilon(\frac{1}{2}, \chi_w,\psi_w)^2
        \varepsilon(\frac{1}{2}, (\chi\tilde{\mu})_w,\psi_w)$ is a $p$-unit.
    \end{enumerate}
\end{prop}
\begin{proof}
    At $\sigma\in \Sigma$
    we have seen that 
    $\omega^\square(\iota(1,\smat{\beta_1&\\&\beta_2}),
    \smat{h_1&\\&h_2})\Phi_{P_\sigma,Z_{\delta,\sigma}}=
    \chi(\beta_1\beta_2)\eta^{-1}(h_1)
    \omega^\square(g_\sigma)\Phi_{P_\sigma,Z_{\delta,\sigma}}$, so
    \[
        P(\eta,\Phi_{P_\sigma,Z_{\delta,\sigma}})=\vol(\C_1)^4\cdot 
        (\Phi_{P_\sigma, Z_{\delta,\sigma}},
        \Phi_{P_\sigma, Z_{\delta,\sigma}})^\flat
        =\vol(\C_1)^4\cdot
        \frac{-\delta^{k+2}}{(-2\pi i)^k}\Gamma(k+1)
    \]
    At a non-split place $v$, since $\eta\mu(h_1)=1=\mu^{-1}(h_2)$
    we can compare $P(\eta, \omega^\square(\xi_v)\Phi_v)$ above
    with \eqref{eq:reduction} and see
    \[
        P(\eta, \omega^\square(\xi_v)\Phi_v)=
        \vol(\K_v^1)^2\cdot 
        \left[
            \int_{\K_v^1}\mm_n(\beta)\,d\beta
        \right]^2=
        \vol(\K_v^1)^4\cdot |4\delta|_w q_w^{-2(e_v-1)}.
    \]
    
    At a place $v=w\bw$, write 
    $h_i=(x_i,x_i^{-1})$ and $\beta_i=(y_i,y_i^{-1})$ for $i=1,2$, with
    $x_i=\iota_w(h_i)=\iota_\bw(h_i)^{-1}$ and
    $y_i=\iota_w(\beta_i)=\iota_\bw(\beta_i)^{-1}$.
    Similar to \eqref{eq:m3}, \eqref{eq:m4}, and \eqref{eq:m5} we have
    \begin{align}
        P(\eta, &\omega^\square(\xi_v)\Phi_v)=
        \chi_w^{-4}(\delta_{v})\notag\\\cdot
        &\int_{\F_v^{\times}}\int_{\F_v}
        \phi_{1w}(x_1^{-1}a_1)\hat{\phi}_{1w}(a_1\delta_v)
        \tilde{\eta}\tilde{\mu}_w(x_1)\,d^\times x\,da_1
        \int_{\F_v^{\times}}\int_{\F_v}
        \phi_{2w}(x_2^{-1}a_2)\hat{\phi}_{2w}(a_2\delta_v)
        \tilde{\mu}_w^{-1}(x_2)
        \,d^\times x_2da_2\notag\\\cdot 
        &\int_{\F_v^{\times}}\int_{\F_v}
        \phi_{1\bw}(x_1b_1y_1^{-1})\hat{\phi}_{1\bw}(b_1\delta_v^{-1})
        \chi^{-1}_w(y_1)|y_1|_v^{-1/2}\,d^\times y_1\,db_1 \notag\\\cdot
        &\int_{\F_v^{\times}}\int_{\F_v}
        \phi_{2\bw}(x_2b_2y_2^{-1})\hat{\phi}_{2\bw}(b_2\delta_{v}^{-1})
        \chi^{-1}_w(y_2)|y_2|_v^{-1/2}\,d^\times y_2\,db_2 \notag\\
        &\overset{y_i\mapsto x_iy_i,\, a_i\mapsto \delta_v^{-1}a_i,\,
        b_i\mapsto \delta_v b_i}{=}
        \chi_w^{-4}(\delta_{v})\notag\\\cdot
        &\int_{\F_v^{\times}}\int_{\F_v}
        \phi_{1w}(x_1^{-1}a_1\delta_{v}^{-1})\hat{\phi}_{1w}(a_1)
        (\chi^{-1}\tilde{\eta}\tilde{\mu})_w(x_1)|x_1|_v^{-1/2}
        \,d^\times x_1\,da_1\label{eq:m3_1}\\\cdot
        &\int_{\F_v^{\times}}\int_{\F_v}
        \phi_{2w}(x_2^{-1}a_2\delta_{v}^{-1})\hat{\phi}_{2w}(a_2)
        (\chi\tilde{\mu})^{-1}_w(x_2)|x_2|_v^{-1/2}\,d^\times x_2\,da_2\label{eq:m3_2}\\\cdot
        &\int_{\F_v^{\times}}\int_{\F_v}
        \phi_{1\bw}(b_1\delta_{v}y_1^{-1})\hat{\phi}_{1\bw}(b_1)
        \chi^{-1}_w(y_1)|y_1|_v^{-1/2}\,d^\times y_1\,db_1 \label{eq:m4'}\\\cdot
        &\int_{\F_v^{\times}}\int_{\F_v}
        \phi_{2\bw}(b_2\delta_{v}y_2^{-1})\hat{\phi}_{2\bw}(b_2)
        \chi^{-1}_w(y_2)|y_2|_v^{-1/2}\,d^\times y_2\,db_2 \label{eq:m5'}
    \end{align}
    While \eqref{eq:m4'} and \eqref{eq:m5'} are equal to 
    exactly \eqref{eq:m4} and \eqref{eq:m5},
    let $x_i\mapsto x_ia_i\delta_v^{-1}$ and
    line \eqref{eq:m3_1} becomes 
    \begin{multline*}
        (\chi(\tilde{\eta}\tilde{\mu})^{-1})_w(\delta_v)
        |\delta_v|_v^{1/2}(1-q_v^{-1})
        \\\cdot 
        \int_{\F_v^\times}\phi_{1,w}(x_1^{-1})(\chi^{-1}\tilde{\eta}\tilde{\mu})_w(x_1)
        |x_1|_v^{-1/2}\,d^\times x_1
        \int_{\F_v^\times}\hat{\phi}_{1,w}(a_1)(\chi^{-1}\tilde{\eta}\tilde{\mu})_w(a_1)
        |a_1|_v^{1/2}\,d^\times a_1\\
        =(\chi\tilde{\eta}\tilde{\mu}^{-1})_w(\delta_v)|\delta_v|_v^{1/2}(1-q_v^{-1})
        \cdot 
        Z(\frac{1}{2},\phi_{1,w}, (\chi^{-1}\tilde{\eta}\tilde{\mu})_w^{-1})
        Z(\frac{1}{2},\hat{\phi}_{1,w}, (\chi^{-1}\tilde{\eta}\tilde{\mu})_w).
    \end{multline*}
    The local functional equation then shows that it is equal to
    \[(\chi(\tilde{\eta}\tilde{\mu})^{-1})_w(\delta_v)|\delta_v|_v^{1/2}(1-q_v^{-1})
    \vol(\oo_v^\times)^2\cdot 
    \varepsilon(\frac{1}{2},  (\chi(\tilde{\eta}\tilde{\mu})^{-1})_w,\psi_w)
    L(\frac{1}{2}, (\chi(\tilde{\eta}\tilde{\mu})^{-1})_v)\]
    except when $w\mid p\fs$ and $\tilde{\eta}_w$ is unramified,
    in which case 
    $Z(\frac{1}{2},\phi_{1,w}, (\chi^{-1}\tilde{\eta}\tilde{\mu})_w^{-1})
    =\vol(\oo_v^\times)$ and \eqref{eq:m3_1} is equal to 
    \[
    (\chi(\tilde{\eta}\tilde{\mu})^{-1})_w(\delta_v)|\delta_v|_v^{1/2}(1-q_v^{-1})
        \vol(\oo_v^\times)^2\cdot 
        \varepsilon(\frac{1}{2}, (\chi(\tilde{\eta}\tilde{\mu})^{-1})_w,\psi_w)
        \frac {L(1/2, (\chi^{-1}\tilde{\eta}\tilde{\mu})_w)}
        {L(1/2, (\chi(\tilde{\eta}\tilde{\mu})^{-1})_w)}.
    \]
    The same computation  shows that 
    \eqref{eq:m3_2} is equal to 
    \[
        (\chi\tilde{\mu})_w(\delta_v)|\delta_v|_v^{1/2}(1-q_v^{-1})\cdot 
        \vol(\oo_v^\times)^2\cdot 
        \varepsilon(\frac{1}{2}, (\chi\tilde{\mu})_w,\psi_w)L(\frac{1}{2}, (\chi\tilde{\mu})_v).
    \]
    Now combine the above computations gives the formulae.
\end{proof}

\begin{cor}\label{cor:period}
    Let $\Phi$, $C(\K^1)$ and $\fc=\fc_+\fc_-$ be
    as in Corollary \ref{cor:Rallis}.
    Let $C_2=\delta^{2\Sigma}\prod_{v\in\finite}c_v$ 
    for $c_v'$ from the previous proposition and
    $z_\delta'\in \A_{\K,f}^\times$ be the product of 
    \begin{itemize}
        \item $(\fc\delta\dd_\F)_w$ when $v=w\bw$ and $v\nmid p\fs\fl\bar{\fl}$,
        \item $(\fc\fl^m\delta\dd_\F)_w$ when $v=w\bw$ and $v\nmid p\fs$,
        \item $\delta_w$ when $v=w\bw$ and $v\mid p\fs$,
    \end{itemize}
    then after making the modification of $\Phi$ at $\fl$,
    the square of the toric period 
    $P_{\bmu}(\Theta^\square_\Phi(\eta,\bnu))^2$ is given by 
    \begin{multline*}
    4(2\pi)^{4\Sigma}C(\K^1)^4C_2
    \frac{\vol([\A_{\K}^1])^2L(\frac{1}{2},\chi)^2
    L(\frac{1}{2},\chi\tilde{\mu})}{\prod_{v\mid\fc_-}(1+q_v^{-1})^4}\\\cdot
    \textnormal{Im}(-\delta)^{k\Sigma}\tilde{\eta}(z_\delta'^{-1})
    \frac{\Gamma((k+1)\Sigma)}{(2\pi)^{k\Sigma}}
    L(\frac{1}{2},\chi(\tilde{\eta}\tilde{\mu})^{-1})\prod_{v\mid p\fs}
    \frac {\varepsilon(\frac{1}{2}, (\chi(\tilde{\eta}\tilde{\mu})^{-1})_w,\psi_w)}
    {L(1/2, (\chi(\tilde{\eta}\tilde{\mu})^{-1})_w)^2}.
    \end{multline*}
\end{cor}
\begin{proof}
    Similar to the argument of Corollary \ref{cor:Rallis},
    the product of all the terms
    from Proposition \ref{prop:period_local} is
\begin{multline*}
(2\pi)^{4\Sigma}
C(\K^1)^4C_2\tilde{\eta}(z_\delta'^{-1})
\left[\frac{L(\frac{1}{2},\chi)}{L^{(\fc_{-})}(1,\epsilon_{\K/\F})}\right]^2
\frac{L(\frac{1}{2},\chi\tilde{\mu})}{L^{(\fc_{-})}(1,\epsilon_{\K/\F})}
\\\cdot
\frac{\textnormal{Im}(-\delta)^{k\Sigma}}{(2\pi)^{k\Sigma}}\Gamma((k+1)\Sigma)
\frac{L(\frac{1}{2},\chi(\tilde{\eta}\tilde{\mu})^{-1})}{L^{(\fc_{-})}(1,\epsilon_{\K/\F})}
\prod
\frac {\varepsilon(\frac{1}{2}, (\chi(\tilde{\eta}\tilde{\mu})^{-1})_w,\psi_w)}
{L(1/2, (\chi(\tilde{\eta}\tilde{\mu})^{-1})_w)^2}
\end{multline*}
and the claim follows from Proposition \ref{prop:period}
and that $\vol([\A_\K^1])=2L(1,\epsilon_{\K/\F})$.
\end{proof}

%

\section{Review of Modular forms}

Let $\W$ and $\V=\W+(-\W)$ be as in the previous section.
Let $\G$ denote the algebraic group
\[
    \G(R)=\GUU(\V)(R)=\{g\in \GL_{4}(\K\otimes_\F R)\mid 
    gw_2g^*=\nu(g)w_2 \text{ for some } \nu(g)\in R^\times\}.
\]
For $g\in M_{4}(\K\otimes_\F R)$, we define the involution
$g^\vee=w_2^{-1}g^*w_2$.
If $g\in \G(R)$, then $g^\vee= \nu(g)g^{-1}$.
To prepare for the $p$-adic interpolation 
of the theta lifts constructed in the previous section,
we review the set-ups of algebraic and $p$-adic modulars
on the unitary similitude group $\G$ following \cite{Hsieh2014}.

Recall that the canonical basis \eqref{def:basisV}
defines a polarization $\V=X_\K\oplus Y_\K$.
Let $Y$ and $X^\vee$ denote respectively the $\OK$-lattices
$\OK\yy^1\oplus \OK\yy^2$ and 
$\dd^{-1}_\K\xx^1\oplus \dd^{-1}_\K\xx^2$.
Then the $\OK$-lattice $M=Y\oplus X$ is self-dual
with respect to the alternating pairing
$\langle\cdot,\cdot\rangle\colon \V\times\V\to \Q$
defined by $\langle v,v'\rangle=\Tr_{\K/\Q} (vw_2v^*)$.
For each finite place $v\in \finite$ we define the open compact subgroup
\begin{equation*}
    \KK_v^0=\{ g\in \G(\F_v)\mid 
    (M\otimes_\OF \oo_v)g=
    (M\otimes_\OF \oo_v)\}.
\end{equation*}

Let $\fc, \fs$ and the choice of Bruhat-Schwartz functions
be as in \S\ref{sec:choice}.
Write $\fn=\fc\fs\dd_\F$,
we consider an open compact subgroup
$\KK(\fn)=\prod_v\KK_v$, where
$\KK_v\subset \KK_v^0$ are open compact subgroups such that 
\begin{enumerate}[label={(\bfseries K\arabic*)}]
    \item $\omega^\square(g)\Phi_v(z)=\Phi_v(z)$ for 
    $g\in \KK_v\cap \UU(\V)(\F_v)$  when $v\nmid p\fs$.
    \item $\KK_v=\{g\in \KK_v^0\mid g\equiv\smat{1&*&*&*\\&*&*&*\\&&1&\\&&*&*}\mod \fn\}$
    when $v\mid \fs$.
    \item $\KK_v=\KK_v^0$ for almost all $v\in\finite$ and when $v\mid p$.
\end{enumerate}
The definition makes sense since $\Phi_v$ 
is independent of the admissible character $\eta$ when $v\nmid p\fs$.
Note that by Lemma \ref{lem:invarianceK} we also have
$\omega^\square(g)\Phi_v(z)=\Phi_v(z)$ for $g\in\KK_v\cap\UU(\V)(\F_v)$.

Write $\fn_\circ=\fc\dd_\F$ when $\fs=\oo_\K$.
We fix such an open compact subgroup $\KK(\fn_\circ)$.
Possibly after shrinking $\KK_v$ at some prime-to-$p$ place $v\in\finite$,
we assume $\KK(\fn_\circ)$ satisfies 
\begin{equation}\tag{neat}
    \KK(\fn_\circ)\subset \{g\in \G(\A_{\F,f})\mid M(g-1)\subset N_0\cdot M\},\quad\nu(\KK)\cap \OF^\times_+\subset (\KK\cap \OF^\times)^2\notag
\end{equation}
for some integer $N_0\geq 3$. 
We will then only consider ideals $\fs$ such that 
$\KK_v^0\subset \KK(\fn_\circ)$. And for such $\fs$
we shall let $\KK(\fn)\subset \KK(\fn_\circ)$ be the subgroup
defined by replacing components at $v\mid\fs$ by the subgroup given above.

On the other hand, for $\K_p\coloneqq \K\otimes_\Q\Q_p$ and
$\F_p\coloneqq \F\otimes_\Q\Q_p$ we have the isomorphisms
$\iota_\Sigma\times\iota_{\Sigma^c}\colon \K_p\cong
\F_p(\Sigma)\times\F_p(\Sigma^c)$,
where we define $\iota_\Sigma=\sum_{w\in\Sigma_p}\iota_w$
and $F_p(\Sigma)$ indicates the $\K_p$-algebra structure on $\F_p$
given by $\iota_\Sigma$.
The notations $\iota_{\Sigma^c}$ and $\F_p(\Sigma^c)$ 
is defined similarly.
We also let $e^+, e^-\in\K_p$ denote the idempotents
corresponding to $\F_p(\Sigma)$ and $\F_p(\Sigma^c)$ respectively.
Then we have the isomorphism
\begin{equation*}
    \iota_\Sigma\colon 
    \Res_{\F/\Q}\G(\Q_p)\mapsto \GL_4(\F_p(\Sigma))\times \F_p^\times,\qquad g_p\mapsto (g_\Sigma, \nu(g_p))
\end{equation*}
which identifies $\KK^0_p\coloneqq \prod_{v\mid p}\KK^0_v$ with 
$\GL_4(\oo_p(\Sigma))\times \oo_p^\times$. 
We define the open compact subgroups
\begin{align*}
    \KK^n&=\{g\in \KK\mid g_\Sigma\equiv \smat{1&&*&*\\&1&*&*\\&&1&\\&&&1}, \nu(g_p)\equiv 1\bmod p^n\},\\
    \KK_1^n&=\{g\in \KK\mid g_\Sigma\equiv \smat{1&*&*&*\\&1&*&*\\&&1&\\&&*&1}, \nu(g_p)\equiv 1  \bmod p^n\},\\
    \KK_0^n&=\{g\in \KK\mid g_\Sigma\equiv \smat{*&*&*&*\\&*&*&*\\&&*&\\&&*&*}\bmod p^n\}.
\end{align*}

\subsection{Kottwitz models}

Let $\keu$ be a finite extension of $\K^{nc}(e^{2\pi i/N_0})$ which is unramified at $p$,
and let $\oeu=\oo_{\keu,(\fp)}$ be the localization
at the prime ideal $\fp$ induced by $\iota_p$.
Let $\oo_{(p)}=\OF\otimes_\Z\Z_{(p)}$ 
and $\oo_{(p),+}$ be the subring of totally positive elements in $\oo_{(p)}$. 

\begin{defn}
    Let $S$ be a locally noetherian connected $\oeu$-scheme
    and $\bar{s}$ be a geometric point of $S$. 
    Write $\KK=\KK(\fn)$ and $\KK^{(p)}\coloneqq \prod_{v\nmid p}\KK_v$.
    An $S$-quadruple of level $\KK^{(p)}$ is a quadruple
    $\underline{A}=(A,\bar{\lambda},\iota,\bar{\eta}^{(p)})_S$ consisting of the following data:
    \begin{enumerate}
        \item $A$ is an abelian scheme of dimension $4d$ over $S$.
        \item $\iota\colon \OK\hookrightarrow \End_S(A)\otimes_\Z \Z_{(p)}$.
        \item $\lambda$ is a prime-to-$p$ polarization of $A$ over $S$
        and $\bar{\lambda}$ is the $\oo_{(p),+}$-orbit
        \[
            \bar{\lambda}=\{ \lambda'\in \Hom(A, A^t)\otimes_\Z\Z_{(p)}\mid 
            \lambda'=\lambda\circ a, a\in \OF_{(p),+}\}.
        \]
        \item $\bar{\eta}^{(p)}=\eta^{(p)}\KK$ is a $\pi_1(S,\bar{s})$-invariant $\KK$-orbit
        of an isomorphism
        $\eta^{(p)}\colon M\otimes \hZ^{(p)}\cong T^{(p)}(A_{\bar{s}})\coloneqq \otimes_{\ell\neq p}T_\ell A_{\bar{s}}$
        between $\OK$-modules.
        Here $\eta^{(p)}g(x)=\eta(xg^\vee)$ for $g\in \KK$.
    \end{enumerate}
    Furthermore, it is required that 
    $(A,\bar{\lambda},\iota,\bar{\eta}^{(p)})_S$ satisfies the following conditions.
    \begin{itemize}
        \item Let $^t$ be the Rosati involution induced by $\lambda$,
        then $\iota(b)^t=\iota(\bar{b})$ for all $b\in \OK$.
        \item Let $e^\lambda$ be the $\pi_1(S,\bar{s})$-invariant Weil pairing induced by $\lambda$,
        which can be regarded as a skew-Hermitian form 
        $e^\lambda\colon T^{(p)}(A_{\bar{s}})\times T^{(p)}(A_{\bar{s}})\to \dd_\K^{-1}\otimes_\Z\hZ^{(p)}$;
        and $e^\eta$ be the skew-Hermitian on $T^{(p)}(A_{\bar{s}})$ induced by the skew-Hermitian form on $M$ via $\eta$, then
        \[
            e^\lambda=u\cdot e^\eta \text{ for some } u \in \A_{\F,f}^{(p)}.
        \]
        That $\eta^{(p)}K$ is $\pi_1(S,\bar{s})$-invariant is the same as that
        the image of $\pi_1(S,\bar{s})$
        under the induced homomorphism $\pi_1(S,\bar{s})\to \G(\A_{\F,f}^{(p)})$ lies in $\KK$.
        \item For all $b\in \OK$, the determinant of $\iota(b)$ on $\Lie A$ satisfies the determinant condition
        \begin{equation*}
            \det(X-\iota(b)\mid \Lie A)=\prod_{\sigma}(X-\sigma(b))^2(X-\sigma(\bar{b}))^2\in \oo_S[X].
        \end{equation*}
    \end{itemize}
\end{defn}

Let $\mathfrak{C}_\KK^{(p)}$ be the fibered category over $SCH_{\oeu}$
where objects over $S$ are the $S$-quadruples and 
\[
    \Hom_{\mathfrak{C}_\KK^{(p)}}(\underline{A},\underline{A}')=
    \{\phi\in \Hom_{\OK}(A,A')\mid \phi^*\bar{\lambda}'=\bar{\lambda},
    \phi(\bar{\eta}'^{(p)})=\bar{\eta}^{(p)}\}
\]
for 
$\underline{A}=(A,\bar{\lambda},\iota,\bar{\eta}^{(p)})_S$ and 
$\underline{A}'=(A',\bar{\lambda}',\iota',\bar{\eta}'^{(p)})_S$.
The functor $\mathfrak{S}^{(p)}_\KK\colon SCH/\oeu\to SETS$ defined by
$\mathfrak{S}^{(p)}_\KK(S)=\{\underline{A}=(A,\bar{\lambda},\iota,\bar{\eta}^{(p)})_S\in \mathfrak{C}^{(p)}_\KK(S)\}/\sim$
is represented by a quasi-projective scheme $S_\G(\KK)$ over $\oeu$.

\subsection{Igusa schemes over the compactification}

For $r\in\{0,1\}$ let $P^r=M^rN^r$ be the stabilizer of $Y^r_\K$ in
the decreasing filtration $Y^0_\K=Y_\K\supset Y^1_\K=\K\yy^2$,
where $M^r$ and $N^r$ are respectively 
the Levi and the unipotent components.
Let $\G_1$ be the unitary similitude group on the subspace generated by $\{\xx^2,\yy^2\}$. 
The set of cusps of genus $2r$ for $S_\G(\KK)$ is defined by
\begin{align*}
    C_0&(\KK)\coloneqq \GL(Y_\K)N^0(\A_{\F,f})\backslash \G(\A_{\F,f})/\KK,\\
    C_1&(\KK)\coloneqq (\GL(Y^1_\K)\times \G_1(\A_{\F,f}))N^1(\A_{\F,f})\backslash \G(\A_{\F,f})/\KK.
\end{align*}
Let $\bar{S}_\G(\KK)$ be a toroidal compacitification over $\oeu$
associated to choices of combinatoric data on each cusp.
There is a quadruple $\underline{\geu}=(\geu, \lambda,\iota,\eta)$ over $\bar{S}_\G(\KK)$
of semi-abelian scheme and its PEL-data,
whose restriction over $S_\G(\KK)$ is the universal quadruple.

\begin{defn}
    Let $n$ be a positive integer and let $M^0=Y\otimes_\Z\Z_p$.
    The Igusa scheme $I_G(\KK^n)$ over $\bar{S}_\G(\KK)$ represents the functor
    \[
    \underline{I_\G}(\KK^n)\coloneqq\underline{\Inj}_{\OK}(\mu_{p^n}\otimes_\Z M^0,\geu).
    \]  
    Let $\KK^n_0$ acts on $j:\mu_{p^n}\otimes_\Z M^0\to \geu$ 
    by $g\cdot j(m^0)=j(m^0\cdot g_p)$. Define 
    $I_\G(\KK^n_\bullet)=I_\G(\KK^n)/\KK^n_\bullet$ for $\bullet=0,1$.
\end{defn}

\subsection{Geometric modular forms}\label{sec:modular_forms}

Let $R$ be an $\oeu$-algebra,
for $\boldsymbol{k}=(a_1,a_2;b_1,b_2)\in\Z^4$, consider the module
\[
    L_{\boldsymbol{k}}(R)=
    \{f\in R[\GL_2\times\GL_2]\mid 
    f(tng)=\boldsymbol{k}^{-1}(t)f(g),\,
    t\in T_2\times T_2,\, n\in N^+_2\times N_2^- \}.
\]
Here $R[\GL_2\times\GL_2]$ denotes the structure ring of 
$\GL_2\times \GL_2$ over $R$,
$T_2$ is the diagonal torus in $\GL_2$,
and $N^+, N^-$ are respectively the upper and lower triangular subgroups.
And $\boldsymbol{k}(\smat{t_1&\\&t_2},\smat{t_3&\\&t_4})\coloneqq 
t_1^{a_1}t_2^{a_2}\cdot t_3^{b_1}t_4^{b_2}$.
Then $L_{\boldsymbol{k}}(R)$ is an algebraic representation of $\GL_2(R)\times \GL_2(R)$
of lowest weight $-\boldsymbol{k}$. 

Define 
$H=\Res_{\OF/\Z} \GL_2\times \Res_{\OF/\Z} \GL_2$,
$T_H=\Res_{\OF/\Z} T_2\times \Res_{\OF/\Z} T_2$,
$N_H=\Res_{\OF/\Z} N^+_2\times \Res_{\OF/\Z} N^-_2$,
and $B_H=N_HT_H$.
Given $\bwt{k}=(a_1,a_2;b_1,b_2)\in \Z[\Sigma]^4$ with 
$a_i=(a_{i,\sigma})$ and
$b_i=(b_{i,\sigma})$, for each $\sigma$ we define 
$\boldsymbol{k}_{\sigma}=(a_{1,\sigma}, a_{2,\sigma}, b_{1,\sigma}, b_{2,\sigma})$.
Then the algebraic representation of $H(R)$ of lowest weight $-\bwt{k}$ is 
$L_{\bwt{k}}(R)\coloneqq \bigotimes_{\sigma\in \Sigma}L_{\boldsymbol{k}_\sigma}(R)$.
We say $\bwt{k}$ is dominant if 
$a_{1,\sigma}\geq a_{2,\sigma}\geq -b_{1,\sigma}\geq -b_{2,\sigma}$
for all $\sigma\in \Sigma$ and put
\[
    \lVert \bwt{k}\rVert\coloneqq (a_{1,\sigma}+a_{2,\sigma}+b_{1,\sigma}+b_{2,\sigma})\in \Z[\Sigma],\quad
    |\bwt{k}|=\sum_{\sigma \in \Sigma}
    (b_{1,\sigma}+b_{2,\sigma})\sigma+
    (a_{1,\sigma}+a_{2,\sigma})\sigma^c\in \Z[I_\K].
\]

Let $\underline{\omega}=e^*\Omega_{\geu/\bar{S}_G(K)}$ be the sheaf of invariant differentials
and decompose
$\underline{\omega}=\oplus_{\sigma\in I_\K}e_\sigma \underline{\omega}$, put
\[
    \bun^+=\bigoplus_{\sigma\in\Sigma}\underline{\Isom}(\oo^2_{\bar{S}_G(K)}, e_\sigma \underline{\omega});\quad
    \bun^-=\bigoplus_{\sigma\in\Sigma^c}\underline{\Isom}(\oo^2_{\bar{S}_G(K)}, e_\sigma \underline{\omega}).
\]
Then $\bun=\bun^+\oplus\bun^-$ is an $H$-torsor over $\bar{S}_G(\KK)$.
Define the automorphic sheave of weight $\bwt{k}$ by $\omega_{\bwt{k}}\coloneqq \bun\times^HL_{\bwt{k}}$.
\begin{defn}
    Let $\bwt{k}$ be a dominant weight.
    When $R$ is an $\Z_{(p)}$-flat $\oeu$-algebra, define the space of geometric modular forms over $R$
    of weight $\bwt{k}$ and level $\KK^n_\bullet$ by
    \[
        \M_{\bwt{k}}(\KK^n_\bullet, R)\coloneqq
        H^0(I_G(\KK^n_\bullet)/R, \omega_{\bwt{k}}),\quad
        \bullet=0, 1, \emptyset.
    \]
    A section $\ff\in \M_{\bwt{k}}(\KK^n_\bullet, R)$ can be viewed
    as an $L_{\bwt{k}}(R)$-valued functions on test objects $(x, \ww)\in \bun(R)$
    that satisfies
    $\ff(x,h\boldsymbol{\omega})
    =\rho_{\bwt{k}}(h)\ff(x,\ww)$
    where $h=(h^+,h^-)\in H(R)$ acts on 
    $\ww=(\ww^+,\ww^-)\in \bun=\bun^+\oplus \bun^-$ by
    $h\cdot \ww(v^+,v^-)=\ww(v^+h^+, v^-h^-)$, for
    $v^+\in \bigoplus_{\Sigma}\oo^2_{\bar{S}_G(K)}$ and
    $v^-\in \bigoplus_{\Sigma^c}\oo^2_{\bar{S}_G(K)}$.
\end{defn}

\subsection{Complex uniformization}

For $\sigma\in I_\K$, let 
$\C(\sigma)=\C$ denote the $\K\otimes\C$-module
on which $\K$ acts through $\sigma$.
We define $\C(\Sigma)=\prod_{\sigma\in\Sigma}\C(\sigma)$
and similarly $\C(\Sigma^c)=\prod_{\sigma\in\Sigma}\C(\sigma c)$.
Then we put $\X=\prod_{\sigma\in\Sigma}\X_{\sigma}$
for $\X_\sigma$ as defined in \S\ref{sec:arch}. In other word
\[
    \X=\{Z=(Z_\sigma)\in M_2(\C(\Sigma))\mid i(Z^*_\sigma-Z_\sigma)>0
    \text{ for all }\sigma\in\Sigma\}.
\]
To each pair $(Z,g)\in \X\times \G(\A_{\F,f})$ we define a quintuple
$(\underline{\cA}_g(Z), j_g)=(\cA_g(Z), \overline{\langle, \rangle}_{can}, \iota_\C, \bar{\eta}_g, j_g)$ as follows.
\begin{enumerate}
    \item $\mathcal{A}_g(Z)\coloneqq \C^{2,2}/M_{[g]}(Z)$, where $\C^{2,2}$ is the $\OK$-module
    $\C(\Sigma^c)^2\oplus \C(\Sigma)^2$ and 
    $M_{[g]}(Z)$ is the image of the $\OK$-lattice $M_{[g]}\coloneqq (M\otimes_\Z\hZ)g^\vee\cap \V$
    under the $\R$-linear map
    $p(Z)\colon \V\otimes_\Q\R\cong \C^{2,2}, p(Z)(v)=c_{2,2}(vB(Z))$.
    Here $c_{2,2}(u_1,u_2)=(\bar{u}_1,u_2)$ on $\C^{2,2}$ and 
    $B(Z)=\smat{Z^*&Z\\\1_2& \1_2}\in M_4(\C^\Sigma)$.
    \item $\overline{\langle,\rangle}_{can}$ is the $\F_+$-orbit of the polarization induced via $p(Z)$ by the alternating pairing on $\C^{2,2}$,
    which is induced from the alternating pairing 
    $\langle\cdot,\cdot\rangle$ on $\V$.
    \item $\iota_\C\colon \OK\to \End \mathcal{A}_g(Z)\otimes_\Z\Q$
    is induced  via $p(Z)$ by the $\OK$-action on $\V$.
    \item the prime-to-$p$ level structure $\iota_g^{(p)}\colon M\otimes\hZ^p\cong M_{[g]}=$ is defined by 
    $\eta_g^{(p)}(x)=xg^\vee$.
    \item identify $\zeta\colon \Z/p^n\Z\cong \mu_{p^n}$ via $\zeta=e^{2\pi i /p^n}$
    \[
        j_\zeta\colon M^0\otimes\mu_{p^n}\cong M^0\otimes\Z/p^n\Z\hookrightarrow
        \mathcal{A}_g(Z)[p^n]=M_{[g]}\otimes \Z/p^n\Z
    \]
    is defined by $j_\zeta(x^0)=x^0g^\vee$.
\end{enumerate}

By the theory of complex uniformization, the $\C$-point of the Igusa schemes $I_\G(\KK^n_\bullet)$ is isomorhpic to the complex variety
$\G(\F)^+\backslash \X\times \G(\A_{\F,f})/\KK^n_\bullet$.

For $g_\infty=(g_\sigma)\in \Res_{\F/\Q}\G(\R)^+$,
let $g_\infty$ acts on $Z=(Z_\sigma)\in \X$ 
as in \S\ref{sec:arch} component-wisely and define
\[
    J(g_\infty, Z)=(J(g_\sigma,Z_\sigma))\in
    \prod_\sigma \GL_2(\C)\times\GL_2(\C)=H(\C).
\]
Let $z^i(\Sigma^c)$ and $z^i(\Sigma)$ be the coordinates on $\C^{2,2}$ for $i=1,2$, 
then $d\underline{z}\coloneqq (dz^1(\Sigma^c),dz^2(\Sigma^c),dz^1(\Sigma),dz^2(\Sigma))$ 
is an $\OK\otimes_\Z \C$-basis of invariant differentials of $\cA_g(Z)$,
which determines an isomorphism $\C^{2,2}\cong \Omega_{\cA}$.
We can view $(\underline{\cA}_g(Z),2\pi i d\underline{z})$
as a test object in $\bun(\C)$, then the map that associates 
a modular form $\ff\in \M_{\bwt{k}}(\KK^n_\bullet, \C)$
to $\ff(Z,g)\coloneqq \ff(\underline{\cA}_g(Z), 2\pi id\underline {z})$
defines an isomorphism from $\M_{\bwt{k}}(\KK^n_\bullet, \C)$ 
to the space of holomorphic modular forms defined below.
\begin{defn}
    Write $\rho^{\bwt{k}}(h)\coloneqq\rho_{\bwt{k}}(h^{-\intercal})$.
    For $\bullet=0,1,\emptyset$,
    a holomorphic modular form
    of weight $\bwt{k}$ and level $\KK^n_\bullet$ is a holomorphic function
    $\ff(Z,g)\colon \X\times \G(\A_{\F,f})\to L_{\bwt{k}}(\C)$ satisfying 
    \begin{equation}\label{defn:holoform} 
        \ff(\alpha Z, \alpha gk)=\nu(\alpha)^{-\lVert \bwt{k}\rVert}\rho^{\bwt{k}}(J(\alpha,Z))
        \ff(Z,g), \quad
        \alpha\in \Res_{\F/\Q}\G(\R)^+,\,
        k\in \KK^n_\bullet
    \end{equation}
    Let $\M_{\bwt{k}}^{\an}(\KK^n_\bullet,\C)$ denote the space of 
    holomorphic modular forms
    of weight $\bwt{k}$ and level $\KK^n_\bullet$.
\end{defn}

Moreover, suppose $\ff(Z,g_f)$, 
for $g_f\in \G(\A_{\F,f}^{(p)})$, has the Fourier expansion
\[
    \ff(Z,g_f)=\sum_{\beta\in\her_2(\F)}a_\beta(g_f)q^\beta,\quad
    a_\beta(g_f)\in L_{\bwt{k}}(\C),
\]
for $q^\beta\coloneqq e^{2\pi i\sum_{\sigma\in\Sigma}\mtr(Z_\sigma\beta)}$.
Let $[g_f]$ be the class of $g_f$ in $C_0(\KK)$ and let
$A\in \GL(M_0)$. 
If we put $m(A)=\smat{A^{-*}&\\&A}\in \G(\F_p)$ and
view $([g_f],A)$ as a cusp for $I_G(\KK^n)$, then
the Fourier expansion of $\ff(Z,m(A)g_f)$,
viewed as a formal series in $q^\beta$, coincides with 
\begin{equation}\label{def:Fourier}
    F_{[g_f]}^{A}(\ff)\coloneqq
    \ff(\underline{\mathscr{M}}_{[g_f]}, A^{-1}j_{\mathscr{M}_{[g_f]}}, d^\times t)
\end{equation}
for a canonical Mumford quintuple 
$(\underline{\mathscr{M}}_{[g_f]}, j_{\mathscr{M}_{[g_f]}})$ 
associated to the cusp.

\begin{defn}\label{defn:autoform}
    Let $\omega$ be an algebraic Hecke character of infinity type $|\bwt{k}|$.
    An automorphic form of weight $\bwt{k}$ and level $\KK^n_\bullet$ is a 
    smooth and slowly-increasing function
    $\ff\colon \G(\A_{\F})\to L_{\bwt{k}}(\C)$ satisfying 
    \begin{equation*}
        \ff(z\alpha gk_\infty k)=\rho^{\bwt{k}}(J(k_\infty, \ii)^{-1})\ff(g)\omega^{-1}(z),
        \quad 
        \alpha\in \G(\F),\, k\in \KK^n_\bullet,\, z\in Z_\G(\A_{\F})
    \end{equation*}
    and $k_\infty$ belongs to the stabilizer $C_\infty\subset \Res_{\F/\Q}\G(\R)^+$ of
    $\ii\coloneqq (\ii_\sigma)\in \X$.
    Let $\AM_{\bwt{k}}(\KK^n_\bullet,\omega)$ denote the space of 
    automorphic forms of weight $\bwt{k}$, level $\KK^n_\bullet$
    and central character $\omega^{-1}$.
\end{defn}
To an automorphic form $\ff(g)\in \AM_{\bwt{k}}(\KK^n_\bullet,\omega)$,
we can associate a holomorphic modular form by
$\ff(Z,g_f)=\omega_f(\nu(g_f))
\rho^{\bwt{k}}(J(g_\infty,\ii))\ff(g_\infty g_f)$
for any $g_\infty\in \Res_{\F/\Q}\G(\R)^+$ satisfying $g_\infty \ii=Z$.

\subsection{p-adic modular forms}

When $R$ is a $p$-adic $\oeu$-algebra and $m\in \Z_{\geq 1}$,
we define $T_{n,m}\coloneqq I_\G(\KK^n)/\oeu_m$ for $n\geq m$,
where $R_m\coloneqq R/p^mR$, and put
\begin{align*}
    V_{n,m}&=H^0(T_{n,m},\oo_{T_{n,m}}),\\
    V_{\bwt{k}}(\KK^n_1,R_m)&=H^0(T_{n,m/R_m},\omega_{\bwt{k}})^{\KK^n_1}
    \cong V^{\KK^n_1}_{m,n}\otimes L_{\bwt{k}}(R_m).
\end{align*}
Here to each $R_m$-quintuple $(\underline{A},j)$
one can define a canonical $(\underline{A},j,\ww(j))\in\bun(R_m)$.
and the isomorphism above is given by 
mapping $\ff\in V_{\bwt{k}}(\KK^n_1,R_m)$ to
$\hat{\ff}(\underline{A},j)\coloneqq \ff(\underline{A},j,\ww(j))$,
which we refer as the $p$-adic avatar of the modular form $\ff$.

\begin{defn}\label{def:avatar}
    Let $\blk$ be the functional on $L_{\bwt{k}}$ given by 
    evaluating at the identity $\id\in H$, 
    which amounts to projecting to the lowest weight vector of
    $\rho_{\bwt{k}}$ (equivalently the highest weight vector of
    $\rho^{\bwt{k}}$).
    We then define 
    \[
        \beta_{\bwt{k}}\colon V_{\bwt{k}}(\KK^n_1,R_m)\to 
        V_{m,n}^{\KK^n_1}\quad
        \beta_{\bwt{k}}(\ff)\coloneqq \blk(\hat{\ff}).
    \]
    Let $V_p(\G,\KK)$ be the space of $p$-adic modular forms
    defined as $\varprojlim_{m}\varinjlim_{n} V_{n,m}^{\KK^n_1}$.
    For $\ff\in\M_{\bwt{k}}(\KK_1^n,R)$, 
    we put $\ff_m=\ff\bmod p^m$, viewed as an element in 
    $V_{\bwt{k}}(\KK^m_1,R_m)$ and define 
    \[
        \hat{\ff}=\varprojlim_m \hat{\ff}_m\in 
        \varprojlim_{m}\varinjlim_{n} V_{\bwt{k}}(\KK^m_1,R_m)
        \text{ and }
        \beta_{\bwt{k}}(f)=\varprojlim_m \beta_{\bwt{k}}(\hat{\ff}_m)\in 
        V_p(\G,\KK),
    \]
    which we refer respectively as the $p$-adic avatar and 
    the p-adic modular form associated to $\ff$.
\end{defn}

Similar to \eqref{def:Fourier}, if $\ff\in V_{m,n}$, 
we can define the Fourier expansion at the cusp $([g_f],A)$ by
\begin{equation*}
    \hat{F}_{[g_f]}^{A}(\ff)\coloneqq
    \ff(\underline{\mathscr{M}}_{[g_f]}, A^{-1}j_{\mathscr{M}_{[g_f]}}),
\end{equation*}
which we extends to $V_p(\G,\KK)$ by taking the limit.
We also define 
$\hat{F}_{[g_f]}^{A}(\ff)=\hat{F}_{[g_f]}^{A}(\beta_{\bwt{k}}(\ff))$
if $\ff\in V_{\bwt{k}}(\KK^n_1,R_m)$.

\begin{lem}\label{lem:Fouerier_compare}
   Write $A=(a_1,a_2)\in \GL(M_0)$
   with respect to the decomposition $M_0=e^+M_0\oplus e^-M_0$.
   If $\ff\in \M_{\bwt{k}}(\KK_1^n,R)$ 
   and $\beta_{\bwt{k}}(\ff)$ is the $p$-adic modular form 
   associated to $\ff$, then the Fourier expansion of 
   $\ff$ at the cusp $([g_f],A)$
   defined in \eqref{def:Fourier} and that of $\beta_{\bwt{k}}(\ff)$
   are related by 
   \[
        \hat{F}_{[g_f]}^A(\beta_{\bwt{k}}(\ff))=
        \blk(\rho^{\bwt{k}}(a_2^{-1},a_1^{-1})F_{[g_f]}^A(\ff)).
   \]
   Here for $(g_1,g_2)\in\GL_2(\F_p)\times \GL_2(\F_p)$,
   we define 
    $\rho^{\bwt{k}}(g_1,g_2)=\prod_{w\in\Sigma_p}\prod_{\sigma\in I_w}
    \rho^{\boldsymbol{k}_\sigma}(g_{1,w},g_{2,w})$.
\end{lem}
\begin{proof}
    This follows from that $d^\times t=\ww(j_{\mathscr{M}})$
    for the Mumford family 
    $\underline{\mathscr{M}}=\underline{\mathscr{M}}_{[g_f]}$, thus
    \begin{multline*}
        \hat{\ff}(\underline{\mathscr{M}},A^{-1}j_{\mathscr{M}})=
        \ff(\underline{\mathscr{M}},A^{-1}j_{\mathscr{M}},
        \ww(A^{-1}j_{\mathscr{M}}))=
        \ff(\underline{\mathscr{M}},A^{-1}j_{\mathscr{M}},
        A^\intercal\ww(j_{\mathscr{M}}))\\=
        \rho_{\bwt{k}}(A^\intercal)
        \ff(\underline{\mathscr{M}},A^{-1}j_{\mathscr{M}},
        \ww(j_{\mathscr{M}}))=
        \rho^{\bwt{k}}(A^{-1})
        \ff(\underline{\mathscr{M}},A^{-1}j_{\mathscr{M}},
        d^\times t)=
        \rho^{\bwt{k}}(a_2^{-1}, a_1^{-1})
        \ff(\underline{\mathscr{M}},A^{-1}j_{\mathscr{M}},
        d^\times t).
    \end{multline*}
    Note that here we follow \cite{Hsieh14} and identify
    $H(\Zp)$ with $\GL(e^-M^0)\times \GL(e^-M^0)$.
\end{proof}

\subsection{Definite unitary groups}

Let $G$ denote the algebraic group $\GUU(\W)$.
Since the skew-Hermitian pairing on $\W$ is defined by
$\bdelta=\delta\id_2$ for $\delta$ as in \ref{delta_cond2}, 
we can identify
\[
    G(R)\coloneqq\GUU(\W)=\{g\in \GL_2(\K\otimes_\F R)\mid 
    gg^*=\nu(g) \text{ for some } \nu(g)\in R^\times\}.
\]
Let $G(R)^+=\{g\in G(R)\mid \nu(g)\in \Nr(\K\otimes_\F R)\}$ and
let $L$ be the $\oo_\K$-lattice $\oo_\K\ee_1\oplus \oo_\K\ee_2$.
Fix a sufficiently small open compact subgroup 
$K\subset G(\A_{\F,f})$. A similar moduli problem of
quadrauples of PEL abelian schemes $\underline{\cB}$
is then represented by a finite $\oeu$-scheme $S_{G}(K)$
whose $\C$-points is the finite set
\[
    S_{G}(K)(\C)=G(\F)^+\backslash G(\A_{\F,f})/K.
\]
by the theory of complex uniformization.
Let $L^{-1}\oplus L^{0}$ be the polarization of 
$L_p\coloneqq L\otimes_\Z\Zp=$ defined by $L^{-1}=e^+L_p, L^0=e^- L_p$.
We can similarly define the Igusa moduli problem of level $p^n$
over $S_G(K)$ by considering the extra structure of
$j\colon \mu_{p^n}\otimes L^0\hookrightarrow \cB[p^n]$.
By the theory of complex multiplication, 
the $p^\infty$-level Igusa tower is then
defined over $\mathcal{W}_p$, the $p$-adic completion 
of the ring of integers of the maximal unramified extension of $\Qp$.

Write $\mathcal{S}=S_{G}(K)$ 
and let $\pi\colon \cB\to \mathcal{S}$ be the universal abelian scheme.
As explained in \cite[\S 2.8]{Hsieh2014},
$\mathcal{H}_\cB=\pi_*(\Omega_{\cB/\mathcal{S}})$
is a free $e^+(\OK\otimes_\Z\oo_\mathcal{S})$-module of rank $2$,
and after possibly replacing $\oeu$ by a finite unramified extension
there exists a 
$e^+(\oo_\K\otimes_\Z\oo_\mathcal{S})$-basis $\ww_\cB$ for $\mathcal{H}_\cB$
and a pair of periods 
$(\Omega_{G,\infty}, \Omega_{G,p})\in
\GL_2(e^+(\OK\otimes_\Z\C))\times GL_2(e^+(\OK\otimes_\Z\mathcal{W}_p))$ 
satisfying 
\begin{equation}\label{eq:CMperiod}
    \Omega_{G,\infty} d\underline{z}_\W=\ww_\cB=\Omega_{G,p}d^\times t,
\end{equation}
where $d\underline{z}_\W$ is the basis of invariant differentials
coming from the complex uniformization
and $d^\times t$ is the basis induced by the structures $j$ over the Igusa tower.
Moreover, since $\bdelta=\delta\id_2$, by \ref{delta_cond2}
the abelian schemes over $\mathcal{S}$
are prime-to-$p$ isogeneous to product of 
two abelian schemes with complex multiplications. 
We may therefore assume 
$\Omega_{G,\infty}=\Omega_\infty\id_2$ and 
$\Omega_{G,p}=\Omega_p\id_2$,
where $(\Omega_\infty, \Omega_p)$ is the pair of CM periods 
as in \cite{Hsieh12}.

\subsubsection{Embedding between Igusa schemes}

Recall that $-\W$ is the skew-Hermitian space defined by $-\delta\id_2$
with the basis $\{\ee^-_1,\ee^-_2\}$.
Let $(-L)$ be the $\OK$-lattice $\OK\ee^-_1\oplus\OK\ee^-_2$
and let $(-L)^{-1}\oplus (-L)^{0}$ be the polarization on
$(-L)_p\coloneqq(-L)\otimes_\Z\Zp$ given by
$(-L)^{-1}=e^-(-L)_p, (-L)^0=e^+ (-L)_p$.
Now on $\V=\W+(\W)$ we have the lattice $\mathbf{L}=L+(-L)$.
By the choice of the basis \eqref{def:basisV}
$M_p\coloneqq M\otimes_\Z\Zp$ and 
$\mathbf{L}_p\coloneqq\mathbf{L}\otimes_\Z\Zp$ coincide
by \ref{delta_cond2}.
On which we have the two polarizations given by
\begin{itemize}
    \item $M^{-1}=X^\vee\otimes_\Z\Zp$ and $M^{0}=Y\otimes_\Z\Zp$,
    \item $\mathbf{L}^{-1}=L^{-1}+(-L)^{-1}$ and 
    $\mathbf{L}^0=L^{0}+(-L)^0$.
\end{itemize}

Given $(g_1,g_2)\in G(\UU(\W)\times \UU(-\W))(\A_{\F,f}^{(p)})$,
let $(\underline{\cB}_{g_1},j_{g_1})$ and $(\underline{\cB}_{g_2},j_{g_2})$
be the associated qintuple defined by the uniformization.
Put $\Upsilon_p=\prod_{v\mid p}\xi_v$ for 
$\xi_v$ given by \eqref{def:xi}, then we have 
and $\mathbf{L}^{0}\xi_p= M^0$. It can then be checked that 
\[
    (\underline{\cA}_{g\Upsilon}(Z_\delta), j_{g\Upsilon})\cong 
    (\underline{\cB}_{g_1}\times \underline{\cB}_{g_2},
    j_{g_1}\times j_{g_2}),
\]
for $g=\iota(g_1,g_2)\in \G(\A_{\F,f})$
and $Z_{\delta}=(\Z_{\delta,\sigma})\in \X$ 
as in Definition \ref{def:point}.
More precisely, suppose $\ff$ is a holomorphic modular form
or a modular form with the $p$-adic avatar $\hat{\ff}$, then 
\begin{align}
    \ff(Z_\delta, g\Upsilon)&=
    \ff(\underline{\cA}_{g\Upsilon}(Z_\delta), j_{g\Upsilon}, 
    2\pi i d\underline{z})=
    \ff(\underline{\cB}_{g_1}\times \underline{\cB}_{g_2}, 
    j_{g_1}\times j_{g_2}, 2\pi id\underline{z}_\W\times 
    2\pi id\underline{z}_\W),\label{eq:pullback}\\
    \hat{\ff}(\underline{\cA}_{g\Upsilon}(Z_\delta), j_{g\Upsilon})&=
    \hat{\ff}(\underline{\cB}_{g_1}\times \underline{\cB}_{g_2}, 
    j_{g_1}\times j_{g_2})=
    \ff(\underline{\cB}_{g_1}\times \underline{\cB}_{g_2}, 
    j_{g_1}\times j_{g_2},d^\times t\times d^\times t).
    \label{eq:padic_pullback}
\end{align}

\section{Hida theory on the definite unitary group}

In this section we review the Hida theory
of algebraic modular forms on the definite unitary group $\UU(\W)$.
Results on the vertical control theorem are quoted from \cite{geraghty},
and the pairing between Hida families
is adopted from the case of definite quaternion from \cite{Hsieh21}.

To simplify notations, for an $\F$-algebra $R$ we write 
\[
    G_1(R)\coloneqq \UU(\W)(R)=\{g\in \GL_{2}(\K\otimes_\F R)\mid 
    gg^*=\id_2 \}.
\]
Given the lattice $L=\oo_\K\ee_1\oplus\oo_\K\ee_2$,
we define for each finite place $v\in\finite$ of $\F$
the open compact subgroup
\[
    K_v^0=\{g\in G_1(\F_v)\mid (L\otimes_\OF\oo_v)g=(L\otimes_\OF\oo_v)\}.
\]
Let $\fc$ and $\fs$ be as in \S\ref{sec:choice}
and write $\fn=\fc\fs\dd_\F$.
Given the choice of compact open subgroups 
$\KK(\fn)$ in the previous section and let $\xi_v\in \G(\F_v)$ be as in 
Proposition \ref{prop:Rallis_local}, we consider an open compact subgroup
$K(\fn)=\prod_vK_v$, where
$K_v\subset K_v^0$ are open compact subgroups such that 
\begin{enumerate}[label={($K$\arabic*)}]
    \item $K_v=\bar{K}_v$ and
    $\iota_{\xi_v}(K_v,1)\subset \KK_v$ when $v\nmid p\fs$.
    \item $K_v=\{g_v\in K_v^0\mid g_w\equiv \smat{1&*\\&*}\mod \fn\}$ when $v\mid \fs$.
    \label{enu:K2}
    \item $K_v=K_v^0$ for almost all $v\in\finite$ and when $v\mid p$.
\end{enumerate}
Note that by the proof of Proposition \ref{prop:xi_split}
we also have $\iota_{\xi_v}(K_v,1)\subset \KK_v$
when $v\mid \fs$.

Write $\fn_\circ=\fc\dd_\F$ when $\fs=\OK$.
Throughout the article we fix such 
an open compact subgroup $K(\fn_\circ)$. 
Possibly after shrinking $K_v$ at some $v\in\finite$
we assume $K(\fn_\circ)$ is sufficiently small in the sense that 
\[
    G_1(\A_{\F,f})=\sqcup_i G_1(\F)t_i K(\fc),\quad
    t_i^{-1}G_1(\F)t_i\cap K(\fc)=\{1\} \text{ for all } i.
\]
Then we will only consider ideals $\fs$ such that 
$K_v^0\subset K(\fn_\circ)$ for all $v\mid \fs$.
And for such $\fs$ we shall let $K(\fn)\subset K(\fn_\circ)$
be the subgroup defined by replacing components at $v\mid \fs$
by the subgroup given in \ref{enu:K2}.

As in the previous section, we have an isomorphism
$\iota_\Sigma\colon\Res_{\F/\Q}G_1(\Q_p)\cong \GL_2(\F_p(\Sigma))$ 
which identifies $K_p^0\coloneqq \prod_{v\in S_p}K_v^0$ with 
$\GL_2(\oo_p(\Sigma))$.
We define the following open compact subgroups and the double quotients
\begin{align*}
    K_1^n(\fn)&=\{(g_v)\in K\mid g_\Sigma\equiv \smat{1&*\\&*}\bmod p^n\},\quad
    C(K^n_1(\fn))\coloneqq G_1(\F)\backslash G_1(\A_{\F,f})/K^n_1(\fn);\\
    K_0^n(\fn)&=\{(g_v)\in K\mid g_\Sigma\equiv \smat{*&*\\&*}\bmod p^n\},\quad
    C(K^n_0(\fn))\coloneqq G_1(\F)\backslash G_1(\A_{\F,f})/K^n_0(\fn).
\end{align*}
When it's not necessary to specify the dependence on $\fn$,
we write $K_\bullet^n=K_\bullet^n(\fn)$ for short.

For $\wt{k}=(a_1,a_2)\in \Z[\Sigma]^2$,
let $L_{\wt{k}}=\otimes_\sigma L_{\wt{k}_\sigma}$
be the algebraic representation 
of $\Res_{\OF/\Z} \GL_2$ of lowest weight $-\wt{k}$
defined as in \ref{sec:modular_forms},
and $\lk$ be the functional defined as in Definition \ref{def:avatar}.
We also fix a lower weight vector $\vk$
such that $\lk(\vk)=1$.
We assume throughout that $\wt{k}=(k,0)$ for some $k\in \Z[\Sigma]$.
If $g_p\in \Res_{\F/\Q}G_1(\Q_p)$ we define
\[
    \rho_{\wt{k}}(g_p)=\prod_{w\in \Sigma_p}
    \prod_{\sigma\in I_w}\rho_{\wt{k}_\sigma}(g_w).
\]
We now let $\oo$ be the ring of integer of a field extension
$E/\Qp$ that contains $\sigma_p(\K)$ for all $\sigma\in \Hom(\K,\C)$
and let $R$ be a $p$-adic $\oo$-algebra.
\begin{defn}\label{def:algform}
    For $\wt{k}\in \Z[\Sigma]^2$ as above
    we let $\pMk(K^n_1,R)$ denote
    the space of algebraic modular forms of weight $\wt{k}$
    and level $K^n_\bullet$ consisting of functions 
    \[
        \pf(g)\colon G_1(\F)\backslash G_1(\A_{\F,f})\to L_{\wt{k}}(R),
        \quad \pf(gu)=\rho_{\wt{k}}(u)^{-1}\pf(g)\text{ when }
        u\in K^n_1
    \]
    where we put $\rho_{\wt{k}}(u)=\rho_{\wt{k}}(u_p)$
    if $u\in K_1^n$.
    We also write $M(K^n_1,R)\coloneqq\pMk(K^n_1,R)$
    when $\wt{k}$ is the trivial weight, in which case
    $f(g)$ is simply a $R$-valued function on $C(K_1^n)$.
\end{defn}

Let notations be as in Proposition \ref{prop:inv_Up},
we define the following Hecke operators acting on 
on $\pMk(K^n_1, R)$.
Recall that we have fixed $\fs=\mathfrak{f}\bar{\mathfrak{f}}$
for $\mathfrak{f}+\bar{\mathfrak{f}}=\OK$.
And when $v=w\bw$ and $v\mid p\fs$ we assume 
$w\in\Sigma_p$ or $w\mid\mathfrak{f}$.
\begin{itemize}
    \item When $v\in\finite$ is prime-to-p finite place that 
    splits into $w\bw$ in $\K$ and $K_v^0\subset K$, we define 
    \begin{equation}\label{eq:hecke_awayp}
    T^{(1)}_w\pf(g)=\sum_{b\in \oo_v/(\varpi_v)}
    \pf(g \smat{\varpi_v &b\\&1})+
    \pf(g \smat{1&\\&\varpi_v }),\quad
    T^{(2)}_w\pf(g)=
    \pf(g \smat{\varpi_v&\\&\varpi_v })
    \end{equation}
    \item When $v=w\bw$ splits and  $w\in \Sigma_p$, 
    for $u\in \oo_v^\times$ we define 
    $\langle u\rangle_{\wt{k}}\pf(g)=\pf(gu)$ and
    \begin{equation}\label{eq:hecke_atp}
    U^{(1)}_{\wt{k},w}\pf(g)=\wt{k}(\varpi_v,1)\sum_{b\in \oo_v/(\varpi_v)}
    \rho_{\wt{k}}(\smat{\varpi_v&b\\&1})
    \pf(g \smat{\varpi_v &b\\&1}),\quad
    U^{2}_{\wt{k},w}\pf(g)=
    \pf(g \smat{\varpi_v&\\&\varpi_v})
    \end{equation}
    \item When $v=w\bw$ splits and $w\mid \mathfrak{f}$, we define 
    $\langle u\rangle\pf(g)=\pf(gu)$ for $u\in \oo_v^\times$ and
    \begin{equation}\label{eq:hecke_ats}
    U^{(1)}_w\pf(g)=\sum_{b\in \oo_v/(\varpi_v)}
    \pf(g \smat{\varpi_v &b\\&1}),\quad
    U^{(2)}_w \pf(g)=\pf(g \smat{\varpi_v&\\&\varpi_v}).
    \end{equation}
\end{itemize}
When $\rho_{\wt{k}}$ is the trivial representation we will 
write $U_w^{(i)}$ for $U_{\wt{k},w}^{(i)}$.

Note that for $w\in\Sigma_p$ and $u\in 1+p^n\oo_v$,
by definition we have 
$\langle u\rangle_{\wt{k}}\pf(g)=\wt{k}(u,u)\pf(g)$
if $\pf(g)\in M_{\wt{k}}(K^n_1,R)$.
Similarly, for $v=w\bw$ with $v\mid\fs$
the operator $\langle u\rangle$ defines 
an action of $\Delta_w(\fs)\coloneqq (\oo_w/\fs)^\times$ on 
$M_{\wt{k}}(K^n_1,R)$.
Therefore if we define the subgroups
$T(p^n)=\prod_{w\in\Sigma_p}T(w^n)$ for $n\geq 0$,
where $T(w^n)=1+p^n\oo_w$ if $n\geq 1$
and $T(w^0)=\oo_w^\times$, then $M_{\wt{k}}(K^n_1,R)$ is naturally a
module of the Iwasawa algebras
\begin{equation}\label{def:lambda}
\Lambda\coloneqq \oo\llbracket T(p^1)\rrbracket\quad
\Lambda^+=\oo\llbracket T(p^0)\times \Delta_\fs\rrbracket\cong
\Lambda[\Delta_p\times\Delta_\fs],
\end{equation}
where $\Delta_p=\prod_{w\in\Sigma_p}\Delta_w$
with $\Delta_w\coloneqq (\oo_w/\varpi_w)^\times$ and 
$\Delta_\fs=\prod_{w\mid \mathfrak{f}}\Delta_w(\fs)$.

Since $K^n_1$ is sufficiently small by assumption,
we have the isomorphism
\[
    \pMk(K^n_1,R)\to 
    \bigoplus_{[g]\in C(K^n_1)}L_{\wt{k}}(R),\quad
    \pf \mapsto (\pf(g))_{[g]\in C(K^n_1)}
\]
and consequently the ordinary projector
$\ordp\coloneqq\lim_{r\to\infty}(\prod_{w\in\Sigma_p}
U^{(1)}_{\wt{k},w}U^{(2)}_{\wt{k},w})^{r!}$
converges to an idempotent in 
the finite-free $R$-algebra $\End_R(\pMk(K^n_1,R))$.
\begin{defn}
We define the subspace of ordinary modular forms as
\[
    \pMk^{\ord}(K^n_1, R)=\ordp\pMk^{\ord}(K^n_1, R)=
    \{\pf\in \pMk(K^n_1,R)\mid 
    \ordp\pf=\pf\}.
\]
\end{defn}

\begin{rem}
    Our normalization of the Hecke operators $\langle u\rangle_{\wt{k}}$
    is different from that of \cite{geraghty}.
\end{rem}

\subsection{p-adic modular forms}
Recall that for a p-adic ring $R$ we define
$R_m=R/p^mR$ when $m\in \Z_{\geq 1}$.
\begin{defn}
    Let $V_{n,m}=\pM(K^n_1,\oo_m)$ for $m,n\geq 1$.
    The space of $p$-adic modular forms and 
    the subspace of ordinary forms are defined respectively as
    \[
        V_p(G_1,K)=\varprojlim_m\varinjlim_n V_{n,m},\quad
        V^{\ord}_p(G_1,K)= \ordp V_p(G_1,K)
        =\varprojlim_m\varinjlim_n V_{n,m}^{\ord}
    \]
    Since the Hecke actions defined above on $V_{n,m}$ are 
    compatibly among the limits, we may extend the Hecke actions to
    $V_p(G_1,K)$ and $V^{\ord}_p(G_1,K)$.
    In particular, each spaces admit a structure of $\Lambda^+$-modules.
\end{defn}

Alternatively, let $\oo[C(K^n_1)]$ be the finite free $\oo$-module
generated on $C(K^n_1)$.
Then the evaluation map defines an isomorphism
$V_{n,m}\cong \Hom_\oo(\oo[C(K^n_1)],\oo_m)$ of $\oo$-modules,
which endow $\oo[C(K^n_1)]$ with Hecke actions
and the structure of a $\Lambda^+$-module.
Define $\oo\llbracket C(K)\rrbracket=\varprojlim_n \oo[C(K^n_1)]$.
By passing to the limit we obtain an isomorphism
\begin{equation*}
    V_p(G_1,K)\cong 
    \varprojlim_m\varinjlim_n \Hom_{\oo}(\oo[C(K^n_1)],\oo_m)
    \cong \varprojlim_m
    \Hom_{\oo}^{\cts}(\oo\llbracket C(K)\rrbracket, \oo_m)
    \cong \Hom_{\oo}^{\cts}(\oo\llbracket C(K)\rrbracket, \oo),
\end{equation*}
where we say a homomorphism 
$\phi\in\Hom_{\oo}(\oo\llbracket C(K)\rrbracket, \oo)$
is continuous if for any $m\geq 1$ there exists 
an integer $n$ such that $\phi\bmod p^m$
factors through $\oo[C(K^n_1)]$.
Consequently we also have the isomorphism
\begin{equation}\label{eq:MM}
    V_p^{\ord}(G_1,K)=
    eV_p(G_1,K)\cong 
    \Hom_{\oo}^{\cts}(e\oo\llbracket C(K)\rrbracket, \oo).
\end{equation}

Since we have assumed $\wt{k}=(k,0)$, 
the action of $\rho_{\wt{k}}(g)$ on $v_{-\wt{k}}$
is trivial modulo $p^m$ when $g\in K^m_1$.
Thus for $\pf\in \pMk(K^n_1,R)$ we have
$\pf_m(g)\coloneqq \lk(\pf(g))\bmod p^m\in V_{m,m}$.
We can therefore define an injective homomorphism
\[
\bbeta_{\wt{k}}\colon \pMk(K^n_1,R)\to V_p(G_1,K)\quad
\beta_{\wt{k}}(\pf)=\varprojlim_m \pf_m
\]
Note that when $w\in\Sigma_p$, the homomorphism 
is equivariant with respect to $U_{\wt{k},w}$
since $\rho_{\wt{k}}(\smat{\varpi_v&b\\&1})v_{-\wt{k}}
=\wt{k}(\varpi_v,1)^{-1}v_{-\wt{k}}$.
It is also straightforward to check that $\bbeta_{\wt{k}}$
is equivariant with respect to 
all the other Hecke operators defined above as well.

The proposition below is a consequence of the control theorem \cite{geraghty}, which describes the image of $\bbeta_{\wt{k}}$.

\begin{prop}\label{prop:control}
    Let $\fa_{\wt{k},n}\subset \Lambda$ be the ideal generated by 
    $\{\langle u\rangle-\wt{k}(u,u)\mid u\in T(p^n)\}$, then
    \begin{equation*}
        \bbeta_{\wt{k}}(\pMk^{\ord}(K^n_1,R))=
        V_p^{\ord}(G,K)[\mathfrak{a}_{\underline{k},n}].
    \end{equation*}
\end{prop}
\begin{proof}
That $\bbeta_{\wt{k}}(\pMk^{\ord}(K^n_1,R))\subset V_p^{\ord}(G,K)[\mathfrak{a}_{\underline{k},n}]$
follows from the equivariance of the Hecke operators.
On the other hand, suppose $\varprojlim_m\pf_m\in V_p^{\ord}(G,K)$,
where $\pf_m\in V_{n_m,m}$ for $n_m$ depending on $m$.
By \cite[Proposition 2.22]{geraghty}
each $\pf_m$ has a unique lift $\pf'_m\in \pMk^{\ord}(K_1^{n_m},R_m)$.
Suppose $\lim_m\pf_m\in V_p(G_1,K)[\fa_{\wt{k},n}]$,
then by definition
\[
   \lk(\pf'_m(g))=\pf_m(g)=\wt{k}(u,u)^{-1}\pf_m(gu)=\lk(\rho_{\wt{k}}(u\id_2)\pf'_m(gu)),\quad
   u\in T(p^n).
\]
Therefore $\pf'_m=[u].\pf'_m$ by the uniqueness
of $\pf'_m\in \pMk^{\ord}(K^{n}_1,R_m)$ from \cite[Lemma 2.19]{geraghty}.
($\pf'_m$ is of level $U(\fl^{n,n'})$ in the notation of loc.cit).
Then $\varprojlim_m\pf'_m$ is an inverse system which gives the preimage.
\end{proof}

\subsection{Pairings}\label{sec:pairing}

Recall that $\rho^{\wt{k}}(g)=\rho_{\wt{k}}(g^{-\intercal})$
defines an algebraic representation of highest weight $\wt{k}$.
In the case when $\wt{k}=(k, 0)$,
there exists a non-degenerate pairing
\[
    \langle\cdot,\cdot\rangle_k\colon 
    \Lk\times \Lk\to R[\frac{1}{k!}] \text{ such that }
    \langle v,v'\rangle_k=(-1)^k\langle v',v\rangle_k \text{ and }
    \langle \rho^{\wt{k}}(g)v,v'\rangle=
    \langle v,\rho^{\wt{k}}(g')v'\rangle,
\]
where $g'=\smat{d&-b\\-c&a}$
if $g=\smat{a&b\\c&d}$ and $k!=\prod_\sigma (k_\sigma)!$.
An explicit description can be found for example in \cite[\S 3]{Hsieh21},
where $\rho^{\wt{k}_\sigma}$ is realized on 
the space of homogeneous polynomials of degree $k_\sigma$ and 
\[
    \langle X^iY^{k_\sigma-i}, X^jY^{k_\sigma-j}\rangle_{k_\sigma}
    =(-1)^i\delta_{i,k_\sigma-j}\binom{k_\sigma}{i}^{-1}.
\]
Under the realization above we can identify 
a lowest weight vector $\vk$ of $\rho_{\wt{k}}$ with 
$(X^{k_\sigma})_{\sigma\in\Sigma}$. 
Then the twisted pairing
$(v,v')_k\coloneqq \langle\rho^{\wt{k}}(\smat{&-1\\1&})v,v'\rangle_k$
satisfies
\begin{equation}\label{eq:pairing_transpose}
    (\rho^{\wt{k}}(g)v,v')_k=(v,\rho^{\wt{k}}(g^\intercal)v')_k,\quad
    (\vk,\vk)_{k}=1.
\end{equation}

For $v\mid p\fs$ let $\tau_v\in G_1(\F_v)$ 
be as in Proposition \ref{prop:Rallis_local}.
We then define 
\[
\tau(\fs)=\prod_{\substack{v=w\bw\\w\mid \mathfrak{f}}}
\tau_v^{\val_w(\fs)}  \quad
\tau_n=\prod_{\substack{v=w\bw\\w\in\Sigma_p }}
\tau_v^{\val_w(p^n)}
\]
and put $\tau_n(\fs)=\tau(\fs)\tau_n$.
We note that $\tau_n(\fs)^{-1}\bar{u}\tau_n(\fs)\in K^n_1(\fn)$
if $u\in K^n_1(\fn)$.

Write $K=K(\fn)$ and let $\alpha\colon \Lambda^+\to R^\times$ 
be a homomorphism of $\oo$-algebras
that factors through the ideal $\fa_{\wt{k},n}$
introduced in Proposition \ref{prop:control} for some $\wt{k}$ and $n$.
We then define 
\[
\pMk(K^n_1,\alpha,R)=\{\pf\in \pMk(K^n_1,R)\mid 
\langle u\rangle \pf=\alpha(u)\pf \text{ for } u\in \Lambda^+\}.
\]
\begin{defn}
For a homomorphism $\alpha$ as above, we define a pairing on 
$\pMk(K^n_1,\alpha,R)$ by
\begin{equation}\label{def:pairings}
    (\pf_1,\pf_2)_n\coloneqq
    \wt{k}(p^n,1)\sum_{C(K^n_0(\fn))}
    (\pf_1(g),
    \rho_{\wt{k}}(\tau_n)\pf_2
    (\bar{g}\tau_n(\fs)))_k\in R[\frac{1}{k!}]
\end{equation}
for $\pf_1,\pf_2\in\pMk(K^n_1,\alpha,R)$.
To check that it is well-defined, 
observe that for $u\in K^n_1$ and
$u'=\tau_n(\fs)^{-1}\bar{u}\tau_n(\fs)$
\begin{multline*}
    (\pf_1(gu),
    \rho_{\wt{k}}(\tau_n)\pf_2
    (\overline{gu}\tau_n(\fs)))_k=
    (\rho_{\wt{k}}(u^{-1})\pf_1(g),
    \rho_{\wt{k}}(\tau_n)\rho_{\wt{k}}(u'^{-1})\pf_2
    (\bar{g}\tau_n(\fs)))_k\\=
    (\pf_1(g),\rho_{\wt{k}}(u^{-\intercal}_\Sigma)
    \rho_{\wt{k}}(u^{\intercal}_\Sigma)\rho_{\wt{k}}(\tau_n)\pf_2
    (\bar{g}\tau_n(\fs)))_k
    =(\pf_1(g),
    \rho_{\wt{k}}(\tau_n)\pf_2
    (\bar{g}\tau_n(\fs)))_k.
\end{multline*}
And the summand is clearly invariant by the central elements
in $K^1_0(\fn)$ by the definition of 
$M_{\wt{k}}(K^n_1,\alpha,R)$.
\end{defn}

\begin{lem}
    Suppose $n\geq\val_p(k!)$ and $\pf_1,\pf_2\in \pMk(K^n_1,\alpha,R)$,
    then $(\pf_1, \pf_2)_{2n}\in R$ when we view $\pf_i$ as elements in
    $\pMk(K^{2n}_1,\alpha,R)$ and 
    \[
        (\pf_1, \pf_2)_{2n}\equiv
        (\lk(\pf_1), \lk(\pf_2))_{2n}\coloneqq
        \sum_{K^{2n}_0(\fn)}
        \lk(\pf_1)(g)\cdot 
        \lk(\pf_2)(\bar{g}\tau_{2n}(\fn))\mod p^n.
    \]
\end{lem}
\begin{proof}
    Let $\{v_{{\wt{k}'}}\}_{{\wt{k}'}\geq -\wt{k}}$ be an $\oo$-basis 
    of $L_{\wt{k}}$ extending $\vk$ that consists of weight vectors.
    Decompose  
    $\pf_i(g)=\sum_{{\wt{k}'}\geq -\wt{k}}\pf_i(g;{\wt{k}'})v_{{\wt{k}'}}$
    with respect to the basis, then
    \[
        \wt{k}(p^{2n},1)\cdot 
        \rho_{\wt{k}}(\tau_{2n})\pf_2
        (\bar{g}\tau_{2n}(\fn))
        =\sum_{{\wt{k}'}\geq -\wt{k}}
        (\wt{k}'+\wt{k})(p^{2n},1)\cdot 
        \pf_2(g\tau_{2n}(\fn);{\wt{k}'})v_{{\wt{k}'}}.
    \]
    Since $\ord_p((\wt{k}'+\wt{k})(p^{2n},1))\geq 2n$
    if $\wt{k}'\neq -\wt{k}$ and by assumption
    $p^{2n}(v,v')_k\in p^n\oo$ for $v,v'\in L_k$,
    we see that 
    \[
        \wt{k}(p^{2n},1)\cdot 
        (\pf_1(g), \rho_{\wt{k}}(\tau_{2n})\pf_2(\bar{g}\tau_{2n}(\fn)))_k\equiv
        \pf_1(g;-\wt{k})\cdot \pf_2(\bar{g}\tau_{2n}(\fn);-\wt{k})=
        \lk(\pf_1)(g)\cdot \lk(\pf_2)(\bar{g}\tau_{2n}(\fn))\mod p^n
    \]
    and the lemma follow.
\end{proof}

We now show that after twisted by Hecke actions,
the pairing \eqref{def:pairings}
actually takes values in $R$ when restricted to
the subspace of ordinary forms. 
In particular, when $R=\oo$ we can define a pairing on
\[
    V_p^{\ord}(G_1,K,\alpha)=
    \{f\in V_p^{\ord}(G_1,K,\alpha)\mid 
    \langle u\rangle f=\alpha(u)f \text{ for } u\in \Lambda^+\}.
\]
To simplify the notation, we introduce the Hecke operator
\begin{equation}\label{eq:Up}
    U_{\wt{k},p}\pf(g)=\wt{k}(p,1)\sum_{b\in \oo_p/(p)}
    \rho_{\wt{k}}(\smat{p&b\\&1})
    \pf(g \smat{p &b\\&1}),
\end{equation}
in which the matrices are identified with elements in $K$
via the isomorphism $\iota_\Sigma\colon K_p\cong \GL_2(\oo_p(\Sigma))$.
It can be checked directly that 
$U_{\wt{k},p}=\prod_{w\in\Sigma_p}(U_{w,\wt{k}}^{(1)})^{\ord_w(p)}$.
And we write $U_p=U_{\wt{k},p}$ 
when $\rho_{\wt{k}}$ is the trivial representation.

\begin{prop}\label{prop:product}
    For $\alpha\colon \Lambda^+\to R$ as above and 
    $\pf_1,\pf_2\in\pMk^{\ord}(K^n_1,\alpha,R)$ the pairing 
    \[
       (\pf_1,U_{\wt{k},p}^{-n}\pf_2)_n
    \]
    takes values in $R$ and is independent of $n$.
    In particular, when $R=\oo$ the pairing extends to 
    a non-degenerate symmetric pairing
    \[
        \B\colon V^{\ord}_p(G_1,K,\alpha)\times 
        V^{\ord}_p(G_1,K,\alpha)\to\oo
    \]
    such that  $\B(\bbeta_{\wt{k}}(\pf_1),\bbeta_{\wt{k}}(\pf_2))=
    (\pf_1,U_{\wt{k},p}^{-n}\pf_2)_n$.
\end{prop}
\begin{proof}
We first show that $(\pf_1,U_{\wt{k},p}^{-n}\pf_2)_n$ is independent of $n$ for 
$\pf_1,\pf_2\in M_{\wt{k}}^{\ord}(K^n_1,\alpha,R)$.
Since $K^n_0=\bigsqcup_{b\in \oo_p/(p)} \smat{1&\\p^nb&1}K^{n+1}_0$, 
\begin{multline*}
        (\pf_1,\pf_2)_{n+1}\coloneqq
        \wt{k}(p^{n+1},1)
        \sum_{C(K^n_0)}\sum_{b\in \oo_p/(p)}
        (\pf_1(g\smat{1&\\p^nb&1}),
        \rho_{\wt{k}}(\tau_{n+1})
        \pf_2(\bar{g}\smat{1&-p^nb\\&1}\tau_{n+1}(\fs)))_k\\=
        \wt{k}(p^{n+1},1)
        \sum_{C(K^n_0)}\sum_{b\in \oo_p/(p)}
        (\pf_1(g),
        \rho_{\wt{k}}(\smat{1&-p^nb\\&1})
        \rho_{\wt{k}}(\tau_{n+1}))
        \pf_2(\bar{g}\smat{1&-p^nb\\&1}\tau_{n+1}(\fs)))_k\\=
        \wt{k}(p^n,1)
        \sum_{C(K^n_0)}
        (\pf_1(g),
        \rho_{\wt{k}}(\tau_{n})
        \wt{k}(p,1)
        \sum_{b\in \oo_p/(p)}
        \rho_{\wt{k}}(\smat{p&-b\\&1})
        \pf_2(\bar{g}\tau_n(\fs)\smat{p&-b\\&1}))_k\\=
          \wt{k}(p^n,1)
          \sum_{C(K^n_0)}
         (\pf_1(g),
        \rho_{\wt{k}}(\tau_{n})
          U_p\pf_2(\bar{g}\tau_n(\fs)))_k=
          (\pf_1,U_p\pf_2)_n.
\end{multline*}
Therefore, suppose $\ppf_1,\ppf_2\in V_p^{\ord}(G,K,\alpha)$
are given by the inverse limits
$\varprojlim_m\pf_{m,i}$ where 
where $\pf_{m,i}\in M^{\ord}(K^m_1,\alpha,\oo_m)$, we then define
\[
    \B(f_1,f_2)=\varprojlim_m
    (\pf_{1,m}, U_p^{-m}\pf_{2,m})_{m}\in \oo
\]
And when $\ppf_i=\bbeta_{\wt{k}}(\pf_i)$ for
$\pf_i\in M^{\ord}_{\wt{k}}(K^n_1,\alpha,\oo)$
we may take $\pf_{i,m}=\lk(\pf_i)\bmod p^m$, 
then by the previous lemma
\[
    (\pf_1,U_{\wt{k},p}^{-n}\pf_2)_{n}=
    (\pf_1,U_{\wt{k},p}^{-m}\pf_2)_{m}\equiv (\pf_{1,m},U_p^{-m}\pf_{2,m})_{m}
    \mod p^m
\]
for all sufficiently large $m$ and therefore 
$\B(\bbeta_{\wt{k}}(\pf_1),\bbeta_{\wt{k}}(\pf_2))
=(\pf_1, U_{\wt{k},p}^{-n}\pf_2)_{n}$.

To show that $\B$ is symmetric and non-degenerate,
it is sufficient to show that the pairing is symmetric and perfect
on the finite free $\oo_m$-module $M(K^n_1,\alpha,\oo_m)$.
We first observe that the untwisted pairing \eqref{def:pairings}
\[
    (\pf_{1,m},\pf_{2,m})_n=
    \sum_{C(K^n_0(\fn))}
    \pf_{1,m}(g)\pf_{2,m}(\bar{g}\tau_n(\fs))
\]
is obviously a perfect pairing and
is indeed symmetric since $g\to \bar{g}\tau_n(\fs)$ is an involution.
\[
    (\pf_{1,m},U_p^{-n}\pf_{2,m})_n=
    (U_p^{-n}\pf_{2,m},U_p^{-n}\pf_{1,m})_{2n}=
    (\pf_{2,m},U_{\wt{k},p}^{-n}\pf_{1,m})_n \mod p^n
\]
that the twisted pairing is symmetric. 
To show that it is perfect, observe that the equation above implies
$(\pf_{1,m},\ordp\pf_{2,m})_n=(\ordp\pf_{1,m},\pf_{2,m})_n$
for $\pf_{i,m}\in M(K_1^n,\alpha, \oo_m)$.
Thus given $\pf'_{2,m}\in M(K^n_1,\alpha,\oo_m)$,
if we put $\pf_{2,m}=U_p^n\ordp \pf'_{2,m}\in 
M^{\ord}(K^n_1,\alpha,\oo_m)$, 
then for any $\pf_{1,m}\in M^{\ord}(K^n_1,\alpha,\oo_m)$ we have
\[
    (\pf_{1,m},\pf'_{2,m})_n=(\ordp\pf_{1,m},\pf'_{2,m})_n=
    (\pf_{1,m},\ordp\pf'_{2,m})_n=
    (\pf_{1,m},U_p^{-n}\pf_{2,m})_n.
\]
Therefore $(*,\pf'_{2,m})_n$ and $(*,U_p^{-n}\pf_{2,m})_n$ 
induces the same map on $M^{\ord}(K^n_1,\alpha,\oo)$.
Therefore the twisted pairing is indeed perfect,
from which we can conclude that the pairing $\B$ is perfect as well.
\end{proof}

\subsection{Hida families}

\begin{defn}\label{hidaform}
Let $\I$ be a local complete Noetherian ring
that is finite over $\Lambda^+$ and flat over $\oo$.
We define the space of $\I$-adic Hida families of level $K$ by
\begin{equation}\label{eq:Hida_family}
    M^{\ord}(K,\I)=
    \{\euF\in V_p^{\ord}(G_1,K)\widehat{\otimes}_\oo \I\mid 
    (\langle u\rangle\otimes\id)\euF=
    (\id\otimes\langle u\rangle)\euF\, \text{ for }
    u\in \Lambda^+\}.
\end{equation}
\end{defn}

\begin{prop}\label{Hidafam}
    There exists an isomorphism  of $\I$-modules
    \[
        M^{\ord}(K,\I)\cong 
        \Hom_{\Lambda^+}(e\oo\llbracket C(K)\rrbracket, \I).
    \]
\end{prop}
\begin{proof}
By the isomorphism \eqref{eq:MM} we have
$V_p^{\ord}(G_1,K)\widehat{\otimes}_\oo\I
\cong \Hom^{\cts}_\oo(e\oo\llbracket C(K)\rrbracket,\I)$.
Then the subspace on which the actions
$(\langle u\rangle\otimes\id)$ and
$(\id\otimes\langle u\rangle)$
coincide is precisely
$\Hom^{\cts}_{\Lambda^+}(e\oo\llbracket C(K)\rrbracket,\I)$.
\end{proof}

For $n\geq 1$,
note that the ideal $(p^n,\fa_{\wt{k}, n})$ is independent of $\wt{k}$
and let $\alpha_n\colon\Lambda^+\to R_n\coloneqq\I/(p^n,\fa_{\wt{k},n})\I$
be the canonical projection.
If $\euF\in M^{\ord}(K,\I)
\cong \Hom^{\cts}_\oo(e\oo\llbracket C(K)\rrbracket,\I)$,
we let $\euF_n\in \Hom_{\Lambda^+}(e\oo\llbracket C(K)\rrbracket,R_n)$
denote the composition with the projection.
By the control theorem we may identify $\euF_n$
as an algebraic modular form in 
\[
    M(K^n_1, R_n,\alpha_n).
\]
By abuse of notation, we now define a pairing 
$\B\colon M^{\ord}(K,\I)\times M^{\ord}(K,\I)\to \I$ by 
\begin{equation}\label{def:Hida_pairing}
    \B(\euF_1,\euF_2)=\varprojlim_n
    (\euF_{1,n},U_p^{-n}\euF_{2,n})_n\in\varprojlim_n R_n\cong \I.
\end{equation}

\begin{prop}\label{prop:Hida_pairing}
    Let $\alpha\colon \I\to \oo$ be a homomorphism
    of $\Lambda^+$-algebras such that 
    $\alpha\vert_{\Lambda^+}$ factors through $\fa_{\wt{k},n}$
    for some $\wt{k}$ and $n$.
    If $\euF\in M^{\ord}(K,\I)$, let 
    $\euF_\alpha\in \Hom_\oo^{\cts}(e\oo\llbracket C(K)\rrbracket,\oo)$
    denote the composition, which we identify with an element in 
    $V_p(G_1,K,\alpha)$. Then for $\euF_1,\euF_2\in M^{\ord}(K,\I)$ we have
    \[
    \alpha(\B(\euF_1,\euF_2))=\B(\euF_{1,\alpha},\euF_{2,\alpha}).
    \]
\end{prop}
\begin{proof}
Write $\ppf_i=\euF_{i,\alpha}$ for $i=1,2$, then we may take
$\ppf_i=\varprojlim_m \pf_{i,m}$ for 
\[
    \pf_{i,m}=\alpha\circ \euF_i \bmod p^m=
    \alpha\circ \euF_{i,m}
\]
when $m\geq n$. 
Then the proposition follows directly from the definition.
\end{proof}

\subsubsection{Between different tame levels}

Given $\fn=\fc\fs\dd_\F$ and $K(\fn)$, let $\ell\in\finite$ 
be a finite place that splits into $\ell=\fl\bar{\fl}$ and is 
prime to $p\fn$. Assume further that $K_\ell^0\subset K(\fn)$
and let $K(\fn\ell)\subset K(\fn)$ 
be the subgroup obtained by replacing the component at $\ell$
by the subgroup $K_\ell$ as in \ref{enu:K2}. Let 
\[
    \phi\colon \Lambda^+_{\fs\ell}\to \Lambda^+_{\fs}
\]
be the canonical projection between the $\Lambda$-algebras
defined as in \eqref{def:lambda}.
View $M_{\wt{k}}(K^n(\fn),R)$ as an $\Lambda_{\fs\ell}$-module via $\phi$,
we can define the $\Lambda_{\fs\ell}$-module homomorphisms 
\[
    \id, V_\fl\colon 
    M_{\wt{k}}(K^n(\fn),R)\to M_{\wt{k}}(K^n(\fn\ell),R).
\]
Here $\id$ is the usual inclusion and 
$(V_\fl \pf)(g)\coloneqq f(g\smat{1&\\&\varpi_\ell})$,
with the matrix identified as an element in $G_1$ via $\iota_\fl$.
If $\I$ is a finite $\Lambda^+_\fs$-algebra that is flat over $\oo$,
viewed also as an $\Lambda^+_{\fs\ell}$-algebra via $\phi$,
the homomorphisms extends naturally to the spaces of Hida families
\[
    \id, V_\fl\colon 
    M^{\ord}(K^n(\fn),\I)\to M^{\ord}(K^n(\fn\ell),\I).
\]

\begin{prop}\label{prop:pair_at_deff_level}
	Let $\B_{\fn}$ and $\B_{\fn\ell}$ denote respectively 
    the $\I$-valued pairings defined by \eqref{def:Hida_pairing}
    on the spaces
    $M^{\ord}(K^n(\fn),\I)$ and $M^{\ord}(K^n(\fn\ell),\I)$, 
    then we have
	\begin{align*}
	&\B_{\fn\ell}(\euF_1,\euF_2)=
	\B_{\fn}(\euF_1,T_{\fl}^{(1)}\cdot \euF_2)\\
	&\B_{\fn\ell}(\euF_1,V_{\fl}\cdot\euF_2)=
	(q_\ell+1) \B_{\fn}(\euF_1,
	T_{\fl}^{(2)}\cdot\euF_2)\\
	&\B_{\fn\ell}
	(V_{\fl}\cdot \euF_1,V_{\fl}\cdot\euF_2)=
	\B_{\fn} (T_{\fl}^{(1)}\cdot\euF_1,
	T_{\fl}^{(2)}\cdot \euF_2)
	\end{align*}
\end{prop}
\begin{proof}
    It suffices to check the equalities over finite levels
    and with trivial weight.
    Let $K_0(\ell)=\{k\in\GL_2(\oo_\ell)\mid 
    k\bmod \varpi_\ell\in B_2(\oo_\ell)\}$
    Pick $\smat{&1\\1&}, \smat{1&\\b&1}, b\in \oo_\ell/(\varpi_\ell)$
    as a set of representatives for $\GL_2(\oo_\ell)/K_0(\ell)$.
    Identify $k$ with elements in $G_1(\F_\ell)$ via $\iota_\fl$,
    then $\iota_\fl(\bar{k})$ goes through 
    \[
        \smat{&1\\1&}, \smat{1&b\\&1}, b\in \oo_\ell/(\varpi_\ell).
    \]
    Let $f_1,f_2\in M^{\ord}(K^n_1(\fn),R)$, then the first equality
    follows from 
	\begin{multline*}
    \sum_{g\in C(K^n_0(\fn\ell))}
	f_1(g)f_2(\bar{g}\tau_n(\fs))=
	\sum_{g\in C(K^n_0(\fn\ell))}
	\sum_{k\in \GL_2(\oo_\ell)/K_0(\oo_\ell)}
	f_1(gk)f_2(\bar{g}\bar{k}\tau_n(\fs\ell))\\=
	\sum_{g} f_1(g)\cdot 
    \left[
    \sum_{b\in\oo_\ell/(\varpi_\ell)}
	f_2(\bar{g}\smat{\varpi_\ell&b\\&1}\tau_n(\fs))+
	f_2(\bar{g}\smat{1&\\&\varpi_\ell}\tau_n(\fs)\smat{&1\\1&})
    \right]=
	\sum_{g\in C(K^n_0(\fn))}f_1(g)
	(T_\fl^{(1)}f_2)(\bar{g}\tau_n(\fs)).
	\end{multline*}
    Similarly, the second equality follows from 
	\begin{multline*}
	\sum_{g\in C(K^n_0(\fn\ell))}
	f_1(g)(V_\fl f_2)(\bar{g}\tau_n(\fs\ell))=
	\sum_{g\in C(K^n_0(\fn))}
	\sum_{k\in \GL_2(\oo_\ell)/K_0(\ell)}
	f_1(gk)(V_\fl f_2)(\bar{g}\bar{k}\tau_n(\fs\ell))\\=
	\sum_{g} f_1(g)\sum_{k}
	f_2(\bar{g}\bar{k}\smat{\varpi_\ell&\\&\varpi_\ell}\tau_n(\fs))=
	(q_\ell+1)
	\sum_{g\in C(K^n_0(\fn))} f_1(g)
	(T_\fl^{(2)}f_2)(\bar{g}\tau_n(\fs)).
	\end{multline*}
	The last equation also follows similarly
	from the above computation.
\end{proof}

\begin{rem}
    By the first equality in the Proposition above
    and the proof of Proposition \ref{prop:product},
    we see that the pairing $\B$ is Hecke equivariant
    in the sense that 
    \[
         \B(T\cdot \euF_1,\euF_2)
         =\B(\euF_1,T\cdot \euF_2)
    \]
    if $T$ is one of the operators from 
    \eqref{eq:hecke_awayp},
    \eqref{eq:hecke_atp}, or
    \eqref{eq:hecke_ats}.
\end{rem}

\section{Construction of Hida families}

Let $\fs$ be as in \S\ref{sec:choice}.
In particular we have $\fs=\mathfrak{f}\bar{\mathfrak{f}}$ with
$\mathfrak{f}+\bar{\mathfrak{f}}=\OK$.
Our goal in this section is to interpolate the pull-back theta lifts
of $\fs$-admissible Hecke characters $\eta$.
We first introduce the following notation.

Let $Cl(\fa)$ denote the $p$-part of the ray class group over $\K$
of an ideal $\fa\subset\OK$. We then define 
\[
    \fG_{\fs}=\varprojlim_n Cl(p^n\fs),
\]
which is isomorphic to the Galois group over $\K$
of the maximal abelian pro-$p$ extension over $\K$
that is unramified away $p\fs$.
We normalize the reciprocity map $\Art\colon \A_{\K,f}\to \fG_\fs$
so that a uniformizer $\varpi_w$ at $w\nmid p\fs$ is sent to 
the geometric Frobenius $\Fr_w$.
If $\eta$ is admissible, then we may identify 
the p-adic avatar $\hat{\eta}$ with a character of $\fG_\fs$
via the reciprocity map. 
We let $\fX_\fs^+$ denote the set of such $p$-adic characters
that come from an admissible Hecke character,
which is a dense subset of the space of continuous
$\overline{\Z}_p$-valued functions on $\fG_\fs$.

Since $p^n\fs=\overline{p^n\fs}$, the Galois group $\Gal(\K/\F)$
acts on $\fG_\fs$ and we define the subgroups
\[
    \fG^{\pm}_\fs=\{\gamma\in\fG_\fs\mid \gamma^c=\gamma^{\pm}\}
\]
and the quotient $\fG_\fs^a\coloneqq \fG_\fs/\fG_\fs^+$,
which we call the maximal anticyclotomic quotient of $\fG_\fs$.
Since $\Art(\A_{\F,f}^\times)\subset \fG_\fs^+$, 
the group homomorphism 
$\A_{\K,f}^\times\to \fG_\fs\to \fG_\fs^a$ factors through 
$\A_{\K,f}^1$ and fits into a commutative diagram
\begin{equation}\label{eq:classgps}
    \begin{tikzcd}
    \A_{\K,f}^\times\arrow[d,"\Art"]
    \arrow[r,"z\mapsto z/\bar{z}"]&\A_{\K,f}^1\arrow[d,"\Art^a"]
    \arrow[r,hookrightarrow]& \A_{\K,f}^\times\arrow[d,"\Art"]\\ 
    \fG_\fs\arrow[r] & \fG_\fs^a \arrow[r,"\gamma\to \gamma/\gamma^c"]
    & \fG_\fs
    \end{tikzcd}
\end{equation}
which identifies the image of $\fG_\fs^a$ with the subgroup $\fG_\fs^-$.
We note that if $v\in\finite$ splits into $w\bw$
and $h\in \K_v^1$, then by definition
$\Art^a(h)\in\fG_\fs^a$ is equal to the image of $\Art(h_w)\in\fG_\fs$
for $h_w=\iota_w(h)\in K_w$.

\subsection{Family of theta lifts}

For $\fs$ as above, let 
$\eta$ be an $\fs$-admissible Hecke character of infinity type $k\Sigma$
and $\Phi_v$ be as in \S\ref{sec:choice}.
Define $\Phi_f=\otimes_{v\in \finite}\Phi_v$
and $\Phi_f'=\omega^\square(\xi^p,1)\Phi_f$
for $\xi\coloneqq\prod_{v\in\finite\setminus S_p}\xi_v$
from Proposition \ref{prop:period_local}.
Then given $P=\otimes_\sigma P_\sigma\in L_{k\Sigma}\coloneqq\otimes_\sigma L_k$
and $Z=(Z_\sigma)\in\X$ we define
\begin{equation}\label{def:theta_eta}
\theta^\square_{P,Z}(\eta)=\theta^\square_{\Phi_{P,Z}}(\eta)
\text{ for }
\Phi_{P,Z}=\otimes_{\sigma\in\Sigma}
P_\sigma(z)e^{2\pi i\mtr(Z_\sigma z^*z)}\cdot \Phi_f',
\end{equation}
which we extends linearly to all $P\in L_{k\Sigma}$.
We simply write $\theta^\square_{P}(\eta)$ when $Z=\ii$.
We then fix an open compact subgroup $\KK=\KK(\fn)$ as in \S 5
so that $\theta^\square_{P,Z}(\eta)$ is invariant by $\KK^n_1$
for some $n$.

Let $\rho=\rho_{k\Sigma}$ be the tensor product representation
of the representations $\rho_k$ from Definition \ref{def:weight},
which defines a representation of 
$\GL_2(\C(\Sigma))\times \GL_2(\C(\Sigma))$ on $L_{k\Sigma}$.
Then $\rho$ is an algebraic representation of lowest weight
$-\bwt{k}$ where $\bwt{k}=(k\Sigma, 0, \Sigma,\Sigma)$.
We identify the contragradient representation on 
the dual space $L_{k\Sigma}^*$ with $\rho^{\bwt{k}}(\C)$.
Then  when we identify the underlying spaces of 
$\rho_{\bwt{k}}$ and $\rho^{\bwt{k}}$,
the functional $\blk$ from Definition \ref{def:avatar}
is given by evaluating at $P_1(z)=z_1^{k\Sigma}\in L_{k\Sigma}$.

\begin{lem}\label{lem:cycle}
    The theta lift $\theta^\square_{P,Z}(\eta)$
    has the central character $\chi_\circ^2\eta$ and satisfies
    \[
        \theta^\square_{P,Z}(\eta)(gg_\infty)=
        \nu(g_\infty)^{(k+2)\Sigma}
        \theta^\square_{\rho(J(g_\infty, Z))P,g_\infty Z}(\eta)(g)
    \]
    when $g\in \Res_{\F/\Q}\G(\V)(\R)^+$.
\end{lem}
\begin{proof}
    The claim about the central character follows from
    Definition \ref{def:theta_ext}. On the other hand, 
    if $g_\infty\in \Res_{\F/\Q}\G(\R)^+$
    and $(g_\infty,h_0)\in \Res_{\F/\Q}G(\UU(\V)\times \UU(V))(\R)$,
    since $P(h_0^{-1}z)\eta(h_0)=P(z)$,
    by Lemma \ref{lem:arch}, 
\begin{multline*}
    \theta^\square_{P,Z}(\eta)(gg_\infty)=
    |\nu(gg_\infty)|\int_{[H]}
    \sum_{z\in \K^2}
    \omega^\square(gg_\infty, h_0h)\Phi_{P,Z}(z)\eta(hh_0)\,dh\\=
    \nu(g_\infty)^{(k+2)\Sigma}|\nu(g)|\int_{[H]}
    \sum_{z\in \K^2}
    \omega^\square(g, h)\Phi_{J(g_\infty, Z)P,g_\infty Z}(h_0^{-1}z)\eta(hh_0)\,dh
    =\nu(g_\infty)^{(k+2)\Sigma}
    \theta^\square_{J(g_\infty,Z) P,g_\infty Z}(\eta)(g).
\end{multline*}
\end{proof}

\begin{defn}\label{def:algtheta}
    Suppose the Schwartz function $\Phi_f$ associated to $\eta$
    is invariant by $\KK^n_1$,  we define the $L^{\bwt{k}}(\C)$-valued
    automorhpic function $\ff(\eta)(g)$ and 
    $\ff(\eta)(Z,g_f)\in \M_{\bwt{k}}^{\an}(\KK^n_1,\C)$ 
    respectively by 
    \begin{align*}
        &\ff(\eta)(g)\colon 
        \big[P\in L_{k\Sigma}\mapsto (2\pi)^{-\Sigma}
        \theta^\square_{P}(\eta)(g)\big]\\
        &\ff(\eta)(Z,g_f)\colon 
        \big[P\in L_{k\Sigma}\mapsto (2\pi)^{-\Sigma}
        \theta^\square_{P,Z}(\eta)(g_f)\big]
    \end{align*}
    Then $\ff(\eta)(Z,g_f)=
    \nu(g_\infty)^{-(k+2)\Sigma}\rho^{\bwt{k}}(J(g_\infty, \ii))
    \ff(\eta)(g_\infty g_f)$ when $Z=g_\infty \ii$
    by Lemma \ref{lem:cycle}.
    Consequently $\ff(Z,g_f)$ satisfies \eqref{defn:holoform}
    since $\lVert\bwt{k}\rVert=(k+2)\Sigma$.
\end{defn}

\begin{prop}
    The modular form $\ff(\eta)\in \M_{\bwt{k}}(\KK^n_1,\C)$
    associated to $\ff(\eta)(Z,g_f)$ is $p$-integral.
    In other word, after enlarging $\keu$,
    we may assume $\ff(\eta)\in \M_{\bwt{k}}(\KK^n_1,\oeu)$
    for some $\oeu=\oo_{\keu,(\fp)}$.
\end{prop}
\begin{proof}
    For $(a_1,a_2)\in \GL(M^0)$ as in 
    Lemma \ref{lem:Fouerier_compare} we put 
    $A'=(a_2^{-\intercal}, a_1^{-\intercal})$ so that 
    $m(a_1,a_2)=\smat{A'&\\&A'^{-*}}$.
    Let $A_1\in\GL_2(\A_{\K,f}^{(p\fn)})$ and $A=(A_1,A')$.
    By Corollary \ref{cor:Fourier}, for $\beta=z^*z\in\her_2(\F)$
    nonzero the Fourier coefficient $(2\pi)^{-\Sigma}
    W_\beta(\smat{A&\\&A^{-*}}\xi^{-1}, \theta_{P,Z}^\square(\eta))$ 
    is equal to 
    \begin{equation}\label{eq:Fourier}
    C(\K^1) \chi_\circ(\det A)\tilde{\eta}(z'')\cdot P(z)
    \prod_{w\in \Sigma_p}\tilde{\eta}_w(z'_1)
    \cdot \prod_{v\in\finite\text{ split }}
    R_{z',v}(\tilde{\eta}_w^{-1}(\varpi_w))\id_{\Xi}(A^*\beta A).
    \end{equation}
    In other word, $\ff(\eta)(Z,\smat{A&\\&A^{-*}}\xi^{-1})$
    has a Fourier expansion 
    $\sum_{\beta}a_\beta(\eta, \smat{A&\\&A^{-*}}\xi^{-1})q^\beta$,
    where the sum goes through nonzero Hermitian matrices 
    of the form $\beta=z^*z$. And the value of 
    $a_\beta(\eta, \smat{A&\\&A^{-*}}\xi^{-1})\in L^{\bwt{k}}(\C)$
    at $P\in L_{k\Sigma}$ is given by \eqref{eq:Fourier}.
    Since the value is unchanged if we replace 
    $z$ by $hz$ for $h\in\K^1$,
    by the strong approximation, we may assume 
    $z_i\in\oo_w$ for all $w\mid p$. Consequently 
    the value in \eqref{eq:Fourier} is $p$-integral
    for any monomial in $L_{k\Sigma}$,
    and the claim follows from the algebraic $q$-expansion principle.
\end{proof}

By the proposition, we can apply Definition \ref{def:avatar}
and define the p-adic modular form 
$\beta_{\bwt{k}}(\ff(\eta))\in V_p(\G,\KK)$ associated to $\ff(\eta)$.
Compare \eqref{eq:Fourier} 
with Lemma \ref{lem:Fouerier_compare} and
recall that $\blk$ is given by evaluating at $P_1(z)=z_1^{k\Sigma}$.
Since $\rho(a_2,a_1)P_1(z)=\det a_1^{-1}P_1(za_2^{-\intercal})$
and $z_{1,\Sigma_p}'\coloneqq (zA)_{1,\Sigma_p}
=z_{1,\Sigma_p}a_2^{-\intercal}$,
the Fourier expansion of $\beta_{\bwt{k}}(\ff(\eta))\in V_p(\G,\KK)$
is given by 
\begin{equation}\label{eq:padic_Fourier}
    C(\K^1) \chi_\circ(\det A)\det a_1^{-\Sigma}\tilde{\eta}(z'')
    \cdot (z_{1,\Sigma_p}')^{k\Sigma}
    \prod_{w\in \Sigma_p}\tilde{\eta}_w(z'_1)
    \cdot \prod_{v\in\finite\text{ split }}
    R_{z',v}(\tilde{\eta}_w^{-1}(\varpi_w))\id_{\Xi}(A^*\beta A).
\end{equation}

\begin{thm}
    There exists a measure of $\fG_\fs$ valued
    in $V_p(\G,\KK(\fn))$ such that 
    for any $\fs$-admissible Hecke character $\eta$
    of infinity type $k\Sigma$ for $k\geq 0$
    \[
        \int_{\fG_\fs}\hat{\eta}d\euF_{\fs}=\bbeta_{\bwt{k}}(\ff(\eta)).
    \]
\end{thm}
\begin{proof}
    Recall from Corollary \ref{cor:Fourier} that 
    the terms $z''$ and $R_{z',v}(\tilde{\eta}^{-1}_w(\varpi_w))$
    in \eqref{eq:padic_Fourier} only involves terms away $p$.
    Therefore, since $\eta$ has the infinity type $k\Sigma$, 
    the expression \eqref{eq:padic_Fourier} can be written as 
    \[
        \sum_i b_i\cdot \hat{\eta}(c_i).
    \]
    for some $b_i\in \oeu$ and $c_i\in \A_{\K,f}^\times$.
    Then as the proof of \cite[Theorem 5.8]{Hsieh2014},
    the existence of the measure follows from 
    the abstract Kummer congruences and the 
    $q$-expansion principle of $p$-adic modular forms.
\end{proof}


\subsection{Family of pull-back theta lifts}

Define the pull-back of \eqref{def:theta_eta} with respect to 
$\bnu=\{\chi^{-1},\chi^{-1}\}$ by
\begin{equation*}
    \theta^\square_{P,Z}(\eta,\bnu)(g)=\int_{[T(\A_\K)]}
    \theta^\square_{P,Z}(\eta)(\iota(g,\bbeta)\Upsilon)
    \chi^{-1}(\beta_1\beta_2)\,d\beta_1d\beta_2,\quad
    \bbeta=\smat{\beta_1&\\&\beta_2}.
\end{equation*}
We fix an open compact subgroup $K=K(\fn)$ as in \S6
so that $\theta^\square_{P,Z}(\eta,\bnu)$
is invariant by $K^n_1$ for some $n$.
Note that while $\chi^{-1}(\beta_i)=\chi_\circ^{-1}(\beta_i)=\beta_i$
for $\beta_i\in\K_\infty^1$, 
apply \eqref{eq:cocycle} to Lemma \ref{lem:cycle} shows that 
\[
    \theta^\square_{P,Z_\delta}(\eta)(\iota(\id_2,\bbeta)\Upsilon)=
    \theta^\square_{J(\id_2,\beta)P,Z_\delta}(\eta)(\Upsilon)=
    \beta_1^{-1}\beta_2^{-1}\cdot
    \theta^\square_{P,Z_\delta}(\eta)(\Upsilon)
\]
and therefor the integral is $\K_\infty^1$-invariant.
On the other hand,
pick for each $v\in\finite$ an open compact subgroup $U_v\subset K_v^1$ 
such that $\chi_v(\alpha)=1$ if $\alpha\in U_v$ and
\begin{equation}\label{def:smallU}
    \theta^\square_{P,Z}(\eta)
    (\iota(\smat{1&\\&\alpha_2}, \smat{\beta_1&\\&\beta_2})\Upsilon)=
    \theta^\square_{P,Z}(\eta)(\Upsilon)
    \text{ when } \alpha_2, \beta_1,\beta_2\in U_v.
\end{equation}
By Proposition \ref{prop:xi_split}
the definition is independent of $\eta$ or $\fs$.
Furthermore by Proposition \ref{prop:non_split_inv} and 
Proposition \ref{prop:inv_Up}
we can take $U_v$ to be the maximal open compact subgroup,
so $U(\fc)\coloneqq\prod_v U_v$ is an open compact subgroup
of $\A_{\K,f}^1$.
Combined with the discussion above, when $Z=Z_\delta$
the integral in $\theta^\square_{P,Z}(\eta,\bnu)$
is invariant by $\beta_i\in \K_\infty^1\times U(\fc)$
can be rewritten as
\begin{equation}\label{def:pullback_eta}
    \theta^\square_{P,Z}(\eta,\bnu)(g)=
    \vol(\K_\infty^1\times U(\fc))^2
    \sum_{\beta_i\in\K^1\backslash\A_{\K,f}^1/U(\fc)}
    \theta^\square_{P,Z_\delta}(\eta)(\iota(g,\bbeta)\Upsilon)
    \chi_\circ^{-1}(\beta_1\beta_2).
\end{equation}

\begin{defn}\label{def:algpullback}
    We define the pull-back of $\ff(\eta)(Z,g_f)$
    from Definition \ref{def:algtheta} as 
    the $L^{\bwt{k}}(\C)$-valued function on $g_f\in\UU(\W)(\A_{\F,f})$
    given by
    \[
        \ff^\circ(\eta)(g_f)=
        \sum_{\beta_i\in\K^1\backslash\A_{\K,f}^1/U(\fc)}
        \ff(\eta)(Z_\delta, \iota(g_f,\bbeta)\Upsilon)
        \chi_\circ^{-1}(\beta_1\beta_2).
    \]
\end{defn}
Let $(\alpha,\beta)\in \UU(\W)\times \UU(-\W)(\F)$, 
then by \eqref{eq:cocycle} we see that
\begin{equation}\label{eq:pullbackcocycle}
    \ff(\eta)(Z_\delta,
    \iota(\alpha,\beta)\iota(g_f,\bbeta)\Upsilon)=
    \rho^{\bwt{k}}(J(\iota(\alpha,\beta),Z_\delta))
    \ff(\eta)(Z_\delta,\iota(g_f,\bbeta)\Upsilon)=
    \rho^{\bwt{k}}(\alpha^{-\intercal},\beta)
    \ff(\eta)(Z_\delta,\iota(g_f,\bbeta)\Upsilon)
\end{equation}

\begin{prop}\label{prop:alg_padic}
    Let $\underline{B}=\underline{B}_{g_f}$ be the quintuple associated 
    to $g_f\in\UU(\W)(\A_{\F,f})$ and define
    \[
        f(\eta)(g_f)=f(\eta)(\underline{B})=
        \sum_{\beta_i\in\K^1\backslash\A_{\K,f}^1/U(\fc)}
        \hat{\ff}(\eta)(\underline{B},\underline{B}_{\bbeta})
        \widehat{\chi}_\circ^{-1}(\beta_1\beta_2),
    \]
    then the values $\ff^\circ(\eta)(g_f)$ and $f(\eta)(\underline{B})$
    are related by 
    \begin{equation}\label{eq:alg_padic}
         \left(\frac{2\pi i}{\Omega_\infty}\right)^{(k+2)\Sigma}
          \ff^\circ(\eta)(g_f)=
          \frac{1}{\Omega_p^{(k+2)\Sigma}}\cdot
        \rho_{\bwt{k}}(g_{\Sigma_p},\id_2)
        f(\eta)(\underline{B}).
    \end{equation}
    In particular, $f(\eta)(g_f)$ is an algebraic modular form
    of weight $\wt{k}\coloneqq(k\Sigma, 0)$
    in the sense of Definition \ref{def:algform}.
\end{prop}
\begin{proof}
    Since \eqref{eq:pullbackcocycle} implies that 
    $\ff(\eta)(Z_\delta,\iota(g_f,\bbeta)\Upsilon)$
    is a modular form of weight $(k\Sigma,0)$ and $(\Sigma,\Sigma)$
    in the first and the second variable respectively,
    by \eqref{eq:pullback} and \eqref{eq:padic_pullback} we have
    \begin{multline*}
        \left(\frac{2\pi i}{\Omega_\infty}\right)^{(k+2)\Sigma}
        \ff(\eta)(Z_\delta, \iota(g_f, \bbeta)\Upsilon)=
        \left(\frac{2\pi i}{\Omega_\infty}\right)^{(k+2)\Sigma}
        \ff(\eta)(\underline{B}\times\underline{B}_{\bbeta}
        ,2\pi i d\underline{z}_{\W}\times 2\pi i d\underline{z}_{\W})\\=
        \ff(\eta)(\underline{B}\times\underline{B}_{\bbeta}
        ,\ww_{B}\times \ww_{B_{\bbeta}})=
        \frac{1}{\Omega_p^{(k+2)\Sigma}}
        \ff(\eta)(\underline{B}\times\underline{B}_{\bbeta}
        ,d^\times t\times d^\times t)=
        \frac{1}{\Omega_p^{(k+2)\Sigma}}
        \rho_{\bwt{k}}(g_{\Sigma_p}, \bbeta_{\Sigma_p^c})
        \hat{\ff}(\eta)(\underline{B}\times\underline{B}_{\bbeta}).
    \end{multline*}
    Here the last equality follows from 
    $\ww(j_B)=g_{\Sigma_p}^{-1}d^\times t$ and 
    $\ww(j_{B_{\bbeta}})=\bbeta_{\Sigma_p^c}^{-1}d^\times t$.
    The difference between $\Sigma_p$ and $\Sigma_p^c$ 
    reflects the different choice of polarizations
    $L^0=e^-L_p$ and $(-L)^0=e^+(-L)_p$.
    Then the quality follows from 
    $\rho_{\bwt{k}}(\id_2,\bbeta_{\Sigma_p^c})=
    \beta_{1,\Sigma_p^c}^{-1}\beta_{2,\Sigma_p^c}^{-1}$ and
    $\beta_{1,\Sigma_p^c}^{-1}\chi_\circ^{-1}(\beta_i)=
    \widehat{\chi}_\circ^{-1}(\beta_i)$.

    Let $\alpha\in \UU(\W)(\F)$ and $u\in K^n_1$.
    Since $\ff^\circ(\eta)(\alpha g_f)=
    \rho^{\bwt{k}}(\alpha^{-\intercal},\id_2)\ff^\circ(\eta)(g_f)=
    \rho_{\bwt{k}}(\alpha,\id_2)\ff^\circ(\eta)(g_f)$
    by \eqref{eq:pullbackcocycle}, we see that
    \[
         \left(\frac{2\pi i}{\Omega_\infty}\right)^{(k+2)\Sigma}
         \rho_{\bwt{k}}(\alpha,\id_2)
          \ff^\circ(\eta)(g_f)=
         \left(\frac{2\pi i}{\Omega_\infty}\right)^{(k+2)\Sigma}
          \ff^\circ(\eta)(\alpha g_fu)=
          \frac{1}{\Omega_p^{(k+2)\Sigma}}\cdot
        \rho_{\bwt{k}}((\alpha g_fu)_{\Sigma_p},\id_2)
        f(\eta)(\alpha g_fu)
    \]
    and therefore
    $f(\eta)(\alpha g_f u)=\rho_{\bwt{k}}(u_{\Sigma_p},\id_2)
    f(\eta)(g_f)$ as desired.
\end{proof}

\begin{prop}\label{prop:eigenvalue}
    The algebraic modular form $f(\eta)$ is an ordinary eigenform
    for the Hecke operators defined from 
    \eqref{eq:hecke_awayp},
    \eqref{eq:hecke_atp}, and \eqref{eq:hecke_ats}.
    Precisely, write $\tilde{\eta}^\wedge$ for 
    the p-adic avatar of $\tilde{\eta}$, then
    \begin{itemize}
    \item When $v\in\finite$ is prime-to-p finite place that 
    splits into $w\bw$ in $\K$ and $K_v^0\subset K$,
    \begin{equation*}
    T^{(1)}_wf(\eta)=[(\chi_\circ^{-1}\tilde{\eta})(\varpi_w)+
    q_w\chi_\circ(\varpi_w)]f(\eta) ,\quad
    T^{(2)}_wf(\eta)=\tilde{\eta}(\varpi_w)f(\eta),
    \end{equation*}
    \item When $v=w\bw$ splits for $w\in \Sigma_p$ and
    $u\in \oo_v^\times$ 
    $\langle u\rangle_{\wt{k}}\pf(g)=\pf(gu)$ and
    \begin{equation*}
    U^{(1)}_{\wt{k},w}f(\eta)=
    (\widehat{\chi}_\circ^{-1}\tilde{\eta}^\wedge)(\varpi_w)
    f(\eta),\quad
    U^{2}_{\wt{k},w}f(\eta)=
    \tilde{\eta}^\wedge(\varpi_w)f(\eta),\quad
    \langle u\rangle_{\wt{k}}f(\eta)=
    \tilde{\eta}^\wedge(u)f(\eta).
    \end{equation*}
    \item When $v=w\bw$ splits for $w\mid \mathfrak{f}$ and 
    $u\in \oo_v^\times$, 
    \begin{equation*}
    U^{(1)}_{w}f(\eta)=
    (\chi_\circ^{-1}\tilde{\eta})(\varpi_w)
    f(\eta),\quad
    U^{2}_{w}f(\eta)=
    \tilde{\eta}(\varpi_w)f(\eta),\quad
    \langle u\rangle f(\eta)=
    \tilde{\eta}(u)f(\eta).
    \end{equation*}
\end{itemize}
\end{prop}
\begin{proof}
    Since we have computed the Hecke eigenvalue of 
    $\omega^\square(\xi_v)\Phi_v$ in Proposition \ref{prop:inv_Up},
    the above formulae follows from the previous proposition
    and the fact that pull-back theta lift
    $\theta^\square_{P,Z}(\eta,\bnu)$ has the central character
    $\eta\vert_{\A_{\K,f}^1}$.
    We note that this latter fact is a consequence of
    Proposition \ref{prop:pullback} and that 
    $\omega(g,h)=\omega(h^{-1}g,1)$ for the Weil representation
    on $\UU(\W)\times \UU(V)$.
\end{proof}

\begin{thm}\label{thm:family}
    We define the pull-back of $\euF_\fs$ to be the 
    $V_p(G_1,K(\fn))$-valued measure $\euF_\fs^\circ$
    of $\fG_\fs$ given by 
    \[
        \int_{\fG_{\fs}} \alpha d\euF^\circ_{\fs}(\underline{B})=
        \sum_{\beta_i\in\K^1\backslash\A_{\K,f}^1/U(\fc)}
        \int_{\fG_{\fs}} \alpha d\euF_{\fs}
        (\underline{B},\underline{B}_{\bbeta})
        \widehat{\chi}_\circ^{-1}(\beta_1\beta_2).
    \]
    Then the measure defines a Hida family 
    in $M^{\ord}(K(\fn), \I_\fs)$ for 
    $\I_\fs=\oo\llbracket \fG_\fs^a\rrbracket$.
\end{thm}
\begin{proof}
    By definition the measure interpolates
    $\beta_{\wt{k}}(f(\eta))$, which is ordinary
    by the previous proposition, for $\fs$-admissible 
    Hecke characters $\eta$ of infinity type $k\Sigma$ for $k\geq0$.
    Since the set of such characters is dense, 
    the measure does take values in $V_p^{\ord}(G_1,K(\fn))$.
    
    To show that $\euF_\fs^\circ$ defines a Hida family valued
    in $\I_\fs$. Observe that since we only consider the values
    of pull-back theta lifts on $G_1=\UU(\W)$, by definition
    $f(\eta)$ only depends on the restriction of $\eta$
    to $\A_\K^1$, or equivalently, 
    by the commutative diagram \eqref{eq:classgps},
    the value $\int\alpha\,d\euF_\fs^\circ$ only depends
    on the restriction of $\alpha$ to $\fG_\fs^-$.
    Thus $F_\fs^\circ$ is actually a measure on 
    $\fG_\fs^-\cong \fG_\fs^a$.
    Such a measure naturally gives an element in
    \[
        V_p^{\ord}(G_1,K(\fn))\hat{\otimes}\I_\fs.
    \]
    That the element satisfies the condition in \eqref{eq:Hida_family} 
    follows from that $f(\eta)$ has the central character
    $\hat{\eta}$ by \eqref{eq:alg_padic}.
\end{proof}

\subsection{Anticyclotomic p-adic L-function}

Having constructed the Hida family,
we can consider the pairing
$\B(\euF_\fs^\circ,\euF_\fs^\circ)\in \I_\fs$
of $\euF^\circ$ with itself as
defined in \eqref{def:Hida_pairing},
which interpolates the values of the pairings
$(f(\eta), U_{\wt{k},p}^{-n}f(\eta))_n$ defined in 
\eqref{def:pairings}
when $\eta$ is an $\fs$-admissible Hecke character
of infinity type $k\Sigma$
by Proposition \ref{prop:Hida_pairing}
and Proposition \ref{prop:pair_at_deff_level}.
We first relate $(f(\eta), f(\eta))_n$
to the inner product of the pull-back theta lift 
$\theta^\square_{P_1,Z_\delta}$, 
assuming that the latter is invariant by $K^n_1(\fn)$.

\begin{lem}\label{lem:compare_Rallis}
   Let $(\cdot,\cdot)_{\wt{k}}$ be the pairing as in
   \eqref{eq:pairing_transpose}, then 
   \begin{multline*}
    (2\pi)^{2\Sigma}\vol(\K_\infty^1\times U(\fc))^4
    \frac{\vol(G_{1\infty})}{(k+1)^d}
    \vol(K^n_1(\fn)) \cdot\frac{1}{\Omega_p^{(2k+4)\Sigma}}
    \sum_{C(K^n_0(\fn))}(f(\eta), f(\eta))_k\\=
   \left(\frac{2\pi i}{\Omega_\infty}\right)^{(2k+4)\Sigma}
    \wt{k}(p^n,1)
    \int_{[G_1]}\theta^\square_{P_1,Z_\delta}(\eta,\bnu)(g)
    \theta^\square_{P_1,Z_\delta}(\eta,\bnu)(\bar{g}\tau_n(\fs))\,dg.
   \end{multline*}
\end{lem}
\begin{proof}
    Write $\rho_{\wt{k}}(g_p)=\rho^{\wt{k}}(g_p^{-\intercal})
    =\rho^{\bwt{k}}(g_{\Sigma_p}^{-\intercal},\id_2)$.
    By \eqref{eq:pairing_transpose} we have
    \[
        (\pf_1(g),\rho_{\wt{k}}(\tau_n)\pf_2(\bar{g}\tau_n(\fs)))_k=
        (\rho_{\wt{k}}(g_p)\pf_1(g),
        \rho_{\wt{k}}(\bar{g}_p)\rho_{\wt{k}}
        (\tau_n)\pf_2(\bar{g}\tau_n(\fs)))_k.
    \]
    Then Proposition \ref{prop:alg_padic} implies that 
    \begin{equation}\label{eq:pairing_compare}
        \left(\frac{2\pi i}{\Omega_\infty}\right)^{(2k+4)\Sigma}
        \wt{k}(p^n,1)\sum_{C(K^n_0(\fn))}
        (\ff^\circ(\eta)(g_f),
        \ff^\circ(\eta)(\bar{g}\tau_n(\fs)))_k=
        \frac{1}{\Omega_p^{(2k+4)\Sigma}}\cdot
        (f(\eta), f(\eta))_n.
    \end{equation}
    On the other hand, identify $\blk(v)=(v, v_{-\wt{k}})_{\wt{k}}$,
    then by Delinition \ref{def:algtheta} and
    \eqref{def:pullback_eta} we have 
    \[
        \theta^{\square}_{P_1,Z_\delta}(\eta,\bnu)(g_f)=
        (2\pi)^\Sigma\vol(\K_\infty^1\times U(\fc))^2
        \blk(\ff^\circ(\eta))(g_f)=
        (2\pi)^\Sigma\vol(\K_\infty^1\times U(\fc))^2
        (\ff^\circ(\eta)(g_f), v_{-\wt{k}})_k.
    \]
    But note that 
    $\theta^{\square}_{P_1,Z_\delta}(\eta,\bnu)(g_\infty g_f)
    =\theta^{\square}_{\rho(g_\infty^{-\intercal},
    \id_2)P_1,Z_\delta}(\eta,\bnu)(g_f)$
    by Lemma \ref{lem:cycle} and \eqref{eq:cocycle}.
    Therefore, identify $\rho^{\bwt{k}}(\C)$
    with the contragredient of $\rho$ and
    apply \eqref{eq:pairing_transpose} shows that
    \[
        \theta^{\square}_{P_1,Z_\delta}(\eta,\bnu)(g_\infty g_f)
        =C\cdot \blk(\rho^{\wt{k}}(g_\infty^\intercal)
        \ff^\circ(\eta))(g_f)
        =C\cdot (\ff^\circ(\eta)(g_f),
        \rho^{\wt{k}}(g_\infty)v_{-\wt{k}})_k,
    \]
    where $C$ is the constant 
    $(2\pi)^\Sigma\vol(\K_\infty^1\times U(\fc))^2$,
    from which we deduce that 
    \begin{multline*}
    \int_{[G_1]}\theta^\square_{P_1,Z_\delta}(\eta,\bnu)(g)
    \theta^\square_{P_1,Z_\delta}(\eta,\bnu)(\bar{g}\tau_n(\fs))\,dg
    \\=C^2\cdot 
    \int_{[G_1]}
    (\ff^\circ(\eta)(g_f),
        \rho^{\wt{k}}(g_\infty)v_{-\wt{k}})_{\wt{k}}
        (\ff^\circ(\eta)(\bar{g}_f\tau_n(\fs)),
        \rho^{\wt{k}}(\bar{g}_\infty)v_{-\wt{k}})_{k}\,dg\\=
        C^2\cdot \frac{\vol(G_{1\infty})}{(k+1)^d}
        \int_{G_1(\F)\backslash G_1(\A_{\F,f})}
        (\ff^\circ(\eta)(g_f),
        \ff^\circ(\eta)(\bar{g}_f\tau_n(\fs))_{\wt{k}}\,dg_f\\=
        C^2\cdot \frac{\vol(G_{1\infty})}{(k+1)^d}
        \vol(K^n_1(\fn))
        \sum_{C(K^n_0(\fn))}
        (\ff^\circ(\eta)(g_f),
        \ff^\circ(\eta)(\bar{g}\tau_n(\fs)))_k
    \end{multline*}
    where we have used Schur's orthogonality for the first equality.
    Then apply the above equality back to the left hand side
    of \eqref{eq:pairing_compare} proves the lemma.
\end{proof}

\begin{prop}
    For each $v\mid\fs$, fix $v=w\bw$ so that $w\mid \fs$
    and let $U_\fs$ be the product of $U_w^{\val_w(\fs)}$, then
\begin{multline*}
    \vol(\K_\infty^1\times U(\fc))^4
    \frac{1}{\Omega_p^{(2k+4)\Sigma}}
    \sum_{C(K^n_0(\fn))}(f(\eta),
   U_{\wt{k},p}^{-n}U_\fs^{-1} f(\eta))_k\\=
   \left(\frac{2\pi}{\Omega_\infty}\right)^{(2k+4)\Sigma}
   (2\pi)^{3\Sigma}C(\K^1)^3C_1\#C(K(\fn_\circ))
        \frac{L(1,\epsilon_{\K/\F})}{\zeta^{(\fc)}(2)}
        \frac{\prod_{v\mid c^+}(1-q_v^{-1})^2}
        {\prod_{v\mid c^-}(1+q_v^{-1})}L(\frac{1}{2},\chi)^2\\
        \textnormal{Im}(\delta)^{k}\tilde{\eta}(z_\delta^{-1})
        \frac{\Gamma((k+2)\Sigma)}{(2\pi)^{(k+2)\Sigma}}
    L(1,\chi^{-2}\tilde{\eta})
    \prod_{v\mid p\fs}
    \frac{(1-\chi^{-2}\tilde{\eta}(\varpi_\bw)q_\bw^{-1})
    (1-\chi^{2}\tilde{\eta}^{-1}(\varpi_w))}
    {\varepsilon(1,(\chi^{-2}\tilde{\eta})_w,\psi_w)}.
\end{multline*} 
\end{prop}
\begin{proof}
    Observe first that by Proposition \ref{prop:eigenvalue}
    the Hecke eigenvalues of $U_{\wt{k},p}^n$ and $U_\fs$
    cancels with $\wt{k}(p^n,1)$ and 
    the terms $\prod_{v\mid p\fs}(\chi_\circ^{-1}
    \tilde{\eta})_w(\varpi_w^n)$ from \eqref{eq:fullRallis}.
    Then combine Lemma \ref{lem:compare_Rallis}
    and Corollary \ref{cor:Rallis} shows that 
\begin{multline*}
    (2\pi)^{2\Sigma}\vol(\K_\infty^1\times U(\fc))^4
    \frac{\vol(G_{1\infty})}{(k+1)^d}\vol(K^n_1(\fn))
    \frac{1}{\Omega_p^{(2k+4)\Sigma}}
    \sum_{C(K^n_0(\fn))}(f(\eta),
   U_p^{-1}U_\fs^{-1} f(\eta))_n\\=
   \left(\frac{2\pi i}{\Omega_\infty}\right)^{(2k+4)\Sigma}
   2(2\pi)^{3\Sigma}C(\K^1)^3C_1
        \frac{L(1,\epsilon_{\K/\F})}{\zeta^{(\fc)}(2)}
        \frac{\prod_{v\mid c^+}(1-q_v^{-1})^2}
        {\prod_{v\mid c^-}(1+q_v^{-1})}L(\frac{1}{2},\chi)^2\\
        \textnormal{Im}(-\delta)^{k}\tilde{\eta}(z_\delta^{-1})
        \frac{\Gamma((k+1)\Sigma)}{(2\pi )^{k\Sigma}}
    L(1,\chi^{-2}\tilde{\eta})
    \prod_{v\mid p\fs}\frac{q_v^{-n_v}}{(1+q_v^{-1})}
    \frac{(1-\chi^{-2}\tilde{\eta}(\varpi_\bw)q_\bw^{-1})
    (1-\chi^{2}\tilde{\eta}^{-1}(\varpi_w))}
    {\varepsilon(1,(\chi^{-2}\tilde{\eta})_w,\psi_w)}.
\end{multline*}
Now the proposition follows from the observation that 
$\Gamma(k+2)=\Gamma(k+1)(k+1)$ and
\[
    \frac{2}{\vol(G_{1\infty})\vol(K^n_1(\fn))}
    \prod_{v\mid p\fs}\frac{q_v^{-n_v}}{(1+q_v^{-1})}=
    \frac{\vol([G_1])}{\vol(G_{1\infty})\vol(K(\fn_\circ))}=
    \#C(K(\fn_\circ)).
\]
\end{proof}

To simplify the notation,
we arrange some of the constants above into a single one and define
\[
C(\K,\chi)=
C(\K^1)^3C_1\#C(K(\fn_\circ))\frac{(2\pi)^{5\Sigma}}
{\vol(\K_\infty^1\times U(\fc))^4}
\frac{L(1,\epsilon_{\K/\F})}{\zeta^{(\fc)}(2)}
\frac{\prod_{v\mid c^+}(1-q_v^{-1})^2}
{\prod_{v\mid c^-}(1+q_v^{-1})}
\left[\frac{L(\frac{1}{2},\chi)}
{\Omega_\infty^\Sigma}\right]^2,
\]
which we note is an algebraic number.
Then we have the following theorem, which follows from 
Proposition \ref{prop:pair_at_deff_level},
Proposition \ref{prop:Hida_pairing},
and the computation above.

\begin{thm}\label{thm:function}
    Let notations be as above and define
    $L_\fs=\Omega_p^{-2\Sigma}
    \B(\euF^\circ_{\fs},U_\fs^{-1}\euF^\circ_{\fs})\in\I_\fs$
    As a measure on $\fG_\fs^a$, when 
    $\eta$ is $\fs$-admissible of infinity type $k\Sigma$
    for $k\geq 0$ we have
    \begin{equation}\label{eq:function}
        \frac{1}{\Omega_p^{(2k+2)\Sigma}}
        \int_{\fG_{\fs}^a}\hat{\eta}\,L_\fs=C(\K,\chi)
        \left(\frac{2\pi}{\Omega_\infty}\right)^{(2k+2)\Sigma}
    \textnormal{Im}(\delta)^{k}\tilde{\eta}(z_\delta^{-1})
    \frac{\Gamma((k+2)\Sigma)}{(2\pi)^{(k+2)\Sigma}}
    L(1,\chi^{-2}\tilde{\eta})
    \prod_{v\mid p\fs}E_v(1,\chi^{-2}\tilde{\eta})
    \end{equation}
    where $E_v(1,\chi^{-2}\tilde{\eta})$ is the modified Euler factor 
    $(1-\chi^{-2}\tilde{\eta}(\varpi_\bw)q_\bw^{-1})
    (1-\chi^{2}\tilde{\eta}^{-1}(\varpi_w))/
    {\varepsilon(1,(\chi^{-2}\tilde{\eta})_w,\psi_w)}$.
\end{thm}
\begin{rem}
    While the Hida family $\euF^\circ$ only interpolates
    $f(\eta)$ when $\eta$ has the infinity type $k\Sigma$,
    one can compare \eqref{eq:function} with, for example,
    \cite[eq (4.13)]{Hsieh22} by setting 
    $\lambda=\chi^{-2}\tilde{\eta}|\cdot|_\K$,
    which has the infinity type $2\Sigma+k(1-c)\Sigma$,
    and check that $L_\fs$ is indeed the anticyclotomic p-adic
    L-function of Hida and Tilouine \cite{Hida1993}
    associated to $\chi^{-2}|\cdot|_\K$.
\end{rem}

\begin{rem}
When $\chi_v\vert_{\K_v^1}\equiv 1$ for all non-split places,
another case of interest is when $\eta$ is of the form 
$\eta=\chi_\circ\nu$, where
$\nu$ is $\fs$-admissible of infinity type $(k+1)\Sigma$ for $k\geq0$.
If we then replace the choice in \S\ref{sec:choice}
at $v\mid \fc$ by
\[
    \phi_{1w}=\phi_{2w}=\id_{\oo_v},\quad
    \phi_{1\bw}=\phi_{2\bw}=\chi_w^{-1}\id_{\oo_v^\times}
\]
then the statement for $v\nmid p\fc\fs$ in
Proposition \ref{prop:Rallis_local} also holds for $v\mid \fc$,
Corollary \ref{cor:Rallis} holds if we replace
$\zeta^{(\fs)}(2)$ by $\zeta(2)$,
and Corollary \ref{cor:Fourier} also holds
if we modify $z''$ and $R_{z',v}$ at $v\mid \fs$ accordingly.
Consequently we can construct a Hida family $\euF'^\circ$
interpolating $f(\chi\nu)$ such that 
$L_\fs'\coloneqq\Omega_p^{-2\Sigma}
\B(\euF'^\circ_{\fs},U_\fs^{-1}\euF'^\circ_{\fs})\in\I_\fs$
satisfies 
\begin{equation}
        \frac{1}{\Omega_p^{(2k+2)\Sigma}}
        \int_{\fG_{\fs}^a}\hat{\nu}\,L'_\fs=C(\K,\chi)
        \left(\frac{2\pi}{\Omega_\infty}\right)^{(2k+2)\Sigma}
    \textnormal{Im}(\delta)^{k}\tilde{\eta}(z_\delta^{-1})
    \frac{\Gamma((k+2)\Sigma)}{(2\pi)^{(k+2)\Sigma}}
    L(1,\tilde{\nu})
    \prod_{v\mid p\fs}E_v(1,\tilde{\nu}).
\end{equation}
\end{rem}

\subsection{Primitivity and toric periods}

We can similarly compare $f(\eta)$ with 
$P_{\bmu}(\theta^\square_{P_1,Z_\delta}(\eta,\bnu))$,
the toric period integral associated to 
$\bnu=\{(\eta\mu)^{-1}, \mu\}$.
Since $\theta^\square_{P_1,Z_\delta}(\eta,\bnu)$ 
has central character $\eta$
and $\rho(\smat{1&\\&\alpha},\id_2)P_1=P_1$
if $\alpha\in\K_\infty^1$,
let $U(\fc)\subset \A_{\K,f}^1$
be the open compact subgroup defined from \eqref{def:smallU},
the same argument as \eqref{def:pullback_eta} shows that 
\begin{multline}\label{eq:period_sum}
    P_{\bmu}(\theta^\square_{P_1,Z_\delta}(\eta,\bnu))\coloneqq
    \int_{[T]}\theta^\square_{P_1,Z_\delta}(\eta,\bnu)(\smat{\alpha_1&\\&\alpha_2})\eta^{-1}(\alpha_1)
    \,d\alpha_1d\alpha_2\\=\vol([\A_\K^1])
    \int_{[\A_\K^1]}\theta^\square_{P_1,Z_\delta}(\eta,\bnu)(\smat{1&\\&\alpha})\,d\alpha=\vol([\A_\K^1])
    \vol(\K_\infty^1\times U(\fc))
    \sum_{\alpha\in\K^1\backslash\A_{\K,f}^1/U(\fc)}
    \theta^\square_{P_1,Z_\delta}(\eta,\bnu)(\smat{1&\\&\alpha}).
\end{multline}

\begin{lem}
The toric period integral 
$P_{\bmu}(\theta^\square_{P_1,Z_\delta}(\eta,\bnu))$ 
can be related to a finite sum of values of $f(\eta)$ by
   \begin{equation}\label{eq:compare_period}
    \left(\frac{2\pi i}{\Omega_\infty}\right)^{(k+2)\Sigma}
    P_{\bmu}(\theta^\square_{P_1,Z_\delta}(\eta,\bnu))=
   \frac{\vol([\A_\K^1])
   (2\pi)^\Sigma\vol(\K_\infty^1\times U(\fc))^3}{\Omega_p^{(k+2)\Sigma}}
    \sum_{\alpha\in\K^1\backslash\A_{\K,f}^1/U(\fc)}
    \blk f(\eta)(\smat{1&\\&\alpha}).
   \end{equation}
\end{lem}
\begin{proof}
Since $\blk$ is given by evaluating at $P_1$, which is invariant by 
$\rho_{\bwt{k}}(\smat{1&\\&\alpha}_{\Sigma_p},\id_2)$,
apply $\blk$ to both sides of Proposition \ref{prop:alg_padic} gives
\[
    \left(\frac{2\pi i}{\Omega_\infty}\right)^{(k+2)\Sigma}
    \blk\ff^\circ(\eta)(\smat{1&\\&\alpha})=
    \frac{1}{\Omega_p^{(k+2)\Sigma}}\cdot
    \blk f(\eta)(\smat{1&\\&\alpha}).
\]
But we have seen in the proof of Lemma \ref{lem:compare_Rallis} that
$\theta^{\square}_{P_1,Z_\delta}(\eta,\bnu)(g_f)=
(2\pi)^\Sigma\vol(\K_\infty^1\times U(\fc))^2
\blk(\ff^\circ(\eta))(g_f)$,
which we plug into the right hand side of \eqref{eq:period_sum}
and conclude the lemma.
\end{proof}

Now apply Corollary \ref{cor:period} to the square
of the left hand side of \eqref{eq:compare_period} gives
\begin{multline*}
   \vol([\A_\K^1])^2(2\pi)^{2\Sigma}\vol(\K_\infty^1\times U(\fc))^6
   \frac{1}{\Omega_p^{(2k+4)\Sigma}}
   \left(\sum_{\alpha\in\K^1\backslash\A_{\K,f}^1/U(\fc)}
    \blk f(\eta)(\smat{1&\\&\alpha})\right)^2\\=
    \left(\frac{2\pi i}{\Omega_\infty}\right)^{(2k+4)\Sigma}
    4(2\pi)^{4\Sigma}C(\K^1)^4C_2
    \frac{\vol([\A_\K^1])^2L(\frac{1}{2},\chi)^2L(\frac{1}{2},\chi\tilde{\mu})}{\prod_{v\mid\fc_-}(1+q_v^{-1})^4}\\
    \cdot
    \textnormal{Im}(-\delta)^{k\Sigma}\tilde{\eta}(z_\delta'^{-1})
    \frac{\Gamma((k+1)\Sigma)}{(2\pi)^{k\Sigma}}
    L(\frac{1}{2},\chi(\tilde{\eta}\tilde{\mu})^{-1})\prod_{v\mid p\fs}
    \frac {\varepsilon(\frac{1}{2}, (\chi(\tilde{\eta}\tilde{\mu})^{-1})_w,\psi_w)}
    {L(1/2, (\chi(\tilde{\eta}\tilde{\mu})^{-1})_w)^2}.
\end{multline*}
To simplify the notation, we collect some of the constant
and define 
\[
    C(\K,\chi)'=
    \frac{(2\pi)^{6\Sigma}}{\vol(\K_\infty^1\times U(\fc))^6}
    \frac{4C(\K^1)^4C_2}{\prod_{v\mid\fc_-}(1+q_v^{-1})^4}
    \left[\frac{L(\frac{1}{2},\chi)}{\Omega_\infty^\Sigma}\right]^2
\]
which is also an algebraic number.
After some cancellation, we can rewrite the equation above as
\begin{multline}\label{eq:padic_sum}
   \frac{1}{\Omega_p^{(2k+4)\Sigma}}
   \left( 
   \sum_{\alpha\in\K^1\backslash\A_{\K,f}^1/U(\fc)}
    \blk f(\eta)(\smat{1&\\&\alpha})\right)^2=
    C(\K,\chi)'
    \frac{L(\frac{1}{2},\chi\tilde{\mu})}{\Omega_\infty^\Sigma}\\\cdot
    \left(\frac{2\pi}{\Omega_\infty}\right)^{(2k+1)\Sigma}
    \textnormal{Im}(-\delta)^{k\Sigma}\tilde{\eta}(z_\delta'^{-1})
    \frac{\Gamma((k+1)\Sigma)}{(2\pi)^{(k+1)\Sigma}}
    L(\frac{1}{2},\chi(\tilde{\eta}\tilde{\mu})^{-1})\prod_{v\mid p\fs}
    \frac {\varepsilon(\frac{1}{2}, (\chi(\tilde{\eta}\tilde{\mu})^{-1})_w,\psi_w)}
    {L(1/2, (\chi(\tilde{\eta}\tilde{\mu})^{-1})_w)^2}
\end{multline}
\begin{rem}
Since the Hida family $\euF^\circ$ interpolates $\blk(f(\eta))$,
the argument as Theorem \ref{thm:function} shows that 
\[
    \sum_{\alpha\in\K^1\backslash\A_{\K,f}^1/U(\fc)}
    \euF^\circ\left(\smat{1&\\&\alpha}\right)
\]
should define a square root anticyclotomic p-adic L-function.
\end{rem}

\begin{thm}\label{thm:primitive}
    The Hida family $\euF^\circ_{\fs}$ is primitive in 
    if aside from \ref{cond:chi1}-\ref{cond:chi3} we further assume
\begin{enumerate}[label={($\chi$\arabic*)}]
    \setcounter{enumi}{3}
    \item $C(\K,\chi)'$ is a $p$-unit.
    \label{cond:chi4}
    \item $\chi_\circ\coloneqq\chi|\cdot|_{\K}^{1/2}$ 
    satisfies the conditions in \cite[Theorem A]{Hsieh12}.
    \label{cond:chi5}
\end{enumerate}
\end{thm}
\begin{proof}
    Decompose $\fG_\fs$ into $\Delta_\fs\times W$,
    where $\Delta_\fs$ is the finite part of $\fG_\fs$
    and $W$ is free of finite rank over $\Zp$ 
    and is independent of $\fs$.
    It suffices to show that for each character 
    $\eta$ of $\Delta_{\fs}$, viewed as an $\fs$-admissible
    finite order Hecke character via the reciprocity map,
    the algebraic modular form $\pf(\eta)$
    is non-vanishing modulo $p$.

    Let $\ell$ be a prime number, 
    $\fl$ be a degree-one prime of $\F$ over $\ell$
    that is split in $\K$ and prime to $p\dd_{\K/\F}\fc\fs$.
    Note that the conditions in \cite[Theorem A]{Hsieh12}
    for $\chi_\circ$ are satisfied if and only if the same set of
    conditions for $\chi_\circ\tilde{\eta}^{-1}$ are satisfies.
    Therefore, assuming \ref{cond:chi5}, we can apply 
    the main result in \cite{Hsieh12} to
    $\chi_\circ\tilde{\eta}^{-1}$, 
    so there exists a finite order character $\mu$
    of conductor dividing powers of $\fl$ such that 
    \begin{equation*}
    \frac{L^{(\fl)}(\frac{1}{2},\chi\tilde{\mu})}{\Omega_\infty^{\Sigma}}
   \frac{L^{(\fl)}(\frac{1}{2},\chi(\tilde{\eta}\tilde{\mu})^{-1})}{\Omega_\infty^{\Sigma}}
    \not\equiv 0 \pmod \fm.
    \end{equation*}
    and $1-(\chi_\circ(\tilde{\eta}\tilde{\mu})^{-1})_w(\varpi_w)
    \not\equiv 0\mod \fm$ for all $w\in\Sigma_p$ or $w\mid\mathfrak{f}$.

    Since the modification at $v=w\bw$ with $w=\fl$ 
    in Proposition \ref{prop:period_local}
    can be written as a $\overline{\Z}_p$-linear combination
    of right-translations of $\phi_v$, we see that some 
    $\overline{\Z}_p$-linear combination
    of right-translations of $\blk(f(\eta))$ is equal 
    to the right hand side of \eqref{eq:padic_sum},
    which is a p-unit by \ref{cond:chi4} and the choice of $\mu$.
    This shows that $f(\eta)$ is indeed non-vanishing modulo $p$.
\end{proof}

\bibliographystyle{amsalpha}
\bibliography{biblio}
\end{document}